\documentclass[10pt]{article}

\RequirePackage{amsthm,amsmath,amsfonts,amssymb}
\RequirePackage{graphicx}

\usepackage[left=2cm,right=2.5cm]{geometry}
\usepackage{authblk} 
\usepackage{epsfig}
\usepackage{color}
\usepackage{dsfont}
\usepackage{bm}
\usepackage{verbatim}
\usepackage{stmaryrd}
\usepackage{bbm}			
\usepackage{enumerate} 
\usepackage{nicefrac}
\usepackage{xfrac}
\usepackage{units} 
\usepackage{mathtools}
\usepackage{pdfcomment}
\usepackage{url}
\usepackage{soul}
\usepackage[labelfont={bf,sf},font={footnotesize},labelsep=space,format=plain]{caption}
\usepackage{acronym}
\usepackage{mathrsfs}
\usepackage[T1]{fontenc}
\usepackage{caption, booktabs}
\usepackage{textcomp}
\usepackage{multirow}
\usepackage{lipsum}
\usepackage{makecell}
\usepackage[section]{placeins}
\usepackage{hologo}
\usepackage{xcolor}
\usepackage{dutchcal}
\usepackage[title]{appendix}%
\usepackage[authoryear,round]{natbib}
\usepackage{lipsum} 
\usepackage{authblk}
\usepackage{float}

 \theoremstyle{plain}
 \newtheorem{theorem}{Theorem}[section]
 \newtheorem{Lemma}[theorem]{Lemma}
 \newtheorem{Cor}[theorem]{Corollary}

 \theoremstyle{definition}

 \newtheorem{Rem}[theorem]{Remark}
 \newtheorem{?}[theorem]{Problem}
 \newtheorem{Ex}[theorem]{Example}

 \newcommand{\ff}{g}
\newcommand{\fg}{f}

\providecommand{\keywords}[1]
{
  \small	
  \textit{Keywords}: \hangindent=7em \hangafter=1 #1
}

\providecommand{\msc}[1]
{
  \small	
  \textit{2020 MSC}: #1
}

\makeatother

\title{\LARGE{On bivariate Archimax copulas: Level sets, mass distributions and related results}}

\date{}

\author{
        \centering
        Nicolas Pascal Dietrich\thanks{Department for Artificial Intelligence and Human Interfaces, University of Salzburg, Hellbrunner Strasse 34, Salzburg, 5020, Salzburg, Austria, nicolaspascal.dietrich@plus.ac.at}
}

\begin{document}

\maketitle

\begin{abstract}
\noindent Motivated by the results in \citep{Mai2011,evc-mass}, which examine the way bivariate Extreme Value copulas distribute their mass, we extend these findings to the larger family of bivariate Archimax copulas $\mathcal{C}_{am}$. Working with Markov kernels (conditional distributions), we analyze the mass distributions of Archimax copulas $C \in \mathcal{C}_{am}$ and show that the support of $C$ is determined by some functions $f^0$, $g^L$ and $g^R$. Additionally, we prove that the discrete component (if any) of $C$ concentrates its mass on the graphs of certain convex functions $f^s$ or non-decreasing functions $g^t$.\\
Investigating the level sets $L^t$ of Archimax copulas $C \in \mathcal{C}_{am}$, we establish that these sets can also be characterized in terms of the afore-mentioned functions $f^s$ and $g^t$. Furthermore, recognizing the close relationship between the level sets $L^t$ of a copula $C$ and its Kendall distribution function $F_C^K$, we provide an alternative proof for the representation of $F_C^K$ for arbitrary Archimax copulas $C\in \mathcal{C}_{am}$ and derive simple expressions for the level set masses $\mu_C(L^t)$.\\
Building upon the fact that Archimax copulas $C \in \mathcal{C}_{am}$ can be represented via two univariate probability measures $\gamma$ and $\vartheta$ — so-called Williamson and Pickands dependence measures — we show that absolute continuity, discreteness and singularity properties of these measures $\gamma$ and $\vartheta$ carry over to the corresponding Archimax copula $C_{\gamma, \vartheta}$. Finally, we derive conditions on $\gamma$ and $\vartheta$ such that the support of the absolutely continuous, discrete or singular component of $C_{\gamma, \vartheta}$ coincides with the support of $C_{\gamma, \vartheta}$.
\end{abstract}


\msc{62G32, 62H05, 26A30, 60A10, 60E05}

\section{Introduction\label{sec:1}}
Examining maxima of independent, identically distributed sequences of random variables, Extreme Value copulas (EVCs) arise naturally and are a well-suited tool for modeling dependence between extreme events. Including various parametric sub-families, such as Gumbel–Hougaard and Galambos copulas \citep{gudendorf2011}, EVCs are applied in numerous fields, notably  finance \citep{longin,mcneil2005} and hydrology \citep{salvadori}. Moreover, EVCs are also theoretically appealing, since they are characterized by a single univariate function $A$.
To be more precise, according to \citep{dur_princ,haan1977,nelsen2006,Pickands} a copula $B$ is an EVC if, and only if there exists a Pickands dependence function $A$, i.e., a convex function $A \colon [0,1] \rightarrow [\frac{1}{2},1]$ fulfilling $\max\{1-x,x\} \leq A(x) \leq 1$ for every $x \in [0,1]$ such that the identity
\begin{equation*}
    B(x,y) = \exp\left((\log(x) + \log(y))A\left(\frac{\log(x)}{\log(x) + \log(y)}\right)\right)
\end{equation*}
holds for every $x,y \in (0,1)$. The family of EVCs has been the subject of extensive study, yielding key results on convergence, mass distributions, and other analytic and measure-theoretic properties \citep{bernoulli, Mai2011, evc-mass}.\\
Considering so-called Archimedean generators $\psi$ (defined in Section \ref{subsec:intro_archimedean}) now allows to naturally extend the family of EVCs to the much larger family of Archimax copulas. Following \citep{caperaa2000}, a copula $C$ is called an Archimax copula if, and only if there exists a convex Archimedean generator $\psi \colon [0,\infty) \rightarrow [0,1]$ with pseudo-inverse $\varphi$ and a Pickands dependence function $A \colon [0,1] \rightarrow [\frac{1}{2},1]$ such that
\begin{equation*}
C(x,y) = \psi\left((\varphi(x) + \varphi(y))A\left(\frac{\varphi(x)}{\varphi(x) + \varphi(y)}\right)\right)
\end{equation*}
holds for every $x,y \in (0,1)$. It was proved in \citep{caperaa2000} and later in \citep{dur_archimax} that a function $C$, defined in the afore-mentioned way indeed constitutes a bona fide copula.
As the name suggests, the family of Archimax copulas $\mathcal{C}_{am}$ encompasses both the family of EVCs $\mathcal{C}_{ev}$ and the family of Archimedean copulas $\mathcal{C}_{ar}$ as sub-families and can therefore be seen as an extension of these two families of copulas. For more background on Archimedean copulas and their applications, we refer to \citep{dur_princ, neslehova,nelsen2006,mult_arch}.
Viewing Archimax copulas from an application perspective, they offer the flexibility to model asymptotic dependence in extreme events while simultaneously capturing joint risks at moderate levels \citep{est-archimax}. For further research in the field of bivariate Archimax copulas we refer to \citep{bacigal2011,caperaa2000,genest2020}. Moreover, an extension to the multivariate setting was considered in \citep{mult-archimax,Mesiar2013}. Regarding inference, an estimator for a multivariate, semi-parametric framework was provided in \citep{est-archimax}.\\
While it has already been established how Archimedean copulas and EVCs distribute their mass \citep{neslehova,mult_arch,Mai2011,evc-mass}, and the level sets of Archimedean copulas have been investigated in \citep{neslehova,mult_arch}, comparable results for the larger family of Archimax copulas $\mathcal{C}_{am}$ remain, to the best of the author's knowledge, unstudied. The current paper addresses this gap by analyzing the mass distributions and level sets of bivariate Archimax copulas, providing several new insights and related findings.\\
Motivated by the fact that the level sets of Archimedean copulas can be represented by convex functions \citep{nelsen2006}, we prove that the level sets of Archimax copulas can similarly be characterized by two functions $f^s$ and $g^t$ (specified in Section \ref{sec:lvl_sets_archimax}) for some $s,t \in [0,1)$. Moreover, we derive conditions on the Pickands dependence function $A$, such that the level sets $L^t$ of the corresponding $C \in \mathcal{C}_{am}$ coincide with the graphs of functions $f^t$ and show that these functions are convex.\\
Building upon the fact that both Archimedean copulas and EVCs can be characterized by univariate probability measures, so-called Williamson measures $\gamma$ and Pickands dependence measures $\vartheta$, respectively \citep{dietrich2024,mult_arch,neslehova,Pickands}, we demonstrate that Archimax copulas can similarly be defined via these measures. Taking this into consideration and providing an explicit representation of (a version of) the Markov kernel $K_C$ of $C \in \mathcal{C}_{am}$, we study the support of $C$ and show how $C$ distributes its mass. To be more precise, we prove that if either of the measures $\gamma$ or $\vartheta$ has full support, the support of $C_{\gamma,\vartheta} \in \mathcal{C}_{am}$ is fully characterized by the afore-mentioned functions $f^s$ and $g^t$ for some $s,t \in [0,1)$.
Furthermore, we derive conditions on $\gamma$ and $\vartheta$ such that $C_{\gamma,\vartheta} \in \mathcal{C}_{am}$ has full support and establish that the discrete component (as specified in Section \ref{subsection:general_notation} and if non-degenerated) of $C_{\gamma,\vartheta}$ concentrates its mass on the graphs of functions $f^s$ or $g^t$ for some $s,t \in [0,1)$. As direct consequence of this result, we establish that $\gamma$ or $\vartheta$ has a point mass if, and only if the corresponding $C_{\gamma,\vartheta} \in \mathcal{C}_{am}$ has non-degenerated discrete component.\\
Given that the level sets $L^t$ of a copula $C$ are closely related to its Kendall distribution function $F_C^K$, working with Markov kernels again, we provide an alternative proof for the representation of $F_C^K$ in the case that $C \in \mathcal{C}_{am}$ \citep{caperaa2000}. Building upon this representation then allows for a simple calculation of the level set masses $\mu_C(L^t)$ of $C\in \mathcal{C}_{am}$.\\
Triggered by the fact that for both Archimedean copulas and Extreme Value copulas $C$, absolute continuity, discreteness and singularity of $C$ (in the sense specified in Section \ref{subsection:general_notation}) is closely related to absolute continuity, discreteness and singularity of the corresponding Williamson measure $\gamma$ or Pickands dependence measure $\vartheta$ (singularity in
the sense that the measure $\gamma$ (or $\vartheta$) has no point masses and its corresponding distribution function $F_\gamma$ (or $F_\vartheta$) has derivative $0$ almost everywhere) \citep{dietrich_arch2024, dietrich2024, mult_arch}, we prove that in the setting of Archimax copulas absolute continuity, discreteness and singularity of $\gamma$ and $\vartheta$ carries over to the corresponding $C_{\gamma,\vartheta} \in \mathcal{C}_{am}$. Finally, we show that if the Williamson measure $\gamma$ is absolutely continuous, has full support, and some minor regularity assumption on the Pickands dependence measure $\vartheta$ (specified in Section \ref{subsec:intro_evc}) is satisfied, then the support of the absolutely continuous component of the corresponding $C_{\gamma,\vartheta} \in \mathcal{C}_{am}$ coincides with the support of $C_{\gamma,\vartheta}$. Analogous results are then established for the discrete and the singular components of $C_{\gamma,\vartheta}$. The afore-mentioned finding generalizes the result in \citep[Corollary 5]{evc-mass}, saying that when $C$ is an EVC and is not the Fréchet-Hoeffding upper bound $M$, the support of $C$ coincides with the support of the absolutely continuous component of $C$.\\
\\
The reminder of this paper is organized as follows: Section \ref{section:notation} contains notations and definitions used throughout this work and is divided into two parts. The first subsection presents general notation and definitions, while
the second revisits key properties of Archimedean and Extreme Value copulas, emphasizing their connections to Williamson and Pickands dependence measures. Additionally, Archimax copulas are introduced and it is established that they can be represented by Williamson and Pickands dependence measures. Section \ref{sec:lvl_sets_archimax} studies the level sets of Archimax copulas. After showing that the level sets can be characterized using functions $f^s$ and $g^t$ for some $s,t \in [0,1)$, we prove that the functions $f^s$ are convex and that the functions $g^t$ are non-decreasing. In Section \ref{sec:markov_mass_dist_discr} a version of the Markov kernel $K_C$ for a given Archimax copula $C$ is provided and the support of $C$, and its relationship to the corresponding Williamson measure $\gamma$ and Pickands dependence measure $\vartheta$ is studied. 
Showing that the discrete component of $C$ (if any) always concentrates its mass on graphs of functions $f^s$ or $g^t$ for some $s,t \in [0,1)$, it is proved that the discrete component of $C$ is non-degenerated if, and only if $\gamma$ or $\vartheta$ has a point mass. In Section \ref{sec:kendall_copula_mass} an explicit representation for the Kendall distribution function $F_C^K$ of an Archimax copula $C$ is provided and this representation is used to calculate the level set mass $\mu_C(L^t)$ of $C$. What is more, the mass $\mu_C(\Gamma(g^t))$ of the graph $\Gamma(g^t)$ is calculated for every $t \in (0,1)$. Finally, Section \ref{section:regularity_results} studies absolute continuity/discreteness/singularity of Archimax copulas $C$ and the support of the absolutely continuous/discrete/singular component of $C$ is investigated.\\
To enhance readability, certain technical lemmas and proofs have been deferred to the Appendix. Additionally, several examples and graphics are provided to clarify the selected procedures and convey some of the underlying concepts.


\section{Notation and preliminaries}\label{section:notation}
\subsection{General notation}\label{subsection:general_notation}
Throughout this contribution we will denote by $\mathcal{C}$ the family of all bivariate copulas and $\mathbb{I} := [0,1]$ will denote the unit interval. For a given random vector $(X,Y)$ on a probability space $(\Omega, \mathcal{F}, \mathbb{P})$ and some copula $C \in \mathcal{C}$, we will write $(X,Y) \sim C$ if $C$ is the joint distribution function of $(X,Y)$ restricted to $\mathbb{I}^2$. The doubly stochastic measure corresponding to the copula $C\in\mathcal{C}$ will be denoted by $\mu_C$, i.e., $\mu_C([0,x] \times [0,y]) := C(x,y)$ 
for every $(x,y) \in \mathbb{I}^2$.
Considering an arbitrary topological space $(S,\tau)$, the Borel $\sigma$-field on $S$ will be denoted by $\mathcal{B}(S)$, and $\mathcal{P}(S)$ denotes the class of all probability measures on $\mathcal{B}(S)$. Furthermore, for a given measure $\nu$ on $\mathcal{B}(S)$, the support of $\nu$ is the complement of the union of all open sets $U$, which fulfill $\nu(U) = 0$ and will be denoted by $\mathrm{supp}(\nu)$.
In this article, the support of a copula $C \in \mathcal{C}$ is defined as the support of its corresponding doubly stochastic measure $\mu_C$. For further information on copulas and doubly stochastic measures, see \citep{dur_princ,nelsen2006}.
We denote the $2$-dimensional Lebesgue measure on $\mathcal{B}(\mathbb{I}^2)$ by $\lambda_2$ and write $\lambda$ for the univariate Lebesgue measure. Moreover, for an arbitrary $x \in S$, the Dirac measure in $x \in S$ will be denoted by $\delta_x$. 
Let $(S,d)$ and $(S',d')$ be two metric spaces, \mbox{$T:S \rightarrow S'$} be a Borel-measurable transformation and  $\nu \in \mathcal{P}(S)$ be a probability measure. Then the push-forward measure $\nu$ under $T$, denoted by $\nu^T$, is defined by
$\nu^T(F):=\nu(T^{-1}(F))$ for every $F \in \mathcal{B}(S')$.\\
In the sequel, conditional distributions and Markov kernels will be key concepts.
A map $K: \mathbb{R}\times\mathcal{B}(\mathbb{R}) \rightarrow \mathbb{I}$ is called a \emph{Markov kernel} from $\mathbb{R}$ to $\mathbb{R}$ if, and only if the function $x\mapsto K(x,E)$ is Borel-measurable for every $E\in\mathcal{B}(\mathbb{R})$ and the map $E\mapsto K(x,E)$ is a probability measure on $\mathcal{B}(\mathbb{R})$ for every $x\in\mathbb{R}$. 
We call $K$ a sub-Markov kernel (or simply sub-kernel) if, for fixed $x \in \mathbb{R}$, the measure $E\mapsto K(x,E)$ only satisfies $K(x,\mathbb{I})\leq 1$ (instead of $K(x,\mathbb{I})=1$). \\
Given a random vector $(X,Y)$ on a probability space $(\Omega,\mathcal{F},\mathbb{P})$, we will call a Markov kernel $K(\cdot,\cdot)$ a \textit{regular conditional distribution} of
$Y$ given $X$ if for every set $E \in \mathcal{B}(\mathbb{R})$ the equation
$$
K(X(\omega), E) = \mathbb{E}(\mathbf{1}_E \circ Y | X)(\omega)
$$
holds for $\mathbb{P}$-almost every $\omega \in \Omega$.
It is a well known result that for every random vector $(X,Y)$, there exists a regular conditional 
distribution $K(\cdot, \cdot)$ of $Y$ given $X$, which is unique for $\mathbb{P}^{X}$-almost every $x\in \mathbb{R}$.
In the case that $C \in \mathcal{C}$ and  $(X,Y) \sim C$, we will denote by $K_{C}:\mathbb{I} \times \mathcal{B}(\mathbb{I}) \to \mathbb{I}$ (a version of) the regular conditional 
distribution of $Y$ given $X$ and call it the \emph{Markov kernel} of $C$.
Considering $x \in \mathbb{I}$, we define the $x$-section $G_x$ of a given set $G \in\mathcal{B}(\mathbb{I}^2)$ by $G_{x}:=\{y
\in \mathbb{I}: (x,y) \in G\}\in\mathcal{B}(\mathbb{I})$. Using \emph{disintegration} (see \citep[Section 5]{Kallenberg} and \citep[Section 8]{Klenke}), the subsequent identity holds for all $G \in \mathcal{B}(\mathbb{I}^2)$:
\begin{align}\label{eq:DI}
	\mu_C(G) = \int_{\mathbb{I}} K_{C}(x,G_x)
	\, \mathrm{d}\lambda(x).
\end{align}
Considering $d \in \{1,2\}$, in this contribution, a measure $\nu$ on $\mathcal{B}(S)$, whereby $S=\mathbb{I}^d$ or $S=[0,\infty)$, is referred to as singular (w.r.t. $\lambda_d$ or $\lambda$)
if, and only if it has no point-masses and there exists a set $G \in \mathcal{B}(S)$ such that 
$\lambda_d(G) = 0$  and $\nu(G) = \nu(S)$. For the sake of simplicity, we will use the term 'singular' for the afore-mentioned measures, rather than the longer expression 'singular without point-masses'. This aligns with the concept of singular distribution functions and is justified by the following straightforward observation: 
The doubly stochastic measure $\mu_C$ corresponding to the copula $C \in \mathcal{C}$ has no point masses and therefore always has degenerated discrete component. Following the same approach as in \citep{evc-mass,mult_arch} and decomposing the Markov kernel $K_C$ corresponding to $C$,
we obtain a decomposition of the measure $\mu_C$ into three components, which may not be degenerated.
To be more precise, we denote by $K_C^{abs}, K_C^{dis}, K_C^{sing}: \mathbb{I} \times \mathcal{B}(\mathbb{I})
\rightarrow \mathbb{I}$ the absolutely continuous, discrete and singular sub-kernels of $K_C$, respectively, and obtain, according to \citep{Lange}, that
\begin{equation}\label{eq:lebesgue_decomp_markov}
K_C(x,F) = K_C^{abs}(x,F) + K_C^{dis}(x,F) + K_C^{sing}(x,F)
\end{equation}
holds for every $x \in \mathbb{I}$ and every $F \in \mathcal{B}(\mathbb{I})$. Denoting by $\mathcal{k}_C$ the Radon-Nikodym derivative of $\mu_C$ w.r.t. $\lambda_2$ (almost everywhere), uniqueness of $K_C$ yields that the measures $K_C^{abs}(x,\cdot)$ and $E \mapsto \int_E \mathcal{k}_C(x,y) \mathrm{d}\lambda(y)$ coincide for $\lambda$-almost every $x \in \mathbb{I}$ \citep{evc-mass}.
Applying  eq. \eqref{eq:lebesgue_decomp_markov} and using disintegration, we define the absolutely continuous, the discrete and the singular components $\mu_C^{abs}$, $\mu_C^{dis}$ and $\mu_C^{sing}$ of $\mu_C$ by
\begin{align}\label{eq:def_abs_dis_sing_copula}
\nonumber\mu_C^{abs}(E\times F) &:= \int_E K_C^{abs}(x,F) \mathrm{d}\lambda(x),\\
\nonumber\mu_C^{dis}(E\times F) &:= \int_E K_C^{dis}(x,F) \mathrm{d}\lambda(x),\\
\mu_C^{sing}(E\times F) &:= \int_E K_C^{sing}(x,F) \mathrm{d}\lambda(x)
\end{align}
for every $E \in \mathcal{B}(\mathbb{I})$ and every $F \in  \mathcal{B}(\mathbb{I})$. We extend  $\mu_C^{abs}$, $\mu_C^{dis}$ and $\mu_C^{sing}$ to full
$\mathcal{B}(\mathbb{I}^2)$ in the standard way.
In this contribution a copula $C$ will be called absolutely continuous, discrete or singular if, and only if $\mu_C^{abs}(\mathbb{I}^2)=1$, $\mu_C^{dis}(\mathbb{I}^2)=1$ or $\mu_C^{sing}(\mathbb{I}^2)=1$ holds. Furthermore, we will denote $\mu_C^{abs}$, $\mu_C^{dis}$, $\mu_C^{sing}$ as the absolutely continuous, the discrete, and the singular components of $C$ (or $\mu_C$), respectively.\\
It is important to note that, when considering a copula $C$, the singular component $\mu_C^{sing}$ of $\mu_C$ does not correspond to the usual definition of the singular component of $\mu_C$ (in the sense that there exists some $N \in \mathcal{B}(\mathbb{I}^2)$ such that $\mu_C^{sing}(N) = \mu_C^{sing}(\mathbb{I}^2)$ and $\lambda_2(N) = 0$). For instance, when $C = M$ ($M$ being the Fréchet-Hoeffding upper bound), its corresponding Markov kernel $K_C$ is the Dirac measure and is therefore discrete, implying that $\mu_C^{dis}(\mathbb{I}^2) = 1$. In contrast, according to the standard definition, $C$ would be considered singular.\\
Given a copula $C \in \mathcal{C}$, and a random vector $(X,Y) \sim C$, then according to \citep[Definition 3.9.5]{dur_princ} the Kendall distribution function of $C$ is defined by
\begin{equation}\label{eq:kendall_dist}
F_C^K(t) := \mathbb{P}(C(X,Y) \leq t)
\end{equation}
for every $t \in \mathbb{I}$. Moreover, following
\citep{genest1993}, Kendall's tau of $C$ is defined by
\begin{equation}\label{eq:kendalls_tau_general}
    \tau(C) := 4\int_{\mathbb{I}^2}C(s,t) \mathrm{d}\mu_C(s,t) - 1.
\end{equation}
Furthermore, we will refer to a function $G\colon \mathbb{I} \rightarrow \mathbb{I}$ as distribution function/measure generating function, if its extension, in the sense that $G(x)=0$ for every $x<0$ and $G(x)=1$ for every $x>1$, is a distribution function/measure generating function. Similarly, we call a function $H\colon [0,\infty) \rightarrow \mathbb{I}$ a distribution function, if its extension via setting $H(x)=0$ for every $x<0$ is a distribution function.
Given an arbitrary interval $I \subseteq \mathbb{R}$ and a function $u \colon I \rightarrow \mathbb{R}$ , we will call $u$ singular if, and only if there exists a set $N\in \mathcal{B}(I)$ with $\lambda(N) = 0$ such that for all $x \in I \setminus N$ the derivative $u'(x)$ exists, $u'(x) = 0$ holds and
$u$ is non-constant on $I$. Moreover, (assuming their existence) the right-hand/left-hand derivative of $u$ will be denoted by $D^+u$ and $D^-u$, respectively. Considering some set $E \subseteq \mathbb{R}^2$ and a function $v \colon E \rightarrow \mathbb{R}$, emphasizing that the right-hand/left-hand derivative is taken w.r.t. a specific variable (and assuming their existence in $(x,y)\in E$), we will sometimes write $\partial_x^+v(x,y)$ or $\partial_x^-v(x,y)$ for the right-hand/left-hand partial derivative of $v$ in $(x,y)\in E$. Given a set $F \subseteq \mathbb{R}$, its complement will be denoted by $F^c$  and the graph of a function $u \colon F \rightarrow \mathbb{R}$ is denoted by $\Gamma(u)$.
Finally, we note that throughout this contribution we consistently use the convention that $a\cdot\infty = \infty$, $\frac{a}{\infty} = 0$ and $\frac{a}{0} = \infty$ for every $a \in \mathbb{R} \setminus \{0\}$.
\subsection{Archimedean, Extreme Value and Archimax copulas}
\subsubsection{Archimedean copulas}\label{subsec:intro_archimedean}
In the sequel, a non-increasing and continuous function $\psi \colon [0,\infty) \rightarrow \mathbb{I}$ is called an Archimedean generator if, and only if it fulfills $\psi(0) = 1$, $\lim_{z \rightarrow \infty}\psi(z) = 0 =: \psi(\infty)$ and it is
strictly decreasing on $[0,\inf\{x\colon \psi(x) = 0\})$. Throughout this contribution, we will refer to Archimedean generators simply as generators. Moreover, we define the pseudo-inverse $\varphi \colon \mathbb{I} \rightarrow [0,\infty]$ of a generator $\psi$ by $\varphi(y) := \inf\{z \in [0,\infty] \colon \psi(z) = y\}$ for every $y \in \mathbb{I}$. It can be easily proved that $\varphi$ is strictly decreasing on $(0,1]$, 
right-continuous at $0$ and satisfies the condition $\varphi(1) = 0$ \citep{mult_arch}. 
A generator $\psi$, or its pseudo-inverse $\varphi$, is referred to as strict if $\varphi(0) = \infty$ and otherwise as non-strict. Note, that strictness of $\psi$ is equivalent to the fact that $\psi(z) > 0$ for every $z \in [0,\infty)$. 
A copula $C \in \mathcal{C}$ is called Archimedean if, and only if there exists a generator $\psi$ with pseudo-inverse $\varphi$ such that
\begin{equation}\label{eq:def_arch_cop}
C(x,y) := \psi\left(\varphi(x) + \varphi(y)\right)
\end{equation}
holds for every $(x,y) \in \mathbb{I}^2$.
Emphasizing the correspondence between generator $\psi$ and copula 
$C$,  we occasionally use $C_\psi$ instead of $C$ in the following. The class of all bivariate Archimedean copulas will be denoted by $\mathcal{C}_{ar}$.
Furthermore, given that $\psi$ is a generator, applying \citep[Proposition 2.1]{neslehova}, the function $C_\psi$ defined according to eq. \eqref{eq:def_arch_cop} is a bivariate Archimedean copula if, and only if $\psi$ is convex.
Convexity of $\psi$ ensures the existence of both $D^-\psi$ and $D^+\psi$ on $(0,\infty)$. According to \citep[Theorem 3.7.4]{Kannan1996} and \citep[Appendix C]{Pollard2001}, the one-sided derivatives coincide for all but at most countably many points in $(0,\infty)$ and $D^-\psi(z) = D^+\psi(z)$ holds for every point of continuity $z$ of $D^-\psi$. 
Furthermore, as shown in \citep{mult_arch}, every convex generator $\psi$ fulfills $D^-\psi(\infty):=\lim_{z \rightarrow \infty}D^-\psi(z) = 0$ and $D^-\psi(z) = 0$ holds for every $z > \varphi(0)$. Finally, note that if $\psi$ is a convex generator also its pseudo-inverse $\varphi$ is convex.
Throughout this article the family of all convex generators will be denoted by $\mathbf{\Psi}$.\\
Following \citep{mult_arch,neslehova,schilling2012}, Archimedean copulas can also be obtained using univariate probability measures $\gamma \in \mathcal{P}([0,\infty))$ fulfilling $\gamma(\{0\}) = 0$, so-called Williamson measures. We define the family of all such measures by
$$
\mathcal{P}_\mathcal{W} := \{\gamma \in \mathcal{P}([0,\infty)) \colon \gamma(\{0\}) = 0\}.
$$
Considering $\gamma\in \mathcal{P}_\mathcal{W}$, we use the convention that $\gamma(\{\infty\}) = 0$. 
Throughout this article we denote the family of all absolutely continuous/discrete and singular Williamson measures by $\mathcal{P}_\mathcal{W}^{abs}$, $\mathcal{P}_\mathcal{W}^{dis}$ and $\mathcal{P}_\mathcal{W}^{sing}$, respectively.
According to \citep{mult_arch,neslehova,schilling2012} there is a one-to-one correspondence between the family of convex generators $\mathbf{\Psi}$ and the family of Williamson measures $\mathcal{P}_\mathcal{W}$. To be more precise, for every $\psi \in \mathbf{\Psi}$, there exists a unique $\gamma \in \mathcal{P}_\mathcal{W}$ such that
\begin{equation}\label{eq:rel_psi_gamma}
    \psi(z) = \int_{[0,\infty)}(1- tz)_+ \mathrm{d}\gamma(t)
\end{equation}
holds for every $z>0$ \citep[Theorem 5.1]{mult_arch}. Moreover, if $\gamma \in \mathcal{P}_\mathcal{W}$ and $\psi \in \mathbf{\Psi}$ is its corresponding generator, then, according to \citep{mult_arch}, the following representation of the left-hand derivative $D^-\psi$ of $\psi$ in terms of $\gamma$ holds for every $z > 0$:
\begin{equation}\label{eq:rel_deriv_psi_gamma}
    D^-\psi(z) = - \int_{[0,\frac{1}{z}]}t \mathrm{d}\gamma(t).
\end{equation}
Finally, we refer to $\gamma$ as strict if its support contains the point $0$ and otherwise as non-strict. It has been established in \citep[Lemma 5.5]{mult_arch} that $\gamma$ is strict if, and only if its corresponding generator $\psi$ is strict. If $\gamma$ is non-strict and $\varphi$ is the pseudo-inverse of the corresponding generator $\psi$, then $\gamma([0,z]) = 0$ holds for every $z < \frac{1}{\varphi(0)}$ \citep{dietrich_arch2024}.

\subsubsection{Extreme Value Copulas}\label{subsec:intro_evc}
Given a copula $C \in \mathcal{C}$, we call $C$ an Extreme Value copula (EVC) if there exists a $B \in \mathcal{C}$ such that
$$
C(x,y)  = \lim_{n \rightarrow \infty} B^n(x^\frac{1}{n},y^\frac{1}{n})
$$
holds for every $x,y \in \mathbb{I}$. Throughout this contribution we denote the family of all bivariate EVCs by $\mathcal{C}_{ev}$. Furthermore, as established in \citep{dur_princ,haan1977,nelsen2006,Pickands}, the following assertions are equivalent:
\begin{enumerate}
    \item $C \in \mathcal{C}_{ev}$
    \item $C$ is max-stable, i.e., $C(x,y) = C^k(x^\frac{1}{k},y^\frac{1}{k})$ for all $k \in \mathbb{N}$ and $x,y \in \mathbb{I}$.
    \item There exists a Pickands dependence function $A$, i.e., a convex function $A \colon \mathbb{I} \rightarrow \mathbb{I}$ fulfilling $\max\{1-x,x\} \leq A(x) \leq 1$ for every $x \in \mathbb{I}$, such that
    \begin{equation}\label{eq:map_eq_pick_copula}
    C(x,y) = (xy)^{A\left(\frac{\log(x)}{\log(xy)}\right)}
    \end{equation}
    holds for all $x,y \in (0,1)$.
\end{enumerate}
We will denote by $\mathcal{A}$ the family of all Pickands dependence functions and let $C_A$ denote the unique EVC induced by $A \in \mathcal{A}$.
Convexity implies that for every Pickands dependence function $A \in \mathcal{A}$ the right-hand derivative $D^+A(x)$ exists for every $x \in [0,1)$ and the left-hand derivative $D^-A(x)$ exists for every $x \in (0,1]$. Moreover, $D^+A(x)=D^-A(x)$ holds for all but at most countably many $x \in (0,1)$ and therefore $A$ is differentiable on the complement of a countable set in $(0,1)$. Furthermore, $D^+A$ is non-decreasing and right-continuous on $[0,1)$ and $D^-A$ is non-decreasing and 
left-continuous on $(0,1]$ \citep{Kannan1996,Pollard2001}. 
We define $D^+A(1):=D^-A(1)$ and therefore obtain that $D^+A$ is a non-decreasing and right-continuous function on the interval $\mathbb{I}$. Considering that $\max\{1-x,x\} \leq A(x) \leq 1$ for every $x \in \mathbb{I}$, it follows that $D^+A(x) \in [-1,1]$ holds for every $x \in \mathbb{I}$.\\
Following \citep{beirlant2004, dietrich2024, haan1977, Pickands}, EVCs can also be characterized via probability measures $\vartheta$ on $\mathbb{I}$ with expected value $\mathbb{E}(\vartheta) := \int_\mathbb{I}x \mathrm{d}\vartheta(x) = \frac{1}{2}$, so-called Pickands dependence measures. Throughout this paper we define the family of all such measures by
$$
\mathcal{P}_\mathcal{A} := \left\{\vartheta \in \mathcal{P}(\mathbb{I}) \colon \int_\mathbb{I}x \mathrm{d}\vartheta(x) = \frac{1}{2}\right\}.
$$
We define the family of all absolutely continuous, discrete and singular Pickands dependence measures by $\mathcal{P}_\mathcal{A}^{abs}$, $\mathcal{P}_\mathcal{A}^{dis}$ and $\mathcal{P}_\mathcal{A}^{sing}$, respectively.
It is well known that the Pickands dependence function associated with the Fréchet-Hoeffding upper bound $M$ is given by $A_M(t) := \max\{1-t,t\}$ for every $t\in \mathbb{I}$ and that the Pickands dependence function associated with the independence copula $\Pi$ is defined by $A_\Pi(t) := 1$ for all $t \in \mathbb{I}$. Furthermore, it is straightforward to see that the Pickands dependence measure $\vartheta_M := \delta_\frac{1}{2}$ corresponds to $M$ and that the measure
$\vartheta_\Pi :=\frac{1}{2}(\delta_{0} + \delta_{1})$ corresponds to $\Pi$.
According to \citep{dietrich2024}, there is a one-to-one correspondence between the family of all Pickands dependence functions $\mathcal{A}$ and the family of all Pickands dependence measures $\mathcal{P}_\mathcal{A}$. To be more precise, given $\vartheta \in \mathcal{P}_\mathcal{A}$ and applying \citep[Lemma Appendix A.2]{dietrich2024}, the Pickands dependence function $A$ corresponding to $\vartheta$ is given by
\begin{equation}\label{eq:rel_pickands_fct_measure}
    A(t) = 1-t +2\int_{[0,t]}\vartheta([0,z]) \mathrm{d}\lambda(z),
\end{equation}
for every $t \in \mathbb{I}$. Let  $\vartheta \in \mathcal{P}_\mathcal{A}$ and denote by $A\in \mathcal{A}$ its corresponding Pickands dependence function. Using \citep[Lemma 4.3]{dietrich2024}, the right-hand derivative $D^+A$ of $A$ can be represented in terms of the Pickands dependence measure $\vartheta$. The following identity holds for every $t \in [0,1)$:
\begin{equation}\label{eq:right_der_pickands_measure}
D^+A(t) = 2\vartheta([0,t]) - 1. 
\end{equation} 
Following \citep{evc-mass} and considering $A \in \mathcal{A}$, we define 
\begin{equation}\label{eq:definition_L_R}
    L := \max\{s \in \mathbb{I} \colon A(s) = 1-s\} \text{ and } R := \min\{s \in \mathbb{I} \colon A(s) = s\}.
\end{equation}
As established in \citep{dietrich2024}, eq. \eqref{eq:definition_L_R} can also be expressed in terms of the corresponding $\vartheta \in \mathcal{P}_\mathcal{A}$:
\begin{equation}\label{eq:definition_L_R_measure}
L = \sup\{x \in \mathbb{I}\colon \vartheta([0,x]) = 0\} \text{ and } R = \inf\{x \in \mathbb{I}\colon \vartheta([0,x]) = 1\},
\end{equation}
with the convention $\sup \emptyset := 0$ and $\inf \emptyset := 1$.
Furthermore, considering $A \in \mathcal{A}$ we define the function $h_A \colon \mathbb{I} \rightarrow [1,\infty]$ by
\begin{equation}\label{eq:fct_h}
h_A(t) := \begin{cases}
    \infty,&\text{if } t = 0,\\
    \frac{A(t)}{t}, &\text{if } t\in (0,1]
\end{cases}
\end{equation}
for every $t \in \mathbb{I}$. Let $L$, $R$ be defined as in eq. \eqref{eq:definition_L_R}, then assuming $L > 0$ and using the fact that $A(t) = 1 - t$ for every $t \in [0,L]$ implies that $h_A(t) = \frac{1-t}{t}$ for every $t \in (0,L]$. What is more, considering $t \in [R,1]$ and applying that $A(t) = t$, obviously yields $h_A(t) = 1$. Lastly, since $A(t) \geq \max\{1-t,t\}$ holds for every $t \in \mathbb{I}$, we conclude that $h_A(t) \geq \frac{\max\{1-t,t\}}{t}$ for every $t \in (0,1]$.\\
According to Lemma \ref{lem:regularity_fcts} the function $h_A$ is strictly decreasing on $(0,\inf\{x \in \mathbb{I} \colon h_A(x) = 1\}) = (0,R)$ and therefore invertible on $(0,R)$ with inverse function $h_A^{-1}$ defined on $(1,\infty)$. Considering that $h_A(s) = 1$ holds for every $s \in [R,1]$, we define the pseudo-inverse $h_A^{[-1]} \colon [1,\infty] \rightarrow \mathbb{I}$ of $h_A$ by 
$$
h_A^{[-1]}(z) :=
\begin{cases}
    R, &\text{if } z = 1,\\
    h_A^{-1}(z), &\text{if } z \in (1,\infty),\\
    0, &\text{if } z = \infty,
\end{cases}
$$
for every $z \in [1,\infty]$. Note, that $h_A^{[-1]}(z) = h_A^{-1}(z)$ holds for every $z \in [1,\infty)$ in the case that $R = 1$. Building upon the afore-mentioned properties of the function $h_A$ and again assuming that $L>0$, it is straightforward to see that $h_A^{[-1]}(z) = \frac{1}{z+1}$ holds for every $z \in\left[\frac{1-L}{L},\infty\right)$. Using that $h_A(t) \geq \frac{1-t}{t}$ for every $t \in (0,1]$ immediately implies $h_A^{[-1]}(z) \geq \frac{1}{z+1}$ for every $z \in [1,\infty)$.\\
Similar to \citep{fuchs_tschimpke,evc-mass}, fixing $A \in \mathcal{A}$, we define the function $G_A \colon \mathbb{I} \rightarrow \mathbb{I}$ by
\begin{equation}\label{eq:funct_f_a}
G_A(t) :=
\begin{cases}
A(t) + D^+A(t)(1-t),& \text{if } t \in [0,1),\\
1,& \text{if } t = 1
\end{cases}
\end{equation}
for every $t \in \mathbb{I}$.
Concerning the regularity (in the sense of monotonicity and convexity) of the functions defined in this section, we prove the following lemma. The lemma serves as an auxiliary result, used to establish the findings of the next sections. For the sake of readability, the proof is provided in the Appendix.
\begin{Lemma}\label{lem:regularity_fcts}
Let $A \in \mathcal{A}$, $h_A$ be defined as in eq. \eqref{eq:fct_h} and let $L$, $R$ be defined according to eq. \eqref{eq:definition_L_R}. Then the following assertions hold:
   \begin{itemize}
        \item[$(i)$] $h_A$ is a non-increasing, convex function on $(0,1]$, which is strictly decreasing on $(0,R]$ and fulfills $h_A(t) = 1$ for every $t \in [R,1]$;
        \item[$(ii)$] $h_A^{[-1]}$ is a strictly decreasing and convex function on $[1,\infty)$;
        \item[$(iii)$] $G_A$ is a non-negative, non-decreasing and right-continuous function on $\mathbb{I}$, which is positive on $(L,1]$. Furthermore, if $L>0$, $G_A(t) = 0$ holds for every $t \in [0,L)$ and $G_A(t) = 1$ holds for every $t \in [R,1]$.
    \end{itemize}
\end{Lemma}
\begin{proof}
    See Appendix \ref{proof:regularty_fct}.
\end{proof}
Next, we give an explicit example of the functions $h_A$ and $h_A^{[-1]}$ for a fixed $A \in \mathcal{A}$.
\begin{Ex}\label{ex:ex_ha_and_inv}
We define the Pickands dependence function
\begin{align*}
    A(t)&:= \begin{cases}
        1-t,& \text{ if } t\in [0,\frac{1}{8}),\\
        t^2-t+\frac{63}{64},& \text{ if } t\in[\frac{1}{8},\frac{1}{4}),\\
        \frac{105}{128}-\frac{3}{32}t,& \text{ if } t\in [\frac{1}{4},\frac{3}{4}),\\
         t,& \text{ if } t\in [\frac{3}{4},1]r
    \end{cases}
\end{align*}
for every $t \in \mathbb{I}$. The function $h_A$ defined as in eq. \eqref{eq:fct_h} is therefore given by
\begin{align*}
    h_A(t)&= \begin{cases}
        \frac{1}{t}-1,& \text{ if } t\in (0,\frac{1}{8}),\\
        t-1+\frac{63}{64t},& \text{ if } t\in[\frac{1}{8},\frac{1}{4}),\\
        \frac{105}{128t}-\frac{3}{32},& \text{ if } t\in [\frac{1}{4},\frac{3}{4}),\\
         1,& \text{ if } t\in [\frac{3}{4},1]
    \end{cases}
\end{align*}
for every $t \in (0,1]$ and we set $h_A(0) = \infty$. Calculating the pseudo-inverse $h_A^{[-1]}$ of $h_A$ yields
\begin{align*}
    h_A^{[-1]}(z)& = \begin{cases}
        \frac{1}{z+1},& \text{ if } z\in [7,\infty),\\
        \frac{1}{2}\left(z+1-\sqrt{(z+1)^2-\frac{63}{16}}\right),& \text{ if } z\in[\frac{51}{16},7),\\
         \frac{105}{128z + 12},& \text{ if } z\in [1,\frac{51}{16}]
    \end{cases}
\end{align*}
for every $z \in [1,\infty)$. Moreover, from eq. \eqref{eq:definition_L_R}, it follows that $L = \frac{1}{8}$ and $R = \frac{3}{4}$.
\end{Ex}
\begin{figure}[!ht]
	\centering
	\includegraphics[width=1\textwidth]{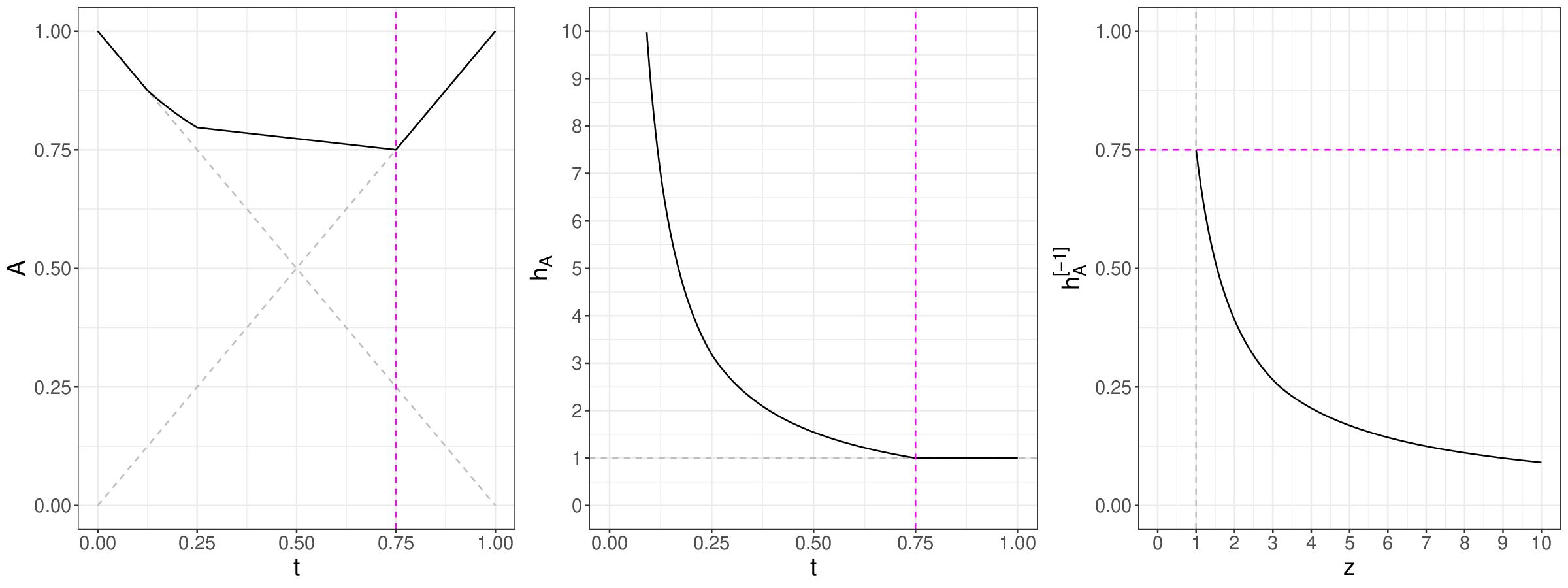}
	\caption{Plots of the Pickands dependence function $A$ (left) as defined in Example \ref{ex:ex_ha_and_inv}, the corresponding function $h_A$ (middle) and its pseudo-inverse $h_A^{[-1]}$ (right). The dashed magenta lines mark $R = \frac{3}{4}$ according to eq. \eqref{eq:definition_L_R}.}
	\label{fig:schneise}
\end{figure}
\subsubsection{Archimax copulas}
Archimax copulas arise when combining convex generators $\psi$ and Pickands dependence functions $A$. To be more precise, a copula $C \in \mathcal{C}$ is called an Archimax copula, if there exists a generator $\psi \in \mathbf{\Psi}$ with pseudo-inverse $\varphi$ and a Pickands dependence function $A \in \mathcal{A}$ such that
\begin{equation}\label{eq:def_archimax}
C(x,y) = \psi\left((\varphi(x) + \varphi(y))A\left(\frac{\varphi(x)}{\varphi(x) + \varphi(y)}\right)\right)
\end{equation}
holds for every $(x,y) \in (0,1)^2$. It can be shown that a function defined as in eq. \eqref{eq:def_archimax} is indeed a bona fide copula \citep{caperaa2000,dur_archimax}.
Stressing the relationship between $\psi$, $A$ and the corresponding copula $C$, we frequently write $C_{\psi,A}$ instead of $C$ throughout this article.
The class of all bivariate Archimax copulas is denoted by $\mathcal{C}_{am}$. Recalling the definition given in eq. \eqref{eq:def_abs_dis_sing_copula}, we denote by $\mathcal{C}_{am}^{abs}$, $\mathcal{C}_{am}^{dis}$ and $\mathcal{C}_{am}^{sing}$ the family of all absolutely continuous, discrete and singular Archimax copulas.
Given an arbitrary $\psi \in \mathbf{\Psi}$ and setting $A(x) = 1$ for every $x \in \mathbb{I}$, it is straightforward to prove that $C_{\psi,A} = C_\psi \in \mathcal{C}_{ar}$ holds. Setting $A(x) = \max\{1-x,x\}$ for every $x\in\mathbb{I}$, it is easily proved \citep{est-archimax} that $C_{\psi,A}(x,y) = M(x,y) := \min\{x,y\}$ holds for every $(x,y) \in \mathbb{I}^2$, regardless of the generator $\psi \in \mathbf{\Psi}$. Lastly, considering $\psi(z) = \exp(-z)$ for every $z \in [0,\infty)$, we obtain that $C_{\psi,A} = C_A \in \mathcal{C}_{ev}$ holds for every $A \in \mathcal{A}$ \citep{caperaa2000}.
\\
Let $\gamma \in \mathcal{P}_\mathcal{W}$ and $\vartheta \in \mathcal{P}_\mathcal{A}$. Building upon the one-to-one correspondences between $\mathcal{P}_\mathcal{W}$ and $\mathbf{\Psi}$, and between $\mathcal{P}_\mathcal{A}$ and $\mathcal{A}$, as discussed in Sections \ref{subsec:intro_archimedean} and \ref{subsec:intro_evc}, we see that $\gamma$ corresponds to a generator $\psi_\gamma$, and $\vartheta$ corresponds to a Pickands dependence function $A_\vartheta$. According to eq. \eqref{eq:def_archimax}, these then induce an Archimax copula $C_{\psi_\gamma,A_\vartheta} \in \mathcal{C}_{am}$. Put differently, every pair $(\gamma,\vartheta)$ of $\gamma \in \mathcal{P}_\mathcal{W}$ and $\vartheta \in \mathcal{P}_\mathcal{A}$ induces a unique Archimax copula $C_{\gamma,\vartheta} \in \mathcal{C}_{am}$.


\section{Level sets of Archimax copulas}\label{sec:lvl_sets_archimax}
\subsection{Representation of level sets via graphs of functions}
Following \citep{dur_princ}, the $t$-level set of a copula $C \in \mathcal{C}$ is defined by
\begin{equation}\label{eq:t-level_set}
L^t := \{(x,y) \in \mathbb{I}^2 \colon C(x,y) = t\}
\end{equation}
for every $t \in \mathbb{I}$.
Moreover, we define the lower $t$-cut of $C$ by
\begin{equation}\label{eq:lower_t_cut}
[C]_t := \{(x,y) \in \mathbb{I}^2 \colon C(x,y) \leq t\}
\end{equation}
for every $t \in \mathbb{I}$. Throughout this section we let $\psi\in \mathbf{\Psi}$, $A \in \mathcal{A}$, $C_{\psi,A} \in \mathcal{C}_{am}$ and $C_\psi \in \mathcal{C}_\psi$. Distinguishing between the level sets of $C_{\psi,A}$ and $C_\psi$, we will denote the $t$-level set of $C_{\psi,A}$ by $L_{\psi,A}^t$ and the $t$-level set of $C_\psi$ by $L_{\psi}^t$ for every $t \in \mathbb{I}$.\\
Motivated by the well known fact that in the Archimedean case the $t$-level set $L_\psi^t$ coincides with the graph of a convex function for every $t \in (0,1]$ \citep{mult_arch,neslehova,nelsen2006}, the primary goal of this section is to establish similar results for the family of Archimax copulas. To be more precise, we will prove that for fixed $t \in \mathbb{I}$, the $t$-level set $L_{\psi,A}^t$ can be represented by the two functions $f^t$ and $g^t$ specified in eq. \eqref{eq:fct_g_t} and \eqref{eq:fct_f}, respectively. Denoting by $\varphi$ the pseudo-inverse of $\psi$ and utilizing the function $h_A$ defined as in eq. \eqref{eq:fct_h}, the copula $C_{\psi,A}$ can also be represented as
\begin{equation}\label{eq:alternative_archimax_copula}
    C_{\psi,A}(x,y) = \psi\left(\varphi(x)h_A\left(\frac{\varphi(x)}{\varphi(x) + \varphi(y)}\right)\right)
\end{equation}
for every $x,y \in (0,1)$. Therefore, fixing $t \in (0,1)$ and $(x,y) \in L_{\psi,A}^t$ yields $C_{\psi,A}(x,y) = t$ and thus, formally solving the afore-mentioned equation for $y$, we arrive at the definition of the following functions.\\
Using the pseudo-inverse $h_A^{[-1]}$ of $h_A$ and considering $t \in (0,1)$ if $\psi$ is strict and $t \in [0,1)$ if $\psi$ is non-strict, we define the function $\fg^t \colon [t,1] \rightarrow \mathbb{I}$ by
\begin{equation}\label{eq:fct_g_t}
\fg^t(x) := 
\begin{cases}
\psi\left(\left(\frac{1}{h_A^{[-1]}\left(\frac{\varphi(t)}{\varphi(x)}\right)}-1\right)\varphi(x)\right), &\text{if } x \in [t,1),\\
t, &\text{if } x =1
\end{cases}
\end{equation}
for every $x \in [t,1]$. Furthermore, if $\psi$ is strict, we define the function $\fg^0 \colon \mathbb{I} \rightarrow \mathbb{I}$ by $\fg^0(x) := 0$ for every $x \in \mathbb{I}$. Finally, the function $\fg^1\colon \{1\} \rightarrow \mathbb{I}$ is defined by $\fg^1(1) := 1$.\\
Applying the fact that $h_A^{[-1]}(s) \geq \frac{1}{s+1}$ for every $s \in [1,\infty)$, it follows that for all $t \in \mathbb{I}$ the inequality $\fg^t(x) \geq t$ holds for every $x\in [t,1]$. The next example illustrates that in the Archimedean case $\Gamma(\fg^t)$ coincides with the $t$-level set $L_\psi^t$ of the corresponding $C_\psi \in \mathcal{C}_{ar}$ for every $t \in(0,1]$.
\begin{Ex}[Archimedean copulas and EVCs]\label{ex:basic_ex_lvl_sets}
    \leavevmode
    \begin{enumerate}
        \item Let $\psi \in \mathbf{\Psi}$ be arbitrary, $A(x) = 1$ for every $x \in \mathbb{I}$ and fix $t \in (0,1]$. Then  $C_{\psi,A} = C_\psi \in \mathcal{C}_{ar}$, $h_A(s) = \frac{1}{s}$ for every $s \in (0,1]$ and $h_A^{[-1]}(z) = \frac{1}{z}$ for every $z \in [1,\infty)$. Using the afore-mentioned facts, we therefore obtain that $\fg^t(x) = \psi(\varphi(t) - \varphi(x))$ for every $x \in [t,1]$. The fact that $\Gamma(\fg^t) = L_\psi^t$ is now obvious.
        In the case that $t = 0$ and $\psi$ is non-strict, $\fg^0(x) = \psi(\varphi(0) - \varphi(x))$ holds for every $x \in \mathbb{I}$ and thus, $L_{\psi}^0 = \{(x,y) \in \mathbb{I}^2 \colon y \leq \fg^0(x)\}$.
        \item Let $\psi(z) = \exp(-z)$ for every $z \in [0,\infty)$ with inverse $\varphi(x) = -\log(x)$ for every $x \in (0,1]$, $A \in \mathcal{A}$ be arbitrary and $t \in (0,1)$. Then $C_{\psi,A} = C_A \in \mathcal{C}_{ev}$ and
        $$
        \fg^t(x) = \begin{cases}
        \exp\left(\left(\frac{1}{h_A^{[-1]}\left(\frac{\log(t)}{\log(x)}\right)}-1\right)\log(x)\right),& \text{if } x\in[t,1),\\
        t,&  \text{if } x = 1 
        \end{cases}
        $$
        for every $x \in [t,1]$.
        \item Let $\psi \in \mathbf{\Psi}$ be arbitrary, $A(x) = \max\{1-x,x\}$ for every $x \in \mathbb{I}$ and  $t \in \mathbb{I}$, then $C_{\psi,A} = M$ as well as $\fg^t(x) = t$ for every $x \in [t,1]$.
        \end{enumerate}
\end{Ex}
Establishing the results of the next sections, the following function will play a key role. Considering $\psi\in \mathbf{\Psi}$ and $t \in (0,1)$, the function $\ff^t \colon \mathbb{I} \rightarrow \mathbb{I}$ is defined by
\begin{equation}\label{eq:fct_f}
\ff^t(x) := \psi\left(\left(\frac{1}{t} -1\right)\varphi(x)\right)
\end{equation}
for every $x \in \mathbb{I}$. Furthermore, we define the functions $\ff^0(x) := 0$ and $\ff^1(x) := 1$ for every $x \in \mathbb{I}$. Given $A \in \mathcal{A}$, fixing $t \in \mathbb{I}$, considering $R$ defined as in eq. \eqref{eq:definition_L_R} and applying the fact that $h_A^{[-1]}(z) \leq R$ holds for every $z\in [1,\infty)$ yields that $\fg^t(x) \leq \ff^R(x)$ for every $x \in [t,1]$.
The next example demonstrates that when the Archimax copula is an EVC, the function $\ff^t$ simplifies to $\ff^t(x) = x^{\frac{1}{t}-1}$ for every $x \in \mathbb{I}$. This is noteworthy since, according to \citep{evc-mass}, the discrete component of an EVC always concentrates its mass on $\Gamma(\ff^t)$. 
\begin{Ex}\label{ex:evc_fct_f_t}
    Let $A \in \mathcal{A}$ be arbitrary, $\psi(z) = \exp(-z)$ for every $z \in [0,\infty)$ with inverse $\varphi(x) = -\log(x)$ for every $x \in (0,1]$ and fix $t \in (0,1)$. Then $C_{\psi,A} = C_A \in \mathcal{C}_{ev}$ as well as $\ff^t(x) = x^{\frac{1}{t}-1}$ for every $x \in \mathbb{I}$.
\end{Ex}
Given $\psi \in \mathbf{\Psi}$, $A \in \mathcal{A}$ and $t \in (0,1]$, the $t$-level set $L_{\psi,A}^t$ of $C_{\psi,A} \in \mathcal{C}_{am}$ can be represented in terms of the functions $\fg^t$ and $\ff^R$, whereby $R\in [\frac{1}{2},1]$ is defined as in eq. \eqref{eq:definition_L_R}. Unlike the Archimedean case \citep{mult_arch,neslehova,nelsen2006}, the $t$-level sets $L_{\psi,A}^t$ do not necessarily correspond to $\Gamma(\fg^t)$. However, as proved in the next lemma, if $R = 1$, $L_{\psi,A}^t$ indeed coincides with $\Gamma(\fg^t)$.
\begin{Lemma}\label{lem:t-lvl_sets}
Let $\psi \in \mathbf{\Psi}$, $A \in \mathcal{A}$, $C_{\psi,A} \in \mathcal{C}_{am}$, $t \in (0,1)$, $\fg^t$ be defined as in eq. \eqref{eq:fct_g_t}, $R$ be defined according to eq. \eqref{eq:definition_L_R} and $g^R$ be defined as in eq. \eqref{eq:fct_f}. Then the following assertions hold:
\begin{itemize}
    \item[$(i)$] $L_{\psi,A}^t = \Gamma(\fg^t) \cup (\{t\} \times [\ff^R(t),1])$;
    \item[$(ii)$] $L_{\psi,A}^t =  \Gamma(\fg^t)$ if, and only if $R = 1$.
    \end{itemize}
\end{Lemma}
\begin{proof}
We prove the first assertion and show that $(x,y) \in L_{\psi,A}^t$ if, and only if $(x,y) \in\Gamma(\fg^t) \cup (\{t\} \times [\ff^R(t),1])$. Consider $(x,y) \in \mathbb{I}^2$, then in the cases $x = 0$ or $y = 0$ or $(x,y) = (1,1)$, both $(x,y) \notin L_{\psi,A}^t$ and $(x,y) \notin \Gamma(\fg^t) \cup (\{t\} \times [\ff^R(t),1])$ holds, due to the fact that $t \in (0,1)$.\\
If $(x,y) \in \{1\} \times (0,1)$ or $(x,y) \in (0,1) \times \{1\}$, then applying $\fg^t(1) = t$ and using the fact that $\fg^t(x) =1$ implies $x = t$, it is straightforward to prove that $(x,y) \in L_{\psi,A}^t$ is equivalent to $(x,y) \in \Gamma(\fg^t) \cup (\{t\} \times [\ff^R(t),1])$.\\
From now on let $(x,y) \in (0,1)^2$. Considering $h_A$ defined according to eq. \eqref{eq:fct_h}, taking into account that $(x,y) \in L_{\psi,A}^t$ implies $h_A\left(\frac{\varphi(x)}{\varphi(x) + \varphi(y)}\right)\varphi(x) < \varphi(0)$ and using eq. \eqref{eq:alternative_archimax_copula}, we obtain that $(x,y) \in L_{\psi,A}^t$ if, and only if the subsequent identity holds:
\begin{equation}\label{eq_equivalence}
h_A\left(\frac{\varphi(x)}{\varphi(x) + \varphi(y)}\right) = \frac{\varphi(t)}{\varphi(x)}.
\end{equation}
We distinguish two cases:
\begin{itemize}
    \item[$(a)$] If $R \leq \frac{\varphi(x)}{\varphi(x) + \varphi(y)}$, then using the fact that $h_A(s) = 1$ for every $s \in [R,1]$, we find that eq. \eqref{eq_equivalence} is equivalent to $x = t$. Together with the observation that the condition $R \leq \frac{\varphi(x)}{\varphi(x) + \varphi(y)}$ is equivalent to $\ff^R(x) \leq y$, this implies that eq. \eqref{eq_equivalence} holds if, and only if $(x,y) \in \{t\}\times [\ff^R(t),1]$.
    \item[$(b)$] If $\frac{\varphi(x)}{\varphi(x) + \varphi(y)} < R$, then applying the fact that $h_A$ is invertible on $(0,R)$ with inverse $h_A^{[-1]} = h_A^{-1}$ on $(1,\infty)$, and using that $\varphi(x)\left(\frac{1}{h_A^{[-1]}\left(\frac{\varphi(t)}{\varphi(x)}\right)}-1\right) < \varphi(0)$ whenever $y = \fg^t(x)$, equivalence preserving transformations show that eq. \eqref{eq_equivalence} holds if, and only if $y = \fg^t(x)$.
\end{itemize}
Now, assume that $(x,y) \in L_{\psi,A}^t$. Then, eq. \eqref{eq_equivalence} holds and in both cases, $(a)$ and $(b)$, we have $(x,y) \in \Gamma(\fg^t) \cup (\{t\}\times [\ff^R(t),1])$.
For $(x,y) \in \{t\}\times [\ff^R(t),1]$, it is clear that $R \leq \frac{\varphi(x)}{\varphi(x) + \varphi(y)}$, and consequently, $(a)$ implies that eq. \eqref{eq_equivalence} holds, yielding $(x,y) \in L_{\psi,A}^t$.\\
At last, we assume that $(x,y) \in \Gamma(\fg^t)$. It is always the case that $y = \fg^t(x) \leq \ff^R(x)$ and therefore $\frac{\varphi(x)}{\varphi(x) + \varphi(y)} \leq R$ holds. We distinguish two cases:
\begin{itemize}
    \item If $\frac{\varphi(x)}{\varphi(x) + \varphi(y)} < R$, $(b)$ directly yields $(x,y) \in L_{\psi,A}^t$.
\item If $\frac{\varphi(x)}{\varphi(x) + \varphi(y)} = R$, then it is obvious that $y = \ff^R(x)$, and hence $\fg^t(x) = \ff^R(x)$, which implies $x = t$. This results in $(x,y) = (t,\ff^R(t)) \in \{t\} \times [\ff^R(t),1]$. Applying $(a)$, we conclude that $(x,y) \in L_{\psi,A}^t$.
\end{itemize}
Combining all of the above cases proves the first assertion.\\
Considering the second assertion and assuming $R = 1$, it follows directly from the first assertion that $L_{\psi,A}^t = \Gamma(\fg^t)$. Proving the other direction, let $L_{\psi,A}^t = \Gamma(\fg^t)$ and suppose that $R < 1$. Then $\ff^R(t) <1$ would hold and thus, there would exist a $y_0 \in \mathbb{I}$ such that $\ff^R(t) < y_0 <1$. Applying the first assertion yields $\Gamma(\fg^t) = \Gamma(\fg^t) \cup (\{t\} \times [\ff^R(t),1])$ and therefore $(t,y_0) \in \{t\} \times [\ff^R(t),1] \subseteq \Gamma(\fg^t)$ would hold, implying that $\ff^R(t)<y_0 = \fg^t(t)$. This contradicts the fact that $\fg^t(x) \leq \ff^R(x)$ for every $x\in [t,1]$ and hence, $R = 1$ follows.
\end{proof}
Let $A\in \mathcal{A}$ be arbitrary and $\psi \in \mathbf{\Psi}$ be a strict generator, then $\psi(z) > 0$ holds for every $z \in [0,\infty)$ and thus, the $0$-level set $L_{\psi,A}^0$ of $C_{\psi,A}\in \mathcal{C}_{am}$ boils down to
$$
L_{\psi,A}^0 = (\{0\} \times \mathbb{I}) \cup (\mathbb{I} \times \{0\}).
$$
In the case that the generator $\psi$ is non-strict, the connection between the $0$-level set $L_{\psi,A}^0$ and the function $\fg^0$ is summarized in the following lemma.
\begin{Lemma}\label{lem:0-lvl_sets}
Let $\psi \in \mathbf{\Psi}$ be a non-strict generator, $A \in \mathcal{A}$, $C_{\psi,A} \in \mathcal{C}_{am}$, $\fg^0$ be defined as in eq. \eqref{eq:fct_g_t}, $R$ be defined according to eq. \eqref{eq:definition_L_R} and $g^R$ be defined as in eq. \eqref{eq:fct_f}. Then the following assertions hold:
\begin{itemize}
        \item[$(i)$] $L_{\psi,A}^0 = \{(x,y) \in \mathbb{I}^2 \colon y \leq \fg^{0}(x)\} \cup (\{0\} \times [\ff^R(0),1])$;
        \item[$(ii)$] $L_{\psi,A}^0 = \{(x,y) \in \mathbb{I}^2 \colon y \leq \fg^{0}(x)\}$ if, and only if $R = 1$.
\end{itemize}
\end{Lemma}
\begin{proof}
The proof is similar to the proof of Lemma \ref{lem:t-lvl_sets}. For the sake of completeness a proof is given in Appendix \ref{proof:0-lvl_sets}.
\end{proof}
\subsection{\texorpdfstring{Convexity and monotonicity properties of the functions $f^s$ and $g^t$}{Convexity and monotonicity properties of f and g}}
As shown in Example \ref{ex:basic_ex_lvl_sets}, in the case that $\psi \in \mathbf{\Psi}$ is arbitrary and $A(x) = 1$ holds for every $x \in \mathbb{I}$, $C_{\psi,A} = C_\psi \in \mathcal{C}_{ar}$ and the $t$-level set $L_\psi^t$ of $C_\psi$ coincides with $\Gamma(\fg^t)$ for every $t \in (0,1]$. It is well-known \citep[Theorem 4.3.2]{nelsen2006} that in this particular case the function $\fg^t$ is convex for every $t \in \mathbb{I}$.
The main goal of this subsection is to show that the function $\fg^t$ is non-increasing and convex for every $t \in [0,1)$ regardless of the choice of $\psi$ and $A$.\\
Let $\psi\in\mathbf{\Psi}$, $\varphi$ be its pseudo-inverse, $A \in \mathcal{A}$, $h_A$ be defined according to eq. \eqref{eq:fct_h} and $h_A^{[-1]}$ be its pseudo-inverse. Furthermore, we fix $t \in (0,1)$ if $\psi$ is strict and $t \in [0,1)$ if $\psi$ is non-strict. We define the function $\xi^t \colon (t,1) \rightarrow [0, \infty)$ by
    \begin{equation}\label{eq:fct_xi_t}
    \xi^t(x) := \left(\frac{1}{h_A^{[-1]}\left(\frac{\varphi(t)}{\varphi(x)}\right)}- 1\right) \varphi(x)
    \end{equation}
for every $x \in (t,1)$. Obviously, $\fg^t(x) = \psi(\xi^t(x))$ holds for every $x \in (t,1)$. Moreover, $\xi^t$ is a continuous function, as it is the composition of continuous functions. The next lemma proves that the functions $\xi^t$ are non-decreasing with non-increasing left-hand derivative $D^-\xi^t$, which together with the fact that $\psi$ is non-increasing and convex will prove that the functions $\fg^t$ are convex.
\begin{Lemma}\label{lem:auxiliary_lvl_curves_ii}
    Let $\psi \in \mathbf{\Psi}$, $\varphi$ be its pseudo-inverse, $A\in\mathcal{A}$, $h_A$ be defined according to eq. \eqref{eq:fct_h} and $L$, $R$ be defined as in eq. \eqref{eq:definition_L_R}. Let $t \in (0,1)$, if $\psi$ is strict and $t \in [0,1)$ if  $\psi$ is non-strict. Then the following assertions hold: 
    \begin{itemize}
        \item[$(i)$] $\xi^t$ is a non-decreasing function;
    \item[$(ii)$] If $0<L<R$, then $\xi^t$ is strictly increasing on $\left(t,\psi\left(\frac{L\varphi(t)}{1-L}\right)\right)$ and $\xi^t(x) = \varphi(t)$ holds for every $x \in \left[\psi\left(\frac{L\varphi(t)}{1-L}\right),1\right)$;
    \item[$(iii)$] $D^-\xi^t$ is a non-increasing function.
    \end{itemize}
\end{Lemma}
\begin{proof}
    See Appendix \ref{proof:auxiliary_lvl_curves_ii}.
\end{proof}
The subsequent theorem proves that the function $\fg^t$ is non-increasing and convex for every $t \in [0,1)$.
\begin{theorem}\label{thm:convexity_level_curves}
    Let $\psi \in \mathbf{\Psi}$, $A \in \mathcal{A}$, $t \in (0,1)$ if $\psi$ is strict and $t \in [0,1)$ if $\psi$ is non-strict. Moreover, let $\fg^t$ be defined as in eq. \eqref{eq:fct_g_t}, and $L$, $R$ be defined according to eq. \eqref{eq:definition_L_R}. Then the following assertions hold:
    \begin{itemize}
    \item[$(i)$] $\fg^t$ is a non-increasing function;
        \item[$(ii)$] If $0 < L < R$, then $\fg^t$ is strictly decreasing on $\left[t,\psi\left(\frac{L\varphi(t)}{1-L}\right)\right)$ and $\fg^t(x) = t$ holds for every $x \in \left[\psi\left(\frac{L\varphi(t)}{1-L}\right),1\right]$;
        \item[$(iii)$] $\fg^t$ is a convex function.
    \end{itemize}
\end{theorem}
\begin{proof}
 We fix $t \in (0,1)$ in the case that $\psi$ is strict and $t \in [0,1)$ in the case that $\psi$ is non-strict. Considering $\xi^t$ defined as in eq. \eqref{eq:fct_xi_t} and using that $\fg^t(x) = \psi(\xi^t(x))$ holds for every $x \in (t,1)$, the fact that $\fg^t$ is a non-increasing function on $(t,1)$ is a direct consequence of Lemma \ref{lem:auxiliary_lvl_curves_ii} and the non-increasingness of $\psi$. Applying Lemma \ref{lem:left_cont_g^t}, the function $\fg^t$ is obviously continuous on $[t,1]$ and therefore $\fg^t$ is non-increasing on the whole interval $[t,1]$, which proves the first assertion.\\
Having that $\xi^t$ is strictly increasing on $\left(t,\psi\left(\frac{L\varphi(t)}{1-L}\right)\right)$ and that $\psi$ is strictly decreasing on $[0,\varphi(0))$ implies strict decreasingness of $\fg^t$ on $\left[t,\psi\left(\frac{L\varphi(t)}{1-L}\right)\right)$. Moreover, taking into account that $\xi^t(x) = \varphi(t)$ for every $x \in \left[\psi\left(\frac{L\varphi(t)}{1-L}\right),1\right)$ and $\fg^t(1) = t$, we obtain that $\fg^t(x) = t$ holds for every $x \in \left[\psi\left(\frac{L\varphi(t)}{1-L}\right),1\right]$. This proves the second assertion.\\
 We show that $\fg^t$ is convex on $[t,1]$. Applying Lemma \ref{lem:auxiliary_lvl_curves_ii}, $\xi^t$ is a non-decreasing function whose left-hand derivative $D^-\xi^t$ exists. Hence, using convexity of $\psi$, applying Lemma \ref{lem:properties_left_hand_der} and calculating the left-hand derivative of $\fg^t$ yields
 $$
 D^-\fg^t(x) = D^-\psi(\xi^t(x))\cdot D^-\xi^t(x)
 $$
 for every $x \in (t,1)$. 
 Using the properties in Lemma \ref{lem:auxiliary_lvl_curves_ii} and the fact that $\psi$ is non-increasing and convex, we conclude that
 $D^-\fg^t$ is non-decreasing on $(t,1)$, which, according to \citep[Appendix C]{Pollard2001}, implies convexity of $\fg^t$ on $(t,1)$. Applying continuity of $\fg^t$, convexity of $\fg^t$ on the whole interval $[t,1]$ follows.
\end{proof}
The next theorem proves that the function $\ff^t$ is non-decreasing for every $t \in (0,1)$.
\begin{theorem}
Let $\psi \in \mathbf{\Psi}$, $\varphi$ be its pseudo-inverse, $t \in (0,1)$ and $\ff^t$ be defined according to eq. \eqref{eq:fct_f}, then 
$\ff^t$ is non-decreasing on $\mathbb{I}$, strictly increasing on $\left[\psi\left(\frac{t\varphi(0)}{1-t}\right),1\right]$ and $\ff^t(x) = 0$ holds for every $x \in \left[0,\psi\left(\frac{t\varphi(0)}{1-t}\right)\right]$.
\end{theorem}
\begin{proof}
We fix $t \in (0,1)$. Using that $\psi$ and $\varphi$ are both non-increasing implies that $\ff^t$ is non-decreasing. If $\psi$ is strict, then $\psi\left(\frac{t\varphi(0)}{1-t}\right) = \psi(\infty) = 0$ holds and therefore applying the monotonicity properties of $\psi$ and $\varphi$ yields that $\ff^t$ is strictly increasing. If $\psi$ is non-strict and $x \in \left[\psi\left(\frac{t\varphi(0)}{1-t}\right),1\right]$, then $(\frac{1}{t}-1)\varphi(x) \leq \varphi(0)$ follows and hence, applying that $\psi$ is strictly decreasing on $[0,\varphi(0)]$ proves that $\ff^t$ is strictly increasing on $\left[\psi\left(\frac{t\varphi(0)}{1-t}\right),1\right]$. Using that  $\ff^t(x) = 0$ whenever $\varphi(0) \leq \frac{(1-t)\varphi(x)}{t}$, implies that $\ff^t(x) = 0$ for every $x \in \left[0,\psi\left(\frac{t\varphi(0)}{1-t}\right)\right]$. This proves the result.
\end{proof}
The subsequent example illustrates some of the results of this section.
\begin{Ex}\label{ex:lvl_sets_example}
Let $A \in \mathcal{A}$, $h_A$ and its pseudo-inverse $h_A^{[-1]}$ be as in Example \ref{ex:ex_ha_and_inv}. Moreover, we define the generator $\psi \in \mathbf{\Psi}$ by
\begin{align*}
    \psi(z) := \begin{cases}
        0,& \text{ if } z \in (8,\infty),\\
        \frac{4}{7}(1-\frac{z}{8}),& \text{ if } z \in (1,8],\\
        \frac{6z+1}{14z},& \text{ if } z \in (\frac{1}{2},1],\\
        \frac{-8z^2+18z+1}{28z},& \text{ if } z \in (\frac{1}{3},\frac{1}{2}],\\
        \frac{1}{28}(28-29z),& \text{ if } z \in [0,\frac{1}{3}]\\
    \end{cases}
\end{align*}
for every $z \in [0,\infty)$. The function $\fg^0$ defined as in eq. \eqref{eq:fct_g_t} is given by
$$
\fg^0(x) = \psi\left(\left(\frac{1}{h_A^{[-1]}\left(\frac{8}{\varphi(x)}\right)}-1\right)\varphi(x)\right)
$$
for every $x \in [0,1)$ and we have $\fg^0(1) = 0$. Recall, that in this example $L = \frac{1}{8}$, $R = \frac{3}{4}$ holds and therefore the functions $\ff^R$ and $\ff^L$ defined as in eq. \eqref{eq:fct_f} are given by $\ff^L(x) = \psi(7\varphi(x))$ and $\ff^R(x) = \psi\left(\frac{\varphi(x)}{3}\right)$ for every $x \in \mathbb{I}$. Furthermore, the $t$-level sets $L_{\psi,A}^t$ of the copula $C_{\psi,A} \in \mathcal{C}_{am}$ are depicted in Figure \ref{fig:lvl_sets_copula} for some $t \in \mathbb{I}$.
\end{Ex}
\begin{figure}[!ht]
	\centering
	\includegraphics[width=13cm,height=13cm]{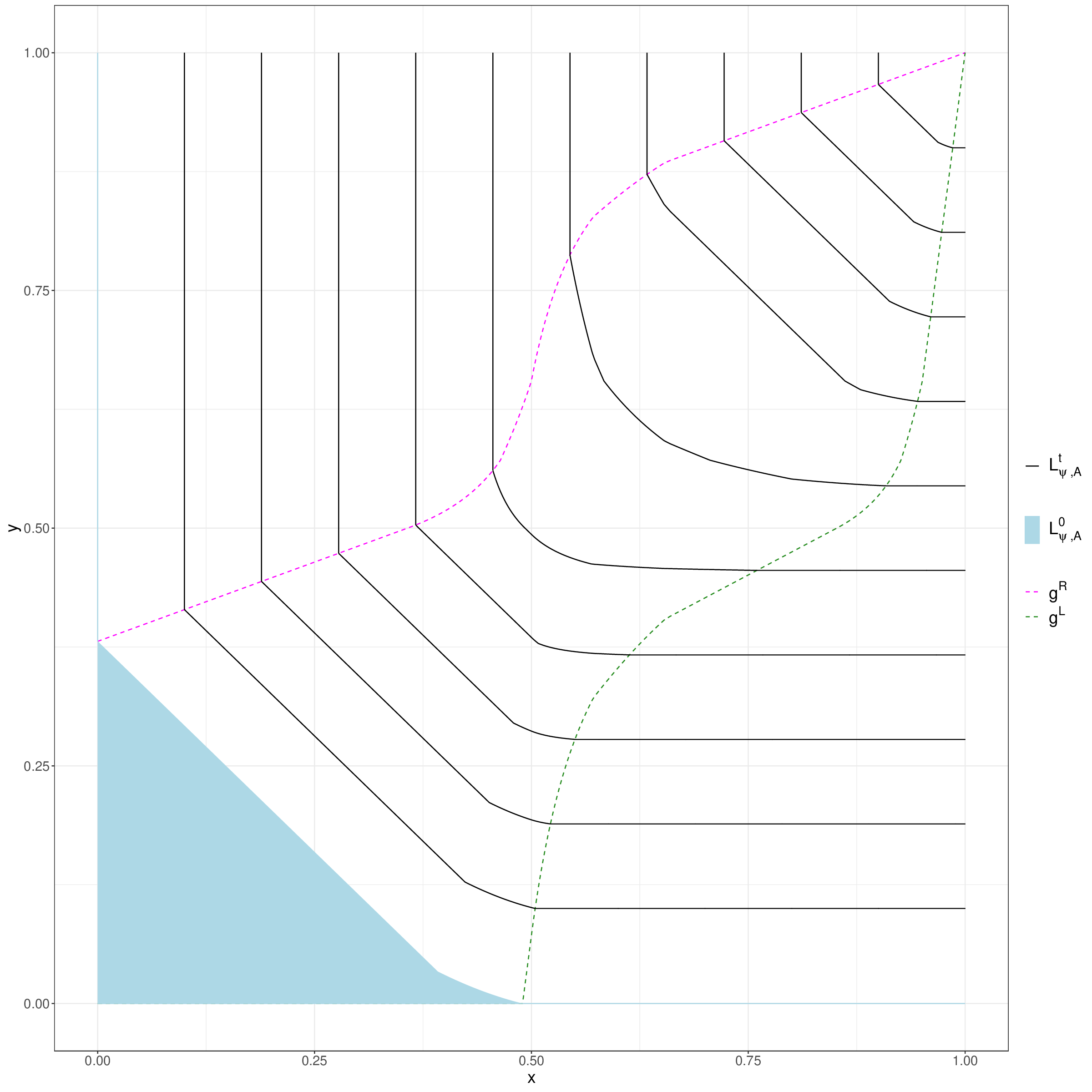}
	\caption{Plots of the $0$-level set $L_{\psi,A}^0$ (lightblue) and $t$-level sets  $L_{\psi,A}^t$ (black) of the Archimax-copula $C_{\psi,A}$ with generator $\psi$ and Pickands dependence function $A$ as defined in Example \ref{ex:lvl_sets_example} for $t \in \{\frac{1}{10},\frac{17}{90},\frac{5}{18},\frac{11}{30},\frac{41}{90},\frac{49}{90},\frac{19}{30},\frac{13}{18},\frac{73}{90},\frac{9}{10}\}$, as well as plots of the functions $g^R$ (dashed magenta) and $g^L$ (dashed green) with $R = \frac{3}{4}$ and $L = \frac{1}{8}$.}
	\label{fig:lvl_sets_copula}
\end{figure}


\section{Markov kernels, mass distributions and discrete components of Archimax copulas} \label{sec:markov_mass_dist_discr}
\subsection{Markov kernels and mass distributions}
Proving the results of the next sections, Markov kernels will be an essential tool. Let $\psi \in \mathbf{\Psi}$, $A \in \mathcal{A}$ and $C_{\psi,A} \in \mathcal{C}_{am}$. Then the next theorem constitutes a version of the Markov kernel $K_{\psi,A}$ of $C_{\psi,A}$.
Recall that we consistently use the convention that $a\cdot\infty = \infty$, $\frac{a}{\infty} = 0$ and $\frac{a}{0} = \infty$ for $a \in \mathbb{R} \setminus \{0\}$.
\begin{theorem}\label{thm:markov_kernel}
Let $\psi \in \mathbf{\Psi}$, $A \in \mathcal{A}$, $h_A$ be defined as in eq. \eqref{eq:fct_h}, $G_A$ be defined according to eq. \eqref{eq:fct_f} and $C_{\psi,A} \in \mathcal{C}_{am}$. Then the function $K_{\psi,A} \colon \mathbb{I}^2 \rightarrow \mathbb{I}$ defined by
\begin{align}\label{eq:markov-kernel}
K_{\psi,A}(x,[0,y]) := \begin{cases}
1,& \text{if } x\in\{0,1\},\\
\frac{D^{-}\psi\left(\varphi(x)h_A\left(\frac{\varphi(x)}{\varphi(x) + \varphi(y)}\right)\right)}{D^-\psi(\varphi(x))}\cdot G_A\left(\frac{\varphi(x)}{\varphi(x) + \varphi(y)}\right),& \text{if } (x,y) \in \Upsilon_0^{(1)},\\
0,& \text{if } (x,y) \in \Upsilon_0^{(2)},
\end{cases}
\end{align}
for every $(x,y) \in \mathbb{I}^2$ is a version of the Markov-kernel of $C_{\psi,A}$, whereby $\Upsilon_0^{(1)} := \{(x,y) \in (0,1) \times \mathbb{I}\colon \fg^0(x) \leq y\}$ and $\Upsilon_0^{(2)} := \{(x,y) \in (0,1) \times \mathbb{I}\colon \fg^0(x) > y\}$.
\end{theorem}
\begin{proof}
Fix $x \in \{0,1\}$, then $y \mapsto K_{\psi,A}(x,[0,y])$ is obviously a distribution function.\\
Considering $x \in (0,1)$, we obtain that $K_{\psi,A}(x,[0,y]) = 0$ for every $y < \fg^0(x)$ and therefore the function $y \mapsto K_{\psi,A}(x,[0,y])$ is non-decreasing and right-continuous on $[0,\fg^0(x))$. Note that in the case that $\psi$ is strict $\fg^0(x) = 0$ holds for every $x \in \mathbb{I}$ and hence, the previous case does not occur.\\
We prove that $y \mapsto K_{\psi,A}(x,[0,y])$ is non-decreasing on $[\fg^0(x),1]$. Applying Lemma \ref{lem:regularity_fcts}, using convexity of $\psi$ and the fact that $\varphi$ is strictly decreasing implies non-increasingness of the function $y \mapsto D^-\psi\left(\varphi(x)h_A\left(\frac{\varphi(x)}{\varphi(x) + \varphi(y)}\right)\right)$. Thus, considering that $D^-\psi(\varphi(x))$ is negative, the first factor in the second line of eq. \eqref{eq:markov-kernel} is non-decreasing. Having that the function $G_A$ is non-decreasing (see Lemma \ref{lem:regularity_fcts}), it follows that also the second factor in eq. \eqref{eq:markov-kernel} is non-decreasing and therefore $y \mapsto K_{\psi,A}(x,[0,y])$ is a non-decreasing function on $[\fg^0(x),1]$.\\
The right-continuity of $y \mapsto K_{\psi,A}(x,[0,y])$ on $[\fg^0(x),1]$ follows from the right-continuity of $G_A$, the left-continuity of $D^-\psi$, and the fact that both $h_A$ and $\varphi$ are non-increasing functions. Combining all of the above observations, $y \mapsto K_{\psi,A}(x,[0,y])$ is a distribution function for fixed $x \in \mathbb{I}$. As usual, working with the extension of $K_{\psi,A}(x,\cdot)$ from $\{[0,y] \colon y \in \mathbb{I}\}$ to $\mathcal{B}(\mathbb{I})$ yields a unique probability measure $K_{\psi,A}(x,\cdot)$ for every $x \in \mathbb{I}$. Moreover, fixing $y \in \mathbb{I}$, the function $x \mapsto K_{\psi,A}(x,[0,y])$ is clearly Borel-measurable. Thus, using a standard Dynkin argument, $x \mapsto K_{\psi,A}(x,F)$ is Borel-measurable for every $F \in \mathcal{B}(\mathbb{I})$. Combining all of the afore-mentioned findings, we conclude that $K_{\psi,A}$ is indeed a Markov-kernel form $\mathbb{I}$ to $\mathcal{B}(\mathbb{I})$.\\
It is left to show that $K_{\psi,A}$ is (a version of) the Markov-kernel of $C_{\psi,A}$. We obtain that
\begin{align}\label{eq:proof_markov_kernel}
   \nonumber \int_{[0,x]}&K_{\psi,A}(s,[0,y]) \mathrm{d}\lambda(s) \\&\nonumber = \int_{(0,x] \cap \{v \in \mathbb{I} \colon \fg^0(v) \leq y\}}\frac{D^-\psi\left(\varphi(s)h_A\left(\frac{\varphi(s)}{\varphi(s) + \varphi(y)}\right)\right)}{D^-\psi(\varphi(s))} G_A\left(\frac{\varphi(s)}{\varphi(s) + \varphi(y)}\right)\mathrm{d}\lambda(s) \\&=
    \int_{(0,x]}\frac{D^-\psi\left(\varphi(s)h_A\left(\frac{\varphi(s)}{\varphi(s) + \varphi(y)}\right)\right)}{D^-\psi(\varphi(s))} G_A\left(\frac{\varphi(s)}{\varphi(s) + \varphi(y)}\right)\mathrm{d}\lambda(s),
\end{align}
for every $(x,y) \in \mathbb{I}^2$, whereby the last equation follows from the fact that $\varphi(s)h_A\left(\frac{\varphi(s)}{\varphi(s) + \varphi(y)}\right) > \varphi(0)$, whenever $y < \fg^0(s)$ and therefore $D^-\psi\left(\varphi(s)h_A\left(\frac{\varphi(s)}{\varphi(s) + \varphi(y)}\right)\right) = 0$ holds in the case that $y < \fg^0(s)$.\\
We prove that for fixed $y \in \mathbb{I}$ the function $s\mapsto C_{\psi,A}(s,y)$ is differentiable for all but at most countably many points $s \in \mathbb{I}$. If $y \in \{0,1\}$, then $s\mapsto C_{\psi,A}(s,y)$ is obviously differentiable on $(0,1)$. Considering $L \in [0,\frac{1}{2}]$ defined as in eq. \eqref{eq:definition_L_R}, assuming $L > 0$ and fixing $y\in(0,1)$, we observe that $\frac{\varphi(s)}{\varphi(s) + \varphi(y)} < L$ whenever $s \in \left(\psi\left(\frac{L\varphi(y)}{1-L}\right),1\right)$ holds. Recalling that $h_A(t) = \frac{1-t}{t}$ for every $t \in (0,L]$, it therefore follows that
$$
C_{\psi,A}(s,y) = \psi\left(\varphi(s)h_A\left(\frac{\varphi(s)}{\varphi(s) + \varphi(y)}\right)\right) = \psi(\varphi(y)) = y
$$
for every $s \in \left(\psi\left(\frac{L\varphi(y)}{1-L}\right),1\right)$ and thus, $\partial_1C_{\psi,A}(s,y) = 0$ holds for every $s \in \left(\psi\left(\frac{L\varphi(y)}{1-L}\right),1\right)$.\\
From now on we assume that $L \in \left[0,\frac{1}{2}\right]$.
Considering that both $\psi$ and $A$ are convex, the derivatives $A'$ and $\psi'$ exist except for at most countably many points and coincide (wherever they exist) with the left-hand and right-hand derivative, respectively.
Using Lemma \ref{lem:auxiliary_markov_kernel} yields strict decreasingness of the function $s \mapsto (\varphi(s) + \varphi(y))A\left(\frac{\varphi(s)}{\varphi(s) + \varphi(y)}\right)$ on the interval $\left(0,\psi\left(\frac{L\varphi(y)}{1-L}\right)\right)$. Furthermore, taking into account that both the functions $\varphi$ and $s \mapsto \frac{\varphi(s)}{\varphi(s) + \varphi(y)}$ are strictly decreasing, we obtain that the set $\Lambda \in \mathcal{B}(\mathbb{I})$ of all points $s \in \mathbb{I}$ at which the derivative $\partial_1C_{\psi,A}(s,y)$ of the function $s \mapsto C_{\psi,A}(s,y)$ does not exist is at most countable infinite. Applying the chain rule and using eq. \eqref{eq:proof_markov_kernel} finally yields
$$
 \int_{[0,x]}K_{\psi,A}(s,[0,y]) \mathrm{d}\lambda(s) = \int_{[0,x]}\partial_1 C_{\psi,A}(s,y) \mathrm{d}\lambda(s) = C_{\psi,A}(x,y)
$$
for every $x,y \in \mathbb{I}$. This proves the result.
\end{proof}
\begin{Rem}[Archimedean copulas and EVCs]
\leavevmode
\begin{enumerate}
    \item If $\psi \in \mathbf{\Psi}$ and $A(t) = 1$ for every $t \in \mathbb{I}$, then the Markov-kernel $K_{\psi,A}$ of $C_{\psi,A}\in \mathcal{C}_{am}$ defined in eq. \eqref{eq:markov-kernel} coincides with the Markov kernel of $C_\psi \in \mathcal{C}_{ar}$ in \citep[Theorem 3.1]{mult_arch}.
    \item If $\psi(z) = \exp(-z)$ for every $z \in [0,\infty)$ and $A\in \mathcal{A}$, then the Markov-kernel $K_{\psi,A}$ of $C_{\psi,A}\in \mathcal{C}_{am}$ given in eq. \eqref{eq:markov-kernel} coincides with the Markov kernel of $C_A \in \mathcal{C}_{ev}$ in \citep[Equation (19)]{evc-mass}.
\end{enumerate}
\end{Rem}
Consider $\psi \in \mathbf{\Psi}$, $A \in \mathcal{A}$, $C_{\psi,A} \in \mathcal{C}_{am}$ and let $L$, $R$ be defined as in eq. \eqref{eq:definition_L_R}. Furthermore let $\fg^0$, $\ff^L$ and $\ff^R$ be defined as in eq. \eqref{eq:fct_g_t} and \eqref{eq:fct_f}, respectively. We define the set
\begin{equation}\label{eq:defin_set_s_l_r}
S_{L,R} := \left\{(x,y) \in \mathbb{I}^2\colon \max\left\{\fg^0(x),\ff^L(x)\right\} \leq y \leq \ff^R(x)\right\}
\end{equation}
and prove that the support of $C_{\psi,A}$ is always a subset of $S_{L,R}$.
\begin{Lemma}\label{lemma:support}
    Let $\psi \in \mathbf{\Psi}$, $A \in \mathcal{A}$, $C_{\psi,A} \in \mathcal{C}_{am}$ and let $L$, $R$ be defined according to eq. \eqref{eq:definition_L_R}. Then 
    \begin{equation}\label{eq:support}
    \mathrm{supp}(\mu_{C_{\psi,A}}) \subseteq S_{L,R}
    \end{equation}
    holds.
\end{Lemma}
\begin{proof}
Note that in the case that $\psi$ is strict, $L= 0$ and $R= 1$, it is the case that $S_{L,R} = \mathbb{I}^2$ holds and therefore $\mathrm{supp}(\mu_{C_{\psi,A}}) \subseteq S_{L,R}$ is obvious.\\
We assume from now on that $\psi$ is non-strict or $L > 0$ or $R < 1$ and prove the desired result by contraposition, i.e., we prove that $(x,y) \in S_{L,R}^c$ implies $(x,y) \notin \mathrm{supp}(\mu_{C_{\psi,A}})$. It suffices to prove that $\mu_{C_{\psi,A}}(S_{L,R}^c) = 0$. We fix $x \in (0,1)$ and show that the Markov kernel $K_{\psi,A}$ of $C_{\psi,A}$ defined in eq. \eqref{eq:markov-kernel} fulfills
$$
K_{\psi,A}(x,[0,\max\{\fg^0(x),\ff^L(x)\})\cup (\ff^R(x),1]) = 0.
$$
Using eq. \eqref{eq:markov-kernel} obviously implies that $K_{\psi,A}(x,[0,\fg^0(x))) = 0$. Assuming $\fg^0(x) < \ff^L(x)$, applying eq. \eqref{eq:markov-kernel} and Lemma \ref{lem:regularity_fcts}, we obtain that
    $$
    K_{\psi,A}(x,[0,\ff^L(x))) =\frac{D^-\psi\left(\frac{\varphi(x)(1-L)}{L}+\right)}{D^-\psi(\varphi(x))}\underbrace{G_A(L-)}_{=0} = 0
    $$
    holds. Assuming $\ff^L(x) \leq \fg^0(x)$, obviously yields $K_{\psi,A}(x,[0,\ff^L(x))) = 0$. Finally, applying Lemma \ref{lem:regularity_fcts}, we conclude that
    $$
    K_{\psi,A}(x,[0,\ff^R(x)]) =\underbrace{\frac{D^-\psi(\varphi(x)h_A(R))}{D^-\psi(\varphi(x))}}_{=1}\underbrace{G_A(R)}_{=1} = 1
    $$
    and therefore $K_{\psi,A}(x,(\ff^R(x),1]) = 1- K_{\psi,A}(x,[0,\ff^R(x)]) = 0$ follows. Combining all of the above and applying disintegration proves that $$
    \mu_{C_{\psi,A}}(S_{L,R}^c) = \int_\mathbb{I} K_{\psi,A}(x,[0,\max\{\fg^0(x),\ff^L(x)\})\cup (\ff^R(x),1]) \mathrm{d}\lambda(x) = 0
    $$
    and hence $\mathrm{supp}(\mu_{C_{\psi,A}}) \subseteq S_{L,R}$ follows.
\end{proof}
Recall that, as discussed in Section \ref{section:notation}, every pair $(\gamma,\vartheta)$ of $\gamma \in \mathcal{P}_\mathcal{W}$ and $\vartheta \in \mathcal{P}_\mathcal{A}$ induces a copula $C_{\gamma,\vartheta} \in \mathcal{C}_{am}$.
Consider an arbitrary Williamson measure $\gamma \in \mathcal{P}_\mathcal{W}$ with corresponding generator $\psi$ and pseudo-inverse $\varphi$ and let $A \in \mathcal{A}$ with $h_A$ being defined as in eq. \eqref{eq:fct_h}. Then, for fixed $x \in (0,1)$ we define the function $H_x \colon \mathbb{I} \rightarrow \mathbb{I}$ by
\begin{equation}\label{eq:function_H_x}
H_x(y) := \frac{\int_{I_y}t \mathrm{d}\gamma(t)}{\int_{I_1}t\mathrm{d}\gamma(t)}
\end{equation}
with the notation $I_y := \left[0,\frac{1}{\varphi(x)h_A\left(\frac{\varphi(x)}{\varphi(x) + \varphi(y)}\right)}\right]$ for every $y \in \mathbb{I}$.
Utilizing eq. \eqref{eq:rel_deriv_psi_gamma} and fixing $x \in (0,1)$, we can now express the Markov kernel $K_{\psi,A}$ of $C_{\psi,A} \in \mathcal{C}_{am}$ given in eq. \eqref{eq:markov-kernel} by
\begin{equation}\label{eq:alternative_markov_kernel}
K_{\psi,A}(x,[0,y]) = H_x(y)G_A\left(\frac{\varphi(x)}{\varphi(x) + \varphi(y)}\right)
\end{equation}
for every $y \in \mathbb{I}$. Note that in the case that $y < \fg^0(x)$, obviously $\frac{1}{\varphi(x)h_A\left(\frac{\varphi(x)}{\varphi(x) + \varphi(y)}\right)} < \frac{1}{\varphi(0)}$ follows and therefore, applying that $\gamma([0,z]) = 0$ for every $z < \frac{1}{\varphi(0)}$, $H_x(y) = 0$ holds for every $y \in [0,\fg^0(x))$.
 Building upon the afore-mentioned facts and taking into account Lemma \ref{lemma:support}, the question arises naturally, under which conditions on $\gamma$ we have equality in eq. \eqref{eq:support}. The next theorem proves that if the support of $\gamma$ coincides with $\left[\frac{1}{\varphi(0)},\infty\right)$, then $\mathrm{supp}(\mu_{C_{\psi,A}}) = S_{L,R}$ holds.
\begin{theorem}\label{thm:arch_full_support_archimax}
Let $\gamma \in \mathcal{P}_\mathcal{W}$, $\psi$ be the corresponding generator, $\vartheta \in \mathcal{P}_\mathcal{A}$ and $A \in \mathcal{A}$ be the corresponding Pickands dependence function. Furthermore, let $C_{\psi,A} \in \mathcal{C}_{am}$, $L$, $R$ be defined according to eq. \eqref{eq:definition_L_R} and $S_{L,R}$ be defined as in eq. \eqref{eq:defin_set_s_l_r}. If $\mathrm{supp}(\gamma) = \left[\frac{1}{\varphi(0)},\infty\right)$, then
$$
\mathrm{supp}(\mu_{C_{\psi,A}}) = S_{L,R}.
$$
\end{theorem}
\begin{proof}
If $L = R$, then $A(x) =\max\{1-x,x\}$ for every $x \in \mathbb{I}$ and therefore $C_{\psi,A} = M$, $\fg^0(x) = 0$ and $\ff^L(x) = x = \ff^R(x)$ holds for every $x \in \mathbb{I}$. Thus, it follows that
$$
\mathrm{supp}(\mu_{C_{\psi,A}}) = \mathrm{supp}(\mu_M) = \{(x,y) \in \mathbb{I}^2 \colon y= x\} = S_{L,R}.
$$
If $L<R$, we fix $x \in (0,1)$ and prove that $y \mapsto K_{\psi,A}(x,[0,y])$ (according to eq. \eqref{eq:alternative_markov_kernel}) is strictly increasing on $(\max\{\fg^0(x),\ff^L(x)\},\ff^R(x))$. Considering $G_A$ defined as in eq. \eqref{eq:funct_f_a} and applying Lemma \ref{lem:regularity_fcts}, we obtain that the function $y \mapsto G_A\left(\frac{\varphi(x)}{\varphi(x) + \varphi(y)}\right)$ is non-decreasing and positive on $(\ff^L(x),1]$. Using the fact that $h_A$ is strictly decreasing on $(0,R]$ and that $\varphi$ is strictly decreasing on $(0,1]$ yields that the function $y \mapsto \frac{1}{\varphi(x)h_A\left(\frac{\varphi(x)}{\varphi(x) + \varphi(y)}\right)}$ is strictly increasing on $(0,\ff^R(x)]$. Moreover, since $\varphi(x)h_A\left(\frac{\varphi(x)}{\varphi(x) + \varphi(y)}\right) < \varphi(0)$ holds for every $y \in (\fg^0(x),1]$ and $\mathrm{supp}(\gamma) = \left[\frac{1}{\varphi(0)},\infty\right)$, it follows that the function $y \mapsto H_x(y) = \frac{\int_{I_y}t \mathrm{d}\gamma(t)}{\int_{I_1}t\mathrm{d}\gamma(t)}$ is strictly increasing and positive on $(\fg^0(x),\ff^R(x))$. Altogether, applying eq. \eqref{eq:alternative_markov_kernel}, we obtain that the function $y \mapsto K_{\psi,A}(x,[0,y])$ is strictly increasing on $(\max\{\fg^0(x),\ff^L(x)\},\ff^R(x))$, as it is the product of a positive non-decreasing and a positive strictly increasing function.\\
Now, let $(a,b) \times (c,d) \subset S_{L,R}$ with $\lambda_2((a,b) \times (c,d)) > 0$. Then obviously $(c,d) \subset (\max\{\fg^0(x),\ff^L(x)\},\ff^R(x))$ for every $x \in (a,b)$ and therefore applying disintegration, we conclude that
$$
\mu_{C_{\psi,A}}((a,b) \times (c,d)) = \int_{(a,b)} \underbrace{K_{\psi,A}(x,(c,d))}_{> 0} \mathrm{d}\lambda(x) > 0.
$$
Put differently, every non-empty open rectangle in $S_{L,R}$ has positive mass, implying that $S_{L,R} \subseteq \mathrm{supp}(\mu_{C_{\psi,A}})$. Utilizing Lemma \ref{lemma:support} finally yields $\mathrm{supp}(\mu_{C_{\psi,A}}) = S_{L,R}$.
\end{proof}
Similar to the previous Theorem we show that $\mathrm{supp}(\mu_{C_{\psi,A}}) = S_{L,R}$ if the Pickands dependence measure $\vartheta$ fulfills $\mathrm{supp}(\vartheta) = [L,R]$.
\begin{theorem}\label{thm:evc_full_support_archimax}
Let $\gamma \in \mathcal{P}_\mathcal{W}$, $\psi$ be its generator, $\vartheta \in \mathcal{P}_\mathcal{A}$, $A \in \mathcal{A}$ be its Pickands dependence function and $C_{\psi,A} \in \mathcal{C}_{am}$. Furthermore, let $L$, $R$ be defined according to eq. \eqref{eq:definition_L_R} and $S_{L,R}$ be defined as in eq. \eqref{eq:defin_set_s_l_r}. If $\mathrm{supp}(\vartheta) = [L,R]$, then
$$
\mathrm{supp}(\mu_{C_{\psi,A}}) = S_{L,R}.
$$
\end{theorem}
\begin{proof}
If $L = R$, then the same arguments as in the proof of Theorem \ref{thm:arch_full_support_archimax} yield $\mathrm{supp}(\mu_{C_{\psi,A}}) = S_{L,R}$.\\
If $L < R$, we fix $x \in (0,1)$, consider the Markov kernel $K_{\psi,A}$ of $C_{\psi,A}$ and prove that $y \mapsto K_{\psi,A}(x,[0,y])$ is strictly increasing on $(\max\{\fg^0(x),\ff^L(x)\},\ff^R(x))$. Applying similar arguments as in the proof of Theorem \ref{thm:arch_full_support_archimax} implies that $H_x$ defined as in eq. \eqref{eq:function_H_x} is non-decreasing and positive on $(\fg^0(x),\ff^R(x))$ and that the function $y \mapsto G_A\left(\frac{\varphi(x)}{\varphi(x) + \varphi(y)}\right)$ is positive on $(\ff^L(x),1]$.\\
We prove that $G_A$ is strictly increasing on $(L,R)$ and fix $t_1,t_2 \in (L,R)$ arbitrarily. Applying that $\mathrm{supp}(\vartheta) = [L,R]$ and using eq. \eqref{eq:right_der_pickands_measure}, we obtain that $D^+A(t_1) < D^+A(t_2)$. Therefore, using convexity of $A$, we conclude that
\begin{align*}
    G_A(t_2) - G_A(t_1) &= A(t_2) - A(t_1) + D^+A(t_2)(1-t_2) - D^+A(t_1)(1-t_1) \\& >
    A(t_2) - [A(t_1) + D^+A(t_1) (t_2-t_1)] \geq 0.
\end{align*}
Having that $\varphi$ is strictly decreasing yields that $y \mapsto G_A\left(\frac{\varphi(x)}{\varphi(x) + \varphi(y)}\right)$ is strictly increasing on $(\ff^L(x),\ff^R(x))$ and thus, utilizing eq. \eqref{eq:alternative_markov_kernel}, the function $y \mapsto K_{\psi,A}(x,[0,y])$ is strictly increasing on $(\max\{\fg^0(x),\ff^L(x)\},\ff^R(x))$. Applying the same chain of arguments as in the proof of Theorem \ref{thm:arch_full_support_archimax} shows that $\mathrm{supp}(\mu_{C_{\psi,A}}) = S_{L,R}$.
\end{proof}
Considering Theorem \ref{thm:arch_full_support_archimax} and Theorem \ref{thm:evc_full_support_archimax}, we determine under which conditions on $\gamma \in \mathcal{P}_\mathcal{W}$ and $\vartheta \in \mathcal{P}_\mathcal{A}$ the corresponding Archimax copula $C_{\gamma,\vartheta}$ has full support.
\begin{Cor}\label{cor:full_support_archimax}
    Let $\gamma \in \mathcal{P}_\mathcal{W}$, $\vartheta \in \mathcal{P_A}$, $C$ be its corresponding Archimax copula and let $L$, $R$ be defined according to eq. \eqref{eq:definition_L_R_measure}. Then the following assertions hold:
    \begin{itemize}
        \item[$(i)$] If $\gamma$ has full support and $L = 0$ as well as $R = 1$, then $\mathrm{supp}(\mu_{C}) = \mathbb{I}^2$.
        \item[$(ii)$] If $\gamma$ is strict and $\vartheta$ has full support, then $\mathrm{supp}(\mu_{C}) = \mathbb{I}^2$.
        \item[$(iii)$] If $\mathrm{supp}(\mu_{C}) = \mathbb{I}^2$, then $\gamma$ is strict and $L = 0$ as well as $R = 1$.
    \end{itemize}
\end{Cor}
\begin{proof}
    Proving the first assertion, applying that $\mathrm{supp}(\gamma) = [0,\infty)$ yields that $\gamma$ and therefore the corresponding generator $\psi$ is strict, which implies that $\fg^0(x) = 0$ holds for every $x \in \mathbb{I}$. Having that $L = 0$ and $R = 1$, we conclude that $\ff^L(x) = 0$ and $\ff^R(x) = 1$ for every $x \in \mathbb{I}$. Theorem \ref{thm:arch_full_support_archimax} then implies that $\mathrm{supp}(\mu_C) =S_{0,1} = \mathbb{I}^2$.\\
    Considering the second assertion, proceeding similarly to the proof of the first assertion and applying Theorem \ref{thm:evc_full_support_archimax} yields $\mathrm{supp}(\mu_{C}) = S_{0,1} = \mathbb{I}^2$.\\
    We prove the third assertion and assume that $\mathrm{supp}(\mu_{C}) = \mathbb{I}^2$. Suppose that $\gamma$ was non-strict or $L > 0$ or $R < 1$ would hold. Then according to Lemma \ref{lemma:support}, $\mathrm{supp}(\mu_{C}) \subseteq S_{L,R} \subsetneq \mathbb{I}^2$ would hold in either of the cases. A contradiction.
\end{proof}
Given $A\in \mathcal{A}$ and $C_A \in \mathcal{C}_{ev}$, we derive a characterization of $\mathrm{supp}(\mu_{C_A})$, which was first established in \citep{Mai2011} and later proved in \citep[Corollary 4]{evc-mass}. The next corollary is an immediate consequence of Theorem \ref{thm:arch_full_support_archimax}.
\begin{Cor}[Mai and Scherer (2011)]\label{cor:mai_scherer_supp_evc}
Let $A\in \mathcal{A}$, $C_{A} \in \mathcal{C}_{ev}$ and $L$, $R$ be defined as in eq. \eqref{eq:definition_L_R}. Then
$$
\mathrm{supp}(\mu_{C_A}) = \{(x,y)\in \mathbb{I}^2\colon \ff^L(x) \leq y \leq  \ff^R(x)\},
$$
whereby $\ff^L(x) = x^{\frac{1}{L}-1}$ and $\ff^R(x) = x^{\frac{1}{R}-1}$ for every $x \in \mathbb{I}$.
\end{Cor}
\begin{proof}
Let $A\in \mathcal{A}$ and $C_{A} \in \mathcal{C}_{ev}$. We define $\psi(z) := \exp(-z)$ for every $z \in [0,\infty)$ and obtain that the corresponding $\gamma \in \mathcal{P}_\mathcal{W}$ is given by $\gamma([0,z]) = \exp(-\frac{1}{z}) + \exp(-\frac{1}{z})\frac{1}{z}$ for every $z \in [0,\infty)$. It is straightforward to see that $\gamma$ has full support and is therefore strict. Moreover, $C_A = C_{\psi,A} \in \mathcal{C}_{am}$, $\fg^0(x) = 0$, $\ff^L(x) = x^{\frac{1}{L}-1}$ and $\ff^R(x) = x^{\frac{1}{R}-1}$ holds for every $x \in \mathbb{I}$ and thus, applying Theorem \ref{thm:arch_full_support_archimax} yields
$$
\mathrm{supp}(\mu_{C_A}) = \mathrm{supp}(\mu_{C_{\psi,A}}) = S_{L,R} = \{(x,y)\in \mathbb{I}^2\colon \ff^L(x) \leq y \leq  \ff^R(x)\}.
$$
\end{proof}
\subsection{Discrete components of Archimax copulas}
Throughout this section let $\gamma \in \mathcal{P}_\mathcal{W}$, $\psi$ be its generator, $\vartheta \in \mathcal{P}_\mathcal{A}$ and $A \in \mathcal{A}$ be its Pickands dependence function. Furthermore, let $C_{\psi,A} \in \mathcal{C}_{am}$ with corresponding doubly stochastic measure $\mu_{C_{\psi,A}}$ and Markov kernel $K_{\psi,A}$ (defined as in eq. \eqref{eq:markov-kernel}). In this section we prove that the discrete component $\mu_{C_{\psi,A}}^{dis}$ of $\mu_{C_{\psi,A}}$ (if any) always concentrates its mass on the graphs of the functions $\fg^s$ or $\ff^t$ for some $s,t \in [0,1)$.\\
If $L<R$, the next lemma shows that for 'good' $x$ the probability measure $K_{\psi,A}(x,\cdot)$ has a point mass in $\fg^t(x)$ if, and only if $\gamma$ has a point mass in $\frac{1}{\varphi(t)}$.
\begin{Lemma}\label{lem:kernel_jump_on_graph_arch}
Let $\gamma \in \mathcal{P}_\mathcal{W}$, $\psi$ be its corresponding generator, $\varphi$ be its pseudo-inverse, $\vartheta \in \mathcal{P}_\mathcal{A}$, $A$ be the corresponding Pickands dependence function, $C_{\psi,A} \in \mathcal{C}_{am}$, $K_{\psi,A}$ be the Markov kernel of $C_{\psi,A}$ according to eq. \eqref{eq:markov-kernel} and $L$, $R$ be defined according to eq. \eqref{eq:definition_L_R} with $L<R$. 
Moreover, let $t \in (0,1)$ if $\psi$ is strict, $t \in [0,1)$ if $\psi$ is non-strict and let $\fg^t$ be defined according to eq. \eqref{eq:fct_g_t}. Then for all but at most countably many $x \in \left(t, \psi\left(\frac{L\varphi(t)}{1-L}\right)\right)$ the following equivalence holds:
\begin{equation*}
    K_{\psi,A}(x,\{\fg^t(x)\})>0 \text{ if, and only if } \frac{1}{\varphi(t)} \text{ is a point mass of } \gamma.
\end{equation*}
\end{Lemma}
\begin{proof}
We fix $t \in (0,1)$ in the case that $\psi$ is strict, $t\in [0,1)$ in the case that $\psi$ is non-strict and assume that $L<R$. We denote by $I$ the set of all point masses of $\vartheta$ in $(0,1)$. Note that $I$ is at most countable infinite and is empty if $\vartheta$ has no point masses in $(0,1)$. Let $h_A$ be defined as in eq. \eqref{eq:fct_h}, $h_A^{[-1]}$ be its pseudo-inverse and $G_A$ be defined as in eq. \eqref{eq:funct_f_a}. Applying eq. \eqref{eq:right_der_pickands_measure} and fixing $s_0 \in (0,1)$ yields that $s_0$ is a point mass of $\vartheta$ if, and only if $s_0$ is a point of discontinuity of $D^+A$. Using the afore-mentioned fact and assuming that $I \neq \emptyset$, we obtain that $h_A^{[-1]}\left(\frac{\varphi(t)}{\varphi(x)}\right)$ is a point of discontinuity of $D^+A$ (and therefore of $G_A$) if, and only if $x \in \bigcup_{s \in I}\left\{\psi\left(\frac{\varphi(t)}{h_A(s)}\right)\right\}$. If $I = \emptyset$, then $G_A$ is obviously continuous. We fix $x \in \left(t, \psi\left(\frac{L\varphi(t)}{1-L}\right)\right) \setminus \bigcup_{s \in I}\left\{\psi\left(\frac{\varphi(t)}{h_A(s)}\right)\right\}$ and note that $\bigcup_{s \in I}\left\{\psi\left(\frac{\varphi(t)}{h_A(s)}\right)\right\} = \emptyset$, if $I$ is empty. Applying the afore-mentioned arguments we obtain that $G_A$ is continuous in $h_A^{[-1]}\left(\frac{\varphi(t)}{\varphi(x)}\right)$ and thus, working with eq. \eqref{eq:alternative_markov_kernel} yields that
    \begin{align*}
        &K_{\psi,A}(x,\{\fg^t(x)\}) = K_{\psi,A}(x,[0,\fg^t(x)]) - K_{\psi,A}(x,[0,\fg^t(x))) \\& =\frac{\int_{[0,\frac{1}{\varphi(t)}]} u \mathrm{d}\gamma(u)}{\int_{[0,\frac{1}{\varphi(x)}]} u \mathrm{d}\gamma(u)}G_A\left(h_A^{[-1]}\left(\frac{\varphi(t)}{\varphi(x)}\right)\right)  - \frac{\int_{[0,\frac{1}{\varphi(t)})} u \mathrm{d}\gamma(u)}{\int_{[0,\frac{1}{\varphi(x)}]} u \mathrm{d}\gamma(u)}G_A\left(h_A^{[-1]}\left(\frac{\varphi(t)}{\varphi(x)}\right)-\right) \\& =
        \frac{\int_{\{\frac{1}{\varphi(t)}\}} u \mathrm{d}\gamma(u)}{\int_{[0,\frac{1}{\varphi(x)}]} u \mathrm{d}\gamma(u)}G_A\left(h_A^{[-1]}\left(\frac{\varphi(t)}{\varphi(x)}\right)\right).
    \end{align*}
    Since $x < \psi\left(\frac{L\varphi(t)}{1-L}\right)$, we obtain $L < h_A^{[-1]}\left(\frac{\varphi(t)}{\varphi(x)}\right)$ and therefore $G_A\left(h_A^{[-1]}\left(\frac{\varphi(t)}{\varphi(x)}\right)\right) > 0$ holds according to Lemma \ref{lem:regularity_fcts}. If $\frac{1}{\varphi(t)}$ is a point mass of $\gamma$, then obviously $K_{\psi,A}(x,\{\fg^t(x)\}) > 0$ holds. If $K_{\psi,A}(x,\{\fg^t(x)\}) > 0$, then $\int_{\{\frac{1}{\varphi(t)}\}} u \mathrm{d}\gamma(u) > 0$ follows and thus, $\frac{1}{\varphi(t)}$ is a point mass of $\gamma$. This proves the result.
\end{proof}
Similarly to Lemma \ref{lem:kernel_jump_on_graph_arch}, the next lemma shows that for 'good' $x$ the probability measure $K_{\psi,A}(x,\cdot)$ has a point mass in $\ff^t(x)$ if, and only if $\vartheta$ has a point mass in $t$.
\begin{Lemma}\label{lem:kernel_jump_on_graph_evc}
Let $\gamma \in \mathcal{P}_\mathcal{W}$, $\psi$ be the corresponding generator, $\vartheta \in \mathcal{P}_\mathcal{A}$, $A$ be the corresponding Pickands dependence function, $C_{\psi,A} \in \mathcal{C}_{am}$ and $K_{\psi,A}$ be the Markov kernel of $C_{\psi,A}$ according to eq. \eqref{eq:markov-kernel}.
Furthermore, let $h_A$ be defined as in eq. \eqref{eq:fct_h}, $t \in (0,1)$ and let $\ff^t$ be defined according to eq. \eqref{eq:fct_f}. Then for all but at most countably many $x \in \left( \psi\left(\frac{\varphi(0)}{h_A(t)}\right),1\right)$ the following equivalence holds:
\begin{equation*}
    K_{\psi,A}(x,\{\ff^t(x)\})>0 \text{ if, and only if } t \text{ is a point mass of } \vartheta.
\end{equation*}
\end{Lemma}
\begin{proof}
Let $t \in (0,1)$ be arbitrary and let $J$ denote the set of all point masses of $\gamma$. Again, $J$ is at most countable infinite and $J = \emptyset$ holds if $\gamma$ has no point masses. It is straightforward to verify that $\gamma$ has a point mass in $\frac{1}{\varphi(x)h_A(t)}$ if, and only if $x \in \bigcup_{s \in J}\left\{\psi\left(\frac{1}{sh_A(t)}\right)\right\}$. We fix $x \in \left(\psi\left(\frac{\varphi(0)}{h_A(t)}\right),1\right)\setminus \bigcup_{s \in J}\left\{\psi\left(\frac{1}{sh_A(t)}\right)\right\}$ and obtain that in this case $\gamma$ has no point mass in $\frac{1}{\varphi(x)h_A(t)}$. Applying this fact and using eq. \eqref{eq:alternative_markov_kernel} again yields that
 \begin{align*}
     K_{\psi,A}(x,\{\ff^t(x)\}) &= K_{\psi,A}(x,[0,\ff^t(x)]) - K_{\psi,A}(x,[0,\ff^t(x))) \\&=
     \frac{\int_{[0,\frac{1}{\varphi(x)h_A(t)}]} u \mathrm{d}\gamma(u)}{\int_{[0,\frac{1}{\varphi(x)}]} u \mathrm{d}\gamma(u)}G_A(t) - \frac{\int_{[0,\frac{1}{\varphi(x)h_A(t)})} u \mathrm{d}\gamma(u)}{\int_{[0,\frac{1}{\varphi(x)}]} u \mathrm{d}\gamma(u)}G_A(t-) \\&=
     \underbrace{\frac{\int_{[0,\frac{1}{\varphi(x)h_A(t)}]} u \mathrm{d}\gamma(u)}{\int_{[0,\frac{1}{\varphi(x)}]} u \mathrm{d}\gamma(u)} (1-t)}_{=: w(x,t)}2\vartheta(\{t\}),
 \end{align*}
whereby in the last line we applied eq. \eqref{eq:right_der_pickands_measure} and used that $G_A(u) = D^+A(u)(1-u) + A(u)$ for every $u \in (0,1)$. Since $\psi\left(\frac{\varphi(0)}{h_A(t)}\right) < x$, we obtain that $\frac{1}{\varphi(0)} < \frac{1}{\varphi(x) h_A(t)}$ and therefore $w(x,t) > 0$ follows. The desired equivalence is now obvious.
\end{proof}
Lemma \ref{lem:kernel_jump_on_graph_arch} and Lemma \ref{lem:kernel_jump_on_graph_evc} are now used to prove the following theorem, stating that the discrete component $\mu_{C}^{dis}$ of an Archimax copula $C$ is non-degenerated if, and only if the corresponding $\gamma \in \mathcal{P}_{\mathcal{W}}$ or the corresponding $\vartheta \in \mathcal{P}_{\mathcal{A}}$ has a point mass.
\begin{theorem}
   Let $\gamma \in \mathcal{P}_\mathcal{W}$, $\vartheta \in \mathcal{P}_\mathcal{A}$ and $C$ be the corresponding Archimax copula. Then the following assertions are equivalent:
   \begin{itemize}
       \item[$(i)$] $\mu_C^{dis}(\mathbb{I}^2) > 0$;
       \item[$(ii)$] $\gamma$ has a point mass $z \in (0,\infty)$ or $\vartheta$ has a point mass $t \in (0,1)$.
   \end{itemize}
   Moreover, if $\mu_C^{dis}$ is non-degenerated, then its mass is concentrated on graphs $\Gamma(\fg^t)$ or $\Gamma(\ff^s)$ for some $s \in (0,1)$ or $t \in [0,1)$.
\end{theorem}
\begin{proof}
Let $\gamma \in \mathcal{P}_\mathcal{W}$, $\psi$ be its generator, $\vartheta \in \mathcal{P}_\mathcal{A}$ and $A$ be the the corresponding Pickands dependence function. Furthermore let $C_{\psi,A} \in \mathcal{C}_{am}$ with Markov kernel $K_{\psi,A}$ defined as in eq. \eqref{eq:markov-kernel} and let $L$, $R$ be defined according to eq. \eqref{eq:definition_L_R_measure}. At first, assume that $L = R$. Then $\vartheta = \delta_\frac{1}{2}$ and therefore $C_{\psi,A} = M$ holds. Since $\mu_M^{dis}(\mathbb{I}^2) = 1$, the desired equivalence follows.\\
    From now on we assume that $L< R$. If $z \in (0,\infty)$ is a point mass of $\gamma$, then using the fact that the corresponding $\varphi$ is strictly decreasing yields the existence of  $t \in [0,1)$ such that $z = \frac{1}{\varphi(t)}$ and therefore applying Lemma \ref{lem:kernel_jump_on_graph_arch} implies that $K_{\psi,A}(x,\{\fg^t(x)\}) > 0$ for $\lambda$-almost every $x \in \left(t, \psi\left(\frac{L\varphi(t)}{1-L}\right)\right)$.
    Assuming that $t \in (0,1)$ is a point mass of $\vartheta$ and using Lemma \ref{lem:kernel_jump_on_graph_evc}, we conclude that $K_{\psi,A}(x,\{\ff^t(x)\}) > 0$ for $\lambda$-almost every $x \in \left(\psi\left(\frac{\varphi(0)}{h_A(t)}\right),1\right)$. In either of the afore-mentioned cases disintegration proves $\mu_{C_{\psi,A}}^{dis}(\mathbb{I}^2) > 0$.\\
    We assume that $\mu_C^{dis}(\mathbb{I}^2) > 0$, then by definition there exists a set $\Lambda \in \mathcal{B}(\mathbb{I})$ with $\lambda(\Lambda) > 0$ such that for every $x \in \Lambda \cap (0,1)$ there exists a $y := y_x \in [\max\{\fg^0(x),g^L(x)\},g^R(x)]$ with $K_{\psi,A}(x,\{y\}) > 0$. Using eq. \eqref{eq:alternative_markov_kernel} yields that
    \begin{align*}
       0 < K_{\psi,A}(x,\{y\}) = H_x(y)\cdot G_A\left(\frac{\varphi(x)}{\varphi(x) + \varphi(y)}\right) - H_x(y-)\cdot G_A\left(\frac{\varphi(x)}{\varphi(x) + \varphi(y)}-\right)
    \end{align*}
    with $H_x(y) = \frac{\int_{I_y}u \mathrm{d}\gamma(u)}{\int_{I_1}u \mathrm{d}\gamma(u)}$ and $I_y = \left[0,\frac{1}{\varphi(x)h_A\left(\frac{\varphi(x)}{\varphi(x) + \varphi(y)}\right)}\right]$. Thus, $u = \frac{1}{\varphi(x)h_A\left(\frac{\varphi(x)}{\varphi(x) + \varphi(y)}\right)}$ is a point mass of $\gamma$ or $s = \frac{\varphi(x)}{\varphi(x) + \varphi(y)}$ is a point of discontinuity of $G_A$, otherwise $K_{\psi,A}(x,\{y\}) = 0$ would hold. If $u$ is a point mass of $\gamma$, obviously $(ii)$ holds. Assuming that $s$ is a point of discontinuity of $G_A$ yields that $s$ is a point of discontinuity of $D^+A$ and therefore applying eq. \eqref{eq:right_der_pickands_measure} proves $(ii)$. \\
    Proving the last assertion and using that $\frac{\varphi(x)}{\varphi(x) + \varphi(y)} \leq R$ whenever  $y \leq g^R(x)$, we obtain that $u = \frac{1}{\varphi(x)h_A\left(\frac{\varphi(x)}{\varphi(x) + \varphi(y)}\right)}$ implies $y = f^{\psi(\frac{1}{u})}(x)$ and $s = \frac{\varphi(x)}{\varphi(x) + \varphi(y)}$ implies $y = g^s(x)$. Assuming that $x \in (0,1)$ and that $y \in [\max\{\fg^0(x),g^L(x)\},g^R(x)]$ is a point mass of $K_{\psi,A}(x,\cdot)$, combining the afore-mentioned arguments yields that $y = f^t(x)$ or $y = g^s(x)$ for some $t \in [0,1)$ or $s \in (0,1)$. Applying disintegration proves the desired result.
\end{proof}
The next example illustrates the results of this subsection.
\begin{Ex}\label{ex:jumps_markov_kernel_graphs}
Consider the distribution function defined by
$$
F_\gamma(z) := \begin{cases}
    0,& \text{ if } z \in [0,\frac{1}{8}),\\
    \frac{4}{7},& \text{ if } z \in [\frac{1}{8},1),\\
    \frac{1}{7}z+\frac{3}{7},& \text{ if } z \in [1,2),\\
    \frac{1}{14}z+\frac{9}{14},& \text{ if } z \in [2,3),\\
    1, & \text{ if } z \in [3,\infty)
\end{cases}
$$
for every $z \in [0, \infty)$. It is straightforward to see that $F_\gamma$ induces a unique $\gamma \in \mathcal{P}_\mathcal{W}$ which corresponds to the generator $\psi$ in Example \ref{ex:lvl_sets_example}. Moreover, we define the distribution function $F_\vartheta$ by
$$
F_\vartheta(t) := \begin{cases}
    0,& \text{ if } t \in [0,\frac{1}{8}),\\
    t,& \text{ if } t \in [\frac{1}{8},\frac{1}{4}),\\
    \frac{29}{64},& \text{ if } t \in [\frac{1}{4},\frac{3}{4}),\\
    1, & \text{ if } t \in [\frac{3}{4},1]
\end{cases}
$$
for every $t \in \mathbb{I}$. Obviously $F_\vartheta$ induces a unique $\vartheta \in \mathcal{P}_\mathcal{A}$, which corresponds to the Pickands dependence function $A$ in Example \ref{ex:ex_ha_and_inv}. Note, that $L$, $R$ defined according to eq. \eqref{eq:definition_L_R} are given by $L = \frac{1}{8}$ and $R = \frac{3}{4}$, respectively, and therefore $L<R$ holds. Let $C_{\psi,A}$ be the Archimax copula corresponding to $\psi$, $A$ and let $K_{\psi,A}$ be its Markov kernel defined as in eq. \eqref{eq:markov-kernel}.
Obviously $\gamma$ has point masses $\{\frac{1}{8},2,3\}$ and therefore applying Lemma \ref{lem:kernel_jump_on_graph_arch} yields that for every $t \in \{0,\frac{4}{7}, \frac{55}{84}\}$
$$
K_{\psi,A}(x,\{\fg^t(x)\}) > 0
$$
holds for all but at most countably many $x \in \left(t, \psi\left(\frac{L\varphi(t)}{1-L}\right)\right)$. Taking into account that $\vartheta$ has point masses $\{\frac{1}{8},\frac{1}{4},\frac{3}{4}\}$, using Lemma \ref{lem:kernel_jump_on_graph_evc}, we conclude that for every $s \in \{\frac{1}{8},\frac{1}{4},\frac{3}{4}\}$
$$
K_{\psi,A}(x,\{\ff^s(x)\}) > 0
$$
holds for all but at most countably many $x \in \left( \psi\left(\frac{\varphi(0)}{h_A(s)}\right),1\right)$. A sample of $C_{\psi,A}$ is depicted in 
Figure \ref{fig:sample_first_ex_archimax_10000}.
\begin{figure}[!ht]
	\centering
\includegraphics[width=1\textwidth]{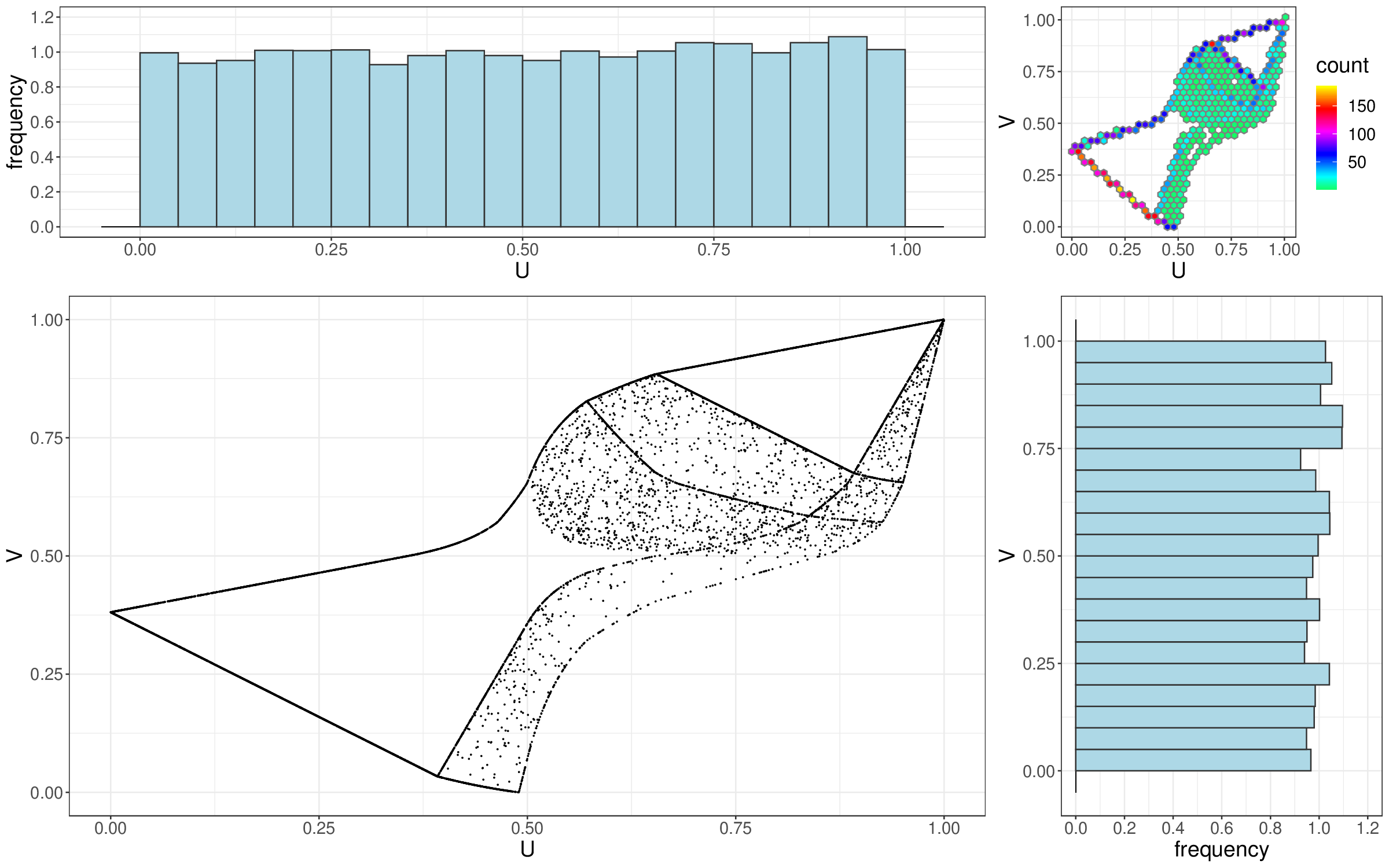}
	\caption{Sample of size 10.000 of the Archimax copula $C_{\psi,A}$ with $\psi$, $A$ being the generator and Pickands dependence function from Example \ref{ex:jumps_markov_kernel_graphs}, its histogram and
the two marginal histograms; sample generated via conditional inverse sampling.}\label{fig:sample_first_ex_archimax_10000}
\end{figure}
\end{Ex}


\section{Kendall distribution function and copula-mass of level sets}\label{sec:kendall_copula_mass}
Throughout this section we let $\gamma \in \mathcal{P}_\mathcal{W}$, $\psi$ be its corresponding generator, $\vartheta \in \mathcal{P}_\mathcal{A}$, $A$ be its corresponding Pickands dependence function and $C_{\psi,A} \in \mathcal{C}_{am}$. The goal of this section is to provide a new proof for the representation of the Kendall distribution function $F_{C_{\psi,A}}^K$ of $C_{\psi,A}$ in terms of $\psi$ and $A$. Building upon this representation, we then calculate the level set mass $\mu_{C_{\psi,A}}(L_{\psi,A}^t)$ for every $t \in [0,1)$ and show that $\mu_{C_{\psi,A}}(L_{\psi,A}^t)$ can be expressed in terms of $\gamma$ and $A$. Prior to proving the afore-mentioned results, we revisit representations of the Kendall distribution functions of Archimedean copulas and EVCs, as these are frequently used throughout this section.
We define the function $\beta_\psi \colon (0,1) \rightarrow (-\infty,0)$ by
\begin{equation}\label{eq:function_beta}
\beta_\psi(t) := D^-\psi(\varphi(t))\varphi(t)
\end{equation}
for every $t \in (0,1)$, set $\beta_\psi(0) := 0$ in the case that $\psi$ is strict, $\beta_\psi(0) := D^-\psi(\varphi(0))\varphi(0)$ in the case that $\psi$ is non-strict and set $\beta_\psi(1) := 0$. Moreover, we observe that, according to \citep{mult_arch, neslehova}, the mass of the $t$-level set $L_\psi^t$ of $C_\psi \in \mathcal{C}_{ar}$ with generator $\psi$ and pseudo-inverse $\varphi$ is given by
\begin{equation}\label{eq:mass_t-lvl_set_archimedean}
\mu_{C_\psi}(L_{\psi}^t) =  \gamma\left(\left\{\frac{1}{\varphi(t)}\right\}\right) = \beta_\psi(t-) - \beta_\psi(t)
\end{equation}
for every $t \in \mathbb{I}$ with the convention that $\beta_\psi(0-) = 0$. Again, following \citep{mult_arch, neslehova}, the Kendall distribution function of $C_\psi$ is given by
\begin{equation}\label{eq:kendall_arch}
F_\psi^K(t) = \gamma\left(\left[0,\frac{1}{\varphi(t)}\right]\right) = t - \beta_\psi(t)
\end{equation}
for every $t \in \mathbb{I}$.
Furthermore, according to \citep{ghoudi1998}, Kendall's tau $\tau_A := \tau(C_A)$ of $C_A$ can be written as
\begin{equation}\label{eq:kendalls_tau_evc}
    \tau_A = \int_{\mathbb{I}}\frac{t(1-t)}{A(t)} \mathrm{d}D^+A(t)
\end{equation}
and the Kendall distribution function $F_A^K := F_{C_A}^K$ is given by
\begin{equation}\label{eq:kendalls_dist_evc}
    F_A^K(t) = t - (1-\tau_A)t \log(t)
\end{equation}
for every $t \in (0,1]$ and $F_A^K(0) = 0$.\\
A representation of the Kendall distribution function $F_{\psi,A}^K$ of $C_{\psi,A}$ has already been provided in \citep{caperaa2000,joe}.
Taking into account that the proof in \citep{caperaa2000} is technically quite involved, we come up with a different proof for the explicit representation of $F_{\psi,A}^K$, relaying on the lower $t$-cut $[C_{\psi,A}]_t$ of $C_{\psi,A}$, eq. \eqref{eq:markov-kernel} and disintegration.
\begin{theorem}\label{thm:upper_t_cut_kendall}
    Let $\psi \in \mathbf{\Psi}$, $A \in \mathcal{A}$, $C_{\psi,A} \in \mathcal{C}_{am}$, $\beta_\psi$ be defined as in eq. \eqref{eq:function_beta}, $\tau_A$ be Kendall's tau of $C_A \in \mathcal{C}_{ev}$ and $F_\psi^K$ be the Kendall distribution function of $C_\psi \in \mathcal{C}_{ar}$. Then the
    Kendall distribution function $F_{\psi,A}^K$ of $C_{\psi,A}$ is given by 
\begin{equation}\label{eq:kendall_archimax}
F_{\psi,A}^K(t) = t + \beta_\psi(t)(\tau_A-1) = F_\psi^K(t) + \beta_\psi(t)\tau_A,
\end{equation}
for every $t \in (0,1]$. Furthermore, $F_{\psi,A}(0) = 0$ holds in the case that $\psi$ is strict and $F_{\psi,A}(0) = \beta_\psi(0)(\tau_A -1)$ holds in the case that $\psi$ is non-strict.
\end{theorem}
\begin{proof}
Let $\psi \in \mathbf{\Psi}$, $A \in \mathcal{A}$ and $C_{\psi,A} \in \mathcal{C}_{am}$. To simplify notation, we set $C := C_{\psi,A}$ throughout this proof.
At first, note that in the case that $L = R$ (whereby $L$, $R$ are defined as in eq. \eqref{eq:definition_L_R}), $C = M$ holds and therefore $F_{\psi,A}^K(t) = F_{M}^K(t) = t$ for every $t \in \mathbb{I}$, as well as $\tau_A = 1$, which proves eq. \eqref{eq:kendall_archimax}.\\
From now on we assume that $L < R$.
Moreover, we assume that $t \in (0,1]$ in the case that $\psi$ is strict and $t \in \mathbb{I}$ in the case that $\psi$ is non-strict.
It is straightforward to prove that the lower $t$-cut of $C$ defined as in eq. \eqref{eq:lower_t_cut} can be written as
$$
[C]_t = \{(x,y) \in [t,1] \times \mathbb{I} \colon y \in [0,\fg^t(x)]\} \cup ([0,t] \times \mathbb{I}),
$$
whereby the function $f^t$ is defined as in eq. \eqref{eq:fct_g_t}. Applying eq. \eqref{eq:markov-kernel} and disintegration yields that
\begin{align}\label{eq:proof_kendall_first}
   \nonumber F_{\psi,A}^K(t) &= \mu_{C}([C]_t) =
   \int_{[0,t]} K_{\psi,A}(x,\mathbb{I}) \mathrm{d}\lambda(x) + \int_{(t,1]} K_{\psi,A}(x,[0,\fg^t(x)]) \mathrm{d}\lambda(x)\\&=
   t + D^-\psi(\varphi(t))\underbrace{\int_{I_t}\frac{G_A\left(h_A^{[-1]}\left(\frac{\varphi(t)}{\varphi(x)}\right)\right)}{D^-\psi(\varphi(x))}\mathrm{d}\lambda(x)}_{=:q},
    \end{align}
    whereby $I_t := \left(t,\psi\left(\frac{L\varphi(t)}{1-L}\right)\right]$. In the last equality of eq. \eqref{eq:proof_kendall_first} we used that $\fg^t(x) = t$ for every $x \in \left[\psi\left(\frac{L\varphi(t)}{1-L}\right),1\right]$.
   Taking into account that the functions $\varphi$ and $h_A^{[-1]}$ are convex and therefore locally Lipschitz continuous, applying change of coordinates for locally Lipschitz continuous functions \citep[Theorem 3]{Hajlasz1993} with $z = \varphi(x)$, $v = \frac{\varphi(t)}{z}$ and $s = h_A^{[-1]}(v)$ implies that
   \begin{align}\label{eq:proof_kendall}
        \nonumber q &= -\int_{[\frac{L\varphi(t)}{1-L},\varphi(t)]}G_A\left(h_A^{[-1]}\left(\frac{\varphi(t)}{z}\right)\right)\mathrm{d}\lambda(z) \\&=
        \nonumber-\varphi(t)\int_{[1,\frac{1-L}{L}]}\frac{G_A\left(h_A^{[-1]}(v)\right)}{v^2}\mathrm{d}\lambda(v) \\& = 
        \nonumber \varphi(t)\int_{[L,R]}\frac{D^+h_A(s)}{h_A(s)^2}G_A(s) \mathrm{d}\lambda(s) \\& =
        \nonumber \varphi(t)\int_{\mathbb{I}}\frac{D^+h_A(s)}{h_A(s)^2}G_A(s) \mathrm{d}\lambda(s)\\& =
        -\varphi(t)\int_{\mathbb{I}}\frac{A(s) - s D^+A(s)}{A(s)^2}G_A(s) \mathrm{d}\lambda(s),
    \end{align}
    whereby the penultimate equality follows form the fact that $D^+h_A(u) = 0$ for every $u \in (R,1]$ and $G_A(u) = 0$ for every $u \in (0,L)$. Furthermore, we used that $h_A(u)= \frac{A(u)}{u}$ for every $u \in (0,1]$ in the last equality.
   Applying \citep[Appendix]{ghoudi1998} yields that the Kendall distribution function $F_A^K$ of $C_A$ can be represented as
   $$
   F_A^K(w) = w - w\log(w)\int_\mathbb{I}\frac{A(s) - sD^+A(s)}{A(s)^2}G_A(s) \mathrm{d}\lambda(s)
   $$
   for every $w \in (0,1]$ and therefore eq. \eqref{eq:kendalls_dist_evc} implies that the following identity holds:
   \begin{align*}
   1 - \tau_A = \int_\mathbb{I}\frac{A(s) - sD^+A(s)}{A(s)^2}G_A(s) \mathrm{d}\lambda(s).
   \end{align*}
   Using this and taking into account eq. \eqref{eq:proof_kendall}, it follows that $q = -\varphi(t)(1-\tau_A)$. Thus, utilizing eq. \eqref{eq:function_beta}, eq. \eqref{eq:proof_kendall_first} simplifies to $F_{\psi,A}(t) = t + \beta_\psi(t)(\tau_A-1)$. Applying eq. \eqref{eq:kendall_arch} then implies eq. \eqref{eq:kendall_archimax}.\\
   If $\psi$ is strict, then obviously $F_{\psi,A}^K(0) = \mu_{C}([C]_0) = \mu_{C}(L_{\psi,A}^0) = 0$ holds. This proves the result.
\end{proof}
The mass of the $t$-level set $L_{\psi,A}^t$ of $C_{\psi,A}$ can now conveniently be calculated using Theorem \ref{thm:upper_t_cut_kendall}. Its mass can be represented in terms of the corresponding Williamson measure $\gamma$.
\begin{Cor}\label{cor:level-set-mass}
    Let $\gamma \in \mathcal{P}_\mathcal{W}$, $\psi$ be its corresponding generator, $\vartheta \in \mathcal{P}_\mathcal{A}$, $A \in \mathcal{A}$ be its Pickands dependence function and let $C_{\psi,A} \in \mathcal{C}_{am}$. Then the mass of the $t$-level set $L_{\psi,A}^t$ is given by
    \begin{equation}\label{eq:t-level-set-mass_archimax}
    \mu_{C_{\psi,A}}(L_{\psi,A}^t) =
    \gamma\left(\left\{\tfrac{1}{\varphi(t)}\right\}\right)(1-\tau_A)
    \end{equation}
    for every $t \in \mathbb{I}$.
\end{Cor}
\begin{proof}
Applying eq. \eqref{eq:mass_t-lvl_set_archimedean} and Theorem \ref{thm:upper_t_cut_kendall}, we obtain that
\begin{align}\label{eq:level_set_mass_archimax}
\nonumber\mu_{C_{\psi,A}}(L_{\psi,A}^t) &= F_{\psi,A}^K(t) - F_{\psi,A}^K(t-) \\&=
\nonumber(\beta_\psi(t-) - \beta_\psi(t))(1-\tau_A) \\&=
\gamma\left(\left\{\frac{1}{\varphi(t)}\right\}\right)(1-\tau_A)
\end{align}
holds for every $t \in \mathbb{I}$ with the convention that $F_{\psi,A}^K(0-) = 0$.
\end{proof}
Considering $\vartheta \in \mathcal{P}_\mathcal{W}$, with Pickands dependence function $A \in \mathcal{A}$ and $C_A \in \mathcal{C}_{ev}$ then, as show in Example \ref{ex:evc_fct_f_t}, for fixed $t \in (0,1)$ the function $\ff^t$ (defined as in eq. \eqref{eq:fct_f}) simplifies to $\ff^t(x) = x^{\frac{1}{t}-1}$ for every $x \in \mathbb{I}$. According to \citep{dietrich2024,evc-mass} the discrete component $\mu_{C_A}^{dis}$ (if non-degenerated) of $\mu_{C_A}$ always concentrates its mass on the graphs $\Gamma(\ff^t)$ for some $t \in (0,1)$. Moreover, the mass of $\Gamma(\ff^t)$ is given by
\begin{equation*}
\mu_{C_A}(\Gamma(\ff^t)) = \frac{2t(1-t)}{A(t)}\vartheta(\{t\}),
\end{equation*}
for every $t \in (0,1)$ \citep{dietrich2024,evc-mass}. The next theorem shows that regardless of the choice of $\gamma \in \mathcal{P}_\mathcal{W}$ and $\vartheta \in \mathcal{P}_\mathcal{A}$ (and therefore regardless of the choice of $A$ and $\psi$), surprisingly, the mass of $\Gamma(\ff^t)$ w.r.t. $\mu_{C_{\psi,A}}$ will also evaluate to $\frac{2t(1-t)}{A(t)}\vartheta(\{t\})$.
\begin{theorem}\label{thm:mass_graph_f_t}
    Let $\gamma \in \mathcal{P}_\mathcal{W}$, $\psi$ be its generator, $\vartheta \in \mathcal{P}_\mathcal{A}$, $A \in \mathcal{A}$ be the corresponding Pickands dependence function, $C_{\psi,A} \in \mathcal{C}_{am}$. Then the following equation holds for every $t \in (0,1)$:
    \begin{equation}\label{eq:mass_graph_f_t}
    \mu_{C_{\psi,A}}(\Gamma(\ff^t)) =
    \frac{2t(1-t)}{A(t)}\vartheta(\{t\}).
    \end{equation}
\end{theorem}
\begin{proof}
Let $t \in (0,1)$ be arbitrary and $J$ be the set of all point masses of $\gamma$. Then proceeding analogously as in Lemma \ref{lem:kernel_jump_on_graph_evc} and fixing $x \in \left(\psi\left(\frac{\varphi(0)}{h_A(t)}\right),1\right) \setminus \bigcup_{s \in J}\{\psi(\frac{1}{sh_A(t)})\}$, $\gamma$ has no point mass in $\frac{1}{\varphi(x)h_A(t)}$ and therefore, utilizing eq. \eqref{eq:rel_deriv_psi_gamma}, $D^-\psi$ is continuous in $\varphi(x) h_A(t)$. Using the fact that $\lambda\left(\bigcup_{s \in J}\{\psi(\frac{1}{sh_A(t)})\}\right) = 0$, applying eq. \eqref{eq:markov-kernel}, disintegration and change of coordinates with $z = \varphi(x)$ and $v = zh_A(t)$, we conclude that
    \begin{align*}
    \mu_{C_{\psi,A}}(\Gamma(\ff^t)) &= 
     \int_{\mathbb{I}} K_{\psi,A}(x, \{\ff^t(x)\}) \mathrm{d}\lambda(x) \\&=
    \int_{\left(\psi\left(\frac{\varphi(0)}{h_A(t)}\right),1\right)} K_{\psi,A}(x, \{\ff^t(x)\}) \mathrm{d}\lambda(x) \\& =
        \int_{\left(\psi\left(\frac{\varphi(0)}{h_A(t)}\right),1\right)\setminus \bigcup_{s \in J}\{\psi(\frac{1}{sh_A(t)})\}} K_{\psi,A}(x, \{\ff^t(x)\}) \mathrm{d}\lambda(x) \\& =
    (1-t)(D^+A(t) - D^+A(t-))\int_{\left[\psi\left(\frac{\varphi(0)}{h_A(t)}\right),1\right]}\frac{D^-\psi(\varphi(x)h_A(t))}{D^-\psi(\varphi(x))} \mathrm{d}\lambda(x) \\& =
    -(1-t)(D^+A(t) - D^+A(t-))\int_{\left[0,\frac{\varphi(0)}{h_A(t)}\right]}D^-\psi(zh_A(t)) \mathrm{d}\lambda(z) \\&=
    -\frac{1-t}{h_A(t)}(D^+A(t) - D^+A(t-))\underbrace{\int_{[0,\varphi(0)]}D^-\psi(v) \mathrm{d}\lambda(v)}_{= -1} \\&=
    \frac{t(1-t)}{A(t)}(D^+A(t) - D^+A(t-))\\&=
    \frac{2t(1-t)}{A(t)}\vartheta(\{t\}),
    \end{align*}
whereby in the second equality we used that $K_{\psi,A}(x,\{g^t(x)\}) = 0$ whenever $x < \psi\left(\frac{\varphi(0)}{h_A(t)}\right)$ and in the last equality we used that, according to eq. \eqref{eq:right_der_pickands_measure}, $D^+A(t) - D^+A(t-) = 2\vartheta(\{t\})$ holds for every $t \in (0,1)$.
\end{proof}
\begin{figure}[H]
	\centering
\includegraphics[width=1\textwidth]{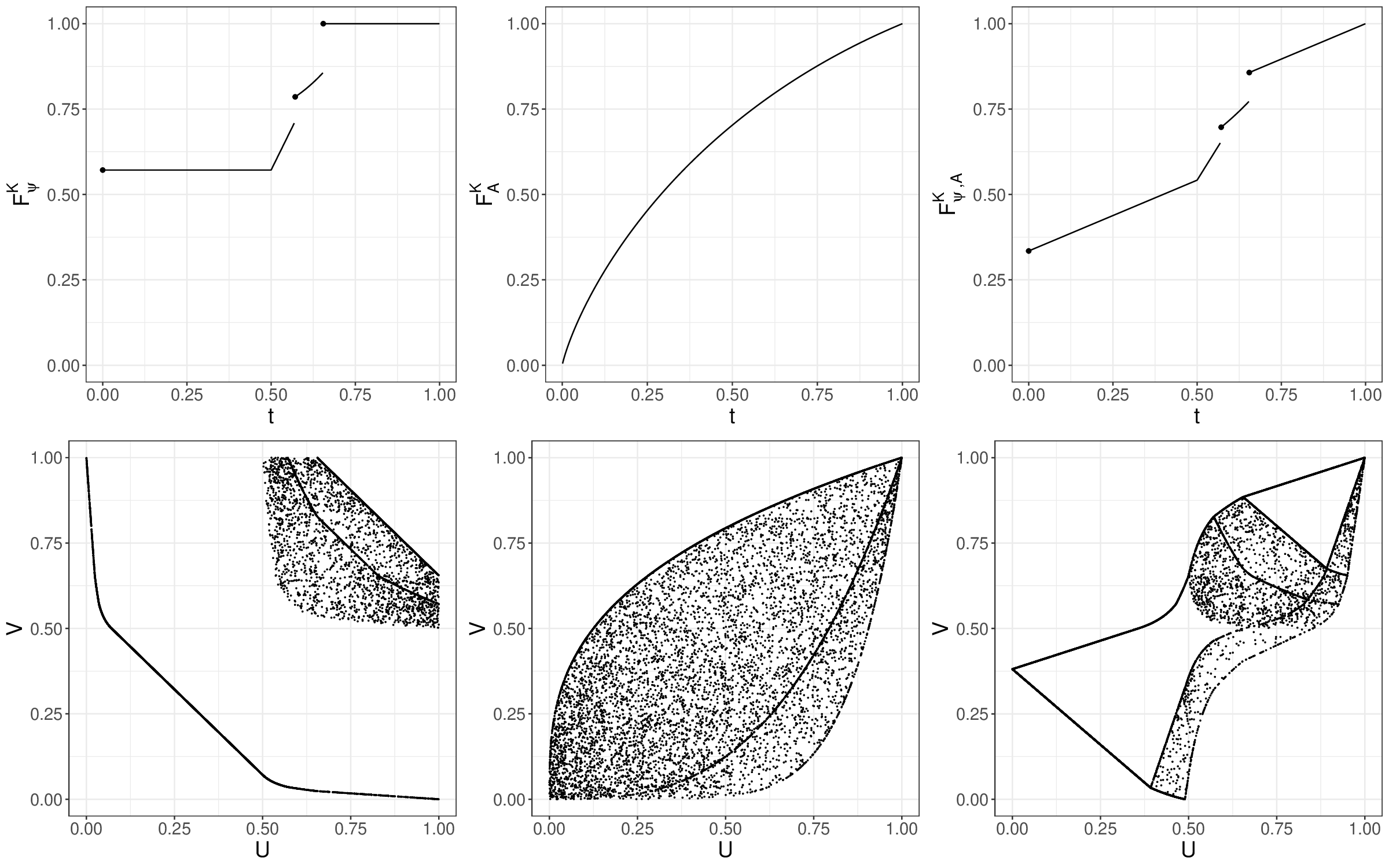}
	\caption{Plots of the Kendall distribution functions $F_\psi^K$ (top left) of $C_\psi$, $F_A^K$ (top middle) of $C_A$ and $F_{\psi,A}^K$ (top right) of $C_{\psi,A}$, as well as the samples of size 10.000 of the copulas $C_\psi$ (bottom left), $C_A$ (bottom middle) and $C_{\psi,A}$ (bottom right) with $\psi \in \mathbf{\Psi}$ defined as in Example \ref{ex:lvl_sets_example} and $A \in \mathcal{A}$ defined according to Example \ref{ex:ex_ha_and_inv}; samples generated via conditional inverse sampling.}\label{fig:first_ex_kendall_samples}
\end{figure}


\section{Absolute continuity and singularity aspects}\label{section:regularity_results}
\subsection{Absolutely continuous, discrete and singular Archimax copulas}
Considering $\gamma \in \mathcal{P}_{\mathcal{W}}$, it was proved in \citep{mult_arch} that absolute continuity, discreteness and singularity of $\gamma$ propagates to the corresponding $C_\gamma \in \mathcal{C}_{ar}$. In this section, the objective is to establish similar results for the larger family of Archimax copulas. To be more precise, considering $\gamma \in \mathcal{P}_\mathcal{W}$, $\vartheta \in \mathcal{P}_\mathcal{A}$ and letting $C$ be the corresponding Archimax copula, we show that absolute continuity, discreteness and singularity of $\gamma$ and $\vartheta$ carry over to the copula $C$.
Let $\psi$ and $A$ be the generator and Pickands dependence function corresponding to $\gamma$ and $\vartheta$, respectively, and let $h_A$ be defined as in eq. \eqref{eq:fct_h}. Then, recall that according to eq. \eqref{eq:function_H_x} for fixed $x \in (0,1)$ the function $H_x \colon \mathbb{I} \rightarrow \mathbb{I}$ is given by
\begin{equation*}
H_x(y) = \frac{\int_{I_y}t \mathrm{d}\gamma(t)}{\int_{I_1}t\mathrm{d}\gamma(t)}
\end{equation*}
with $I_y = \left[0,\frac{1}{\varphi(x)h_A\left(\frac{\varphi(x)}{\varphi(x) + \varphi(y)}\right)}\right]$ for every $y \in \mathbb{I}$. Using convexity of $\psi$ and the properties outlined in Lemma \ref{lem:regularity_fcts}, we deduce that $H_x$ is non-decreasing. Furthermore, the fact that $\gamma$ is a probability measure combined with the continuity of the functions $\varphi$ and $h_A$, ensures right-continuity of $H_x$. Observing that $H_x(1) = 1$, it follows that $H_x$ is a distribution function and consequently, $H_x$ induces a unique probability measure $\nu_{H_x}$.
    Moreover, fixing $x \in (0,1)$ and revisiting eq. \eqref{eq:alternative_markov_kernel}, the Markov kernel $K_{\psi,A}$ of $C_{\psi,A} = C$ is given by
    \begin{equation*}
    K_{\psi,A}(x,[0,y]) = H_x(y)\cdot G_A\left(\frac{\varphi(x)}{\varphi(x) + \varphi(y)}\right)
    \end{equation*}
    for every $y \in \mathbb{I}$. Notice that, using the properties in Lemma \ref{lem:regularity_fcts}, also $G_A$ is a distribution function and therefore corresponds to a unique probability measure $\kappa_{G_A}$ on $\mathbb{I}$.\\
    Establishing the main result of this section, we proceed as follows:
    We prove that absolute continuity/discreteness/singularity of the measures $\gamma$ and $\vartheta$ is preserved by the measures $\nu_{H_x}$ and $\kappa_{G_A}$, respectively. Hence, applying eq. \eqref{eq:alternative_markov_kernel}, these properties also extend to the measure $K_{\psi,A}(x,\cdot)$. Finally, by disintegrating, we obtain absolute continuity, discreteness, or singularity of the copula $C_{\psi,A}$.
Since the proof of the main theorem is technically quite involved, we split it into multiple lemmas.
The next two lemmas show that if the probability measures $\gamma$ and $\vartheta$ are absolutely continuous/discrete/singular, then so are the measures $\nu_{H_x}$ and $\kappa_{G_A}$. For the sake of readability the proofs are deferred to Appendix \ref{sec:appendix_proofs}.
\begin{Lemma}\label{lem:help_regularity_H_psi_A}
Let $\gamma \in \mathcal{P}_\mathcal{W}$, $\psi$ be the corresponding generator, $\varphi$ be its pseudo-inverse, $A \in \mathcal{A}$ and $x \in (0,1)$. Furthermore, let $H_x$ be defined according to eq. \eqref{eq:function_H_x} and $\nu_{H_x}$ be its corresponding probability measure. Then the following assertions hold:
\begin{itemize}
    \item[$(i)$] If $\gamma$ is absolutely continuous, then $\nu_{H_x}$ is absolutely continuous.
    \item[$(ii)$] If $\gamma$ is discrete, then $\nu_{H_x}$ is discrete.
    \item[$(iii)$] If $\gamma$ is singular, then $\nu_{H_x}$ is singular.
    \end{itemize}
\end{Lemma}
\begin{proof}
See Appendix \ref{proof:help_regularity_H_psi_A}.
\end{proof}
\begin{Lemma}\label{lem_help_regularity}
Let $\vartheta \in \mathcal{P}_\mathcal{A}$, $A \in \mathcal{A}$ be its corresponding Pickands dependence function, $G_A$ be defined according to eq. \eqref{eq:funct_f_a} and $\kappa_{G_A}$ be its corresponding probability measure. Then the following assertions hold:
\begin{itemize}
    \item[$(i)$] If $\vartheta$ is absolutely continuous, then $\kappa_{G_A}$ is absolutely continuous.
    \item[$(ii)$] If $\vartheta$ is discrete with $\vartheta(\{0\}) = 0 = \vartheta(\{1\})$, then $\kappa_{G_A}$ is discrete.
    \item[$(iii)$] If $\vartheta$ is singular, then $\kappa_{G_A}$ is singular.
    \end{itemize}
\end{Lemma}
\begin{proof}
See Appendix \ref{proof:reg_G_A}.
\end{proof}
\begin{Rem}
The restriction $\vartheta(\{0\}) = 0 = \vartheta(\{1\})$ in the second assertion of Lemma \ref{lem_help_regularity} is necessary, since the Pickands dependence measure $\vartheta = \frac{1}{2}\delta_0 + \frac{1}{2}\delta_1$ corresponds to the Pickands dependence function $A(t) = 1$ for every $t \in \mathbb{I}$. Therefore, $G_A(t) = 1$ holds for every $t \in \mathbb{I}$, which is obviously an absolutely continuous function on $\mathbb{I}$.
\end{Rem}
The next lemma shows that if $\gamma \in \mathcal{P}_\mathcal{W}$ and $\vartheta \in \mathcal{P}_\mathcal{A}$ are both absolutely continuous then the corresponding Archimax copula $C$ is absolutely continuous as well.
\begin{Lemma}\label{lem:gamma_theta_abs_cont}
Let $\gamma \in \mathcal{P}_\mathcal{W}$, $\vartheta \in \mathcal{P}_\mathcal{A}$ and $C$ be the corresponding Archimax copula. Then the following assertion holds:
$$
\text{If } \gamma \text{ and } \vartheta \text{ are absolutely continuous, then } C \in \mathcal{C}_{am}^{abs}.
$$
\end{Lemma}
\begin{proof}
Let $\gamma \in \mathcal{P}_\mathcal{W}^{abs}$, $\psi$ be its corresponding generator with pseudo-inverse $\varphi$, $\vartheta \in \mathcal{P}_\mathcal{A}^{abs}$, $A$ be its Pickands dependence function, $C_{\psi,A} \in \mathcal{C}_{am}$ and $K_{\psi,A}$ be the Markov kernel of $C_{\psi,A}$ defined as in eq. \eqref{eq:markov-kernel}. We fix $x \in (0,1)$. Applying Lemma \ref{lem_help_regularity}, $G_A$ is absolutely continuous. Furthermore, using that the function $y \mapsto \frac{\varphi(x)}{\varphi(x) + \varphi(y)}$ is obviously absolutely continuous and strictly increasing yields that the function $y \mapsto G_A\left(\frac{\varphi(x)}{\varphi(x) + \varphi(y)}\right)$ is absolutely continuous as well, see \citep[Proposition 129]{pap2002}.\\
Since $\gamma$ is absolutely continuous, Lemma \ref{lem:help_regularity_H_psi_A} yields absolute continuity of the function $H_x$ (defined as in eq. \eqref{eq:function_H_x}) on $\mathbb{I}$.
Using that products of absolutely continuous functions on compact intervals are absolutely continuous and working with eq. \eqref{eq:alternative_markov_kernel}, we conclude that the probability measure $K_{\psi,A}(x,\cdot)$ is absolutely continuous. Disintegration finally implies that $\mu_{C_{\psi,A}}^{abs}(\mathbb{I}^2) = 1$, which proves the result.
\end{proof}
We prove a similar result for the case that $\gamma$ and $\vartheta$ are discrete:
\begin{Lemma}\label{lem:gamma_theta_discrete}
Let $\gamma \in \mathcal{P}_\mathcal{W}$, $\vartheta \in \mathcal{P}_\mathcal{A}$ with $\vartheta(\{0\}) = 0 = \vartheta(\{1\})$ and $C$ be the corresponding Archimax copula. Then the following assertion holds:
$$
\text{If } \gamma \text{ and } \vartheta \text{ are discrete, then } C \in \mathcal{C}_{am}^{dis}.
$$
\end{Lemma}
\begin{proof}
Let $\gamma \in \mathcal{P}_\mathcal{W}^{dis}$, $\psi$ be its corresponding generator with pseudo-inverse $\varphi$, $\vartheta \in \mathcal{P}_\mathcal{A}^{dis}$ with $\vartheta(\{0\}) = 0 = \vartheta(\{1\})$, $A$ be its Pickands dependence function, $C_{\psi,A} \in \mathcal{C}_{am}$ and $K_{\psi,A}$ be the Markov kernel of $C_{\psi,A}$ according to eq. \eqref{eq:markov-kernel}.\\
We show that for fixed $x \in (0,1)$ the probability measure $K_{\psi,A}(x,\cdot)$ is discrete on $\mathbb{I}$. Let $G_A$ be defined as in eq. \eqref{eq:funct_f_a} and $x\in (0,1)$ be arbitrary. We prove that the function $y \mapsto G_A\left(\frac{\varphi(x)}{\varphi(x) + \varphi(y)}\right)$ induces a discrete probability measure on $\mathbb{I}$.
Applying Lemma \ref{lem_help_regularity} yields discreteness of the corresponding probability measure $\kappa_{G_A}$ and thus, the existence of $\alpha_1,\alpha_2,... \in \mathbb{I}$ with $\sum_{i \in \mathbb{N}} \alpha_i = 1$ and $s_1,s_2,... \in (0,1)$ such that $G_A(t) = \sum_{i \in \mathbb{N}} \alpha_i\delta_{s_i}([0,t])$ holds for every $t \in \mathbb{I}$. Note, that in the case that $0<L$ or $R < 1$ holds, according to Lemma \ref{lem:regularity_fcts}, we obtain that $L \leq s_i$ or $s_i \leq R$. Using the fact that $s_i \in \left[0,\frac{\varphi(x)}{\varphi(x) + \varphi(y)}\right]$ is equivalent to $y \geq \ff^{s_i}(x)$ for every $y \in [\ff^L(x),\ff^R(x)]$ and $i \in \mathbb{N}$, implies that
$$
G_A\left(\frac{\varphi(x)}{\varphi(x) + \varphi(y)}\right) = \sum_{i \in \mathbb{N}}\alpha_i \mathbf{1}_{[\ff^{s_i}(x),1]}(y)
$$
for every $y \in [\ff^L(x),\ff^R(x)]$. Furthermore, using Lemma \ref{lem:regularity_fcts} again, $G_A\left(\frac{\varphi(x)}{\varphi(x) + \varphi(y)}\right) = 0$ holds for every $y \in [0,\ff^L(x))$ and $G_A\left(\frac{\varphi(x)}{\varphi(x) + \varphi(y)}\right) = 1$ holds for every $y \in [\ff^R(x),1]$. Overall we obtain that the distribution function $y \mapsto G_A\left(\frac{\varphi(x)}{\varphi(x) + \varphi(y)}\right)$ induces a unique, discrete probability measure on $\mathbb{I}$ with point masses $\alpha_i$ in $\ff^{s_i}(x)$.\\
Considering $H_x$ defined as in eq. \eqref{eq:function_H_x} and applying Lemma \ref{lem:help_regularity_H_psi_A}, it follows that $H_x$ induces a discrete probability measure $\nu_{H_x}$ on $\mathbb{I}$. Finally, utilizing that products of discrete distribution functions are discrete and using eq. \eqref{eq:alternative_markov_kernel}, we conclude that $K_{\psi,A}(x,\cdot)$ is a discrete probability measure on $\mathbb{I}$. Consequently, disintegration yields $C \in \mathcal{C}_{am}^{dis}$.
\end{proof}
At last, we prove that if $\gamma$ and $\vartheta$ are singular then the corresponding Archimax copula $C$ is singular.
\begin{Lemma}\label{lem:gamma_theta_singular}
Let $\gamma \in \mathcal{P}_\mathcal{W}$, $\vartheta \in \mathcal{P}_\mathcal{A}$ and $C$ be the corresponding Archimax copula. Then the following assertion holds:
$$
\text{If } \gamma \text{ and } \vartheta \text{ are singular, then } C\in\mathcal{C}_{am}^{sing}.
$$
\end{Lemma}
\begin{proof}
Let $\gamma \in \mathcal{P}_\mathcal{W}^{sing}$, $\psi$ be its corresponding generator with pseudo-inverse $\varphi$, $\vartheta \in \mathcal{P}_\mathcal{A}^{sing}$, $A$ be its Pickands dependence function, $C_{\psi,A} \in \mathcal{C}_{am}$ and $K_{\psi,A}$ be its Markov kernel according to eq. \eqref{eq:markov-kernel}.\\
We prove that for fixed $x \in (0,1)$ the probability measure $K_{\psi,A}(x,\cdot)$ is singular on $\mathbb{I}$. Let $G_A$ be defined as in eq. \eqref{eq:funct_f_a} and $x \in (0,1)$ be arbitrary. We prove that the function $y \mapsto G_A\left(\frac{\varphi(x)}{\varphi(x) + \varphi(y)}\right)$ is singular.
Applying Lemma \ref{lem_help_regularity} yields that $G_A$ is singular and therefore defining the set
$$
\Lambda := \{s \in \mathbb{I} \colon G_A'(s) \text{ exists and } G_A'(s) = 0\},
$$
we conclude that $\lambda(\Lambda) = 1$. Denoting by $\iota$ the inverse function of $y \mapsto \frac{\varphi(x)}{\varphi(x) + \varphi(y)}$ and defining the set
$$
\Upsilon := \left\{y \in \mathbb{I} \colon \frac{\varphi(x)}{\varphi(x) + \varphi(y)} \in \Lambda\right\},
$$
we obtain that $\iota(z) = \psi\left(\left(\frac{1}{z}-1\right)\varphi(x)\right) = \ff^z(x)$ for every $z \in \left[\frac{\varphi(x)}{\varphi(x) + \varphi(0)},1\right]$ as well as $\Upsilon = \iota(\Lambda)$. Using that $\iota$ is locally absolutely continuous and applying \citep[Theorem 3.41]{leoni2024} yields that $\lambda(\Upsilon) = 1$. This implies that
$$
\frac{\partial}{\partial y}G_A\left(\frac{\varphi(x)}{\varphi(x) + \varphi(y)}\right) = G_A'\left(\frac{\varphi(x)}{\varphi(x) + \varphi(y)}\right)\frac{-\varphi(x)\varphi'(y)}{(\varphi(x) + \varphi(y))^2} = 0
$$
for every $y \in \Upsilon$. Observing that the function $y \mapsto G_A\left(\frac{\varphi(x)}{\varphi(x) + \varphi(y)}\right)$ is obviously continuous, we obtain that $y \mapsto G_A\left(\frac{\varphi(x)}{\varphi(x) + \varphi(y)}\right)$ induces a singular probability measure on $\mathbb{I}$.\\
Applying Lemma \ref{lem:help_regularity_H_psi_A}, we conclude that the function $H_x$ induces a singular probability measure on $\mathbb{I}$ as well. Using that products of singular distribution functions are singular and working with eq. \eqref{eq:alternative_markov_kernel}, we finally deduce that the probability measure $K_{\psi,A}(x,\cdot)$ is singular. Now disintegration yields the desired result.
\end{proof}
The next theorem summarizes Lemma \ref{lem:gamma_theta_abs_cont}, Lemma \ref{lem:gamma_theta_discrete} and Lemma \ref{lem:gamma_theta_singular}.
\begin{theorem}\label{thm:regularity_measures_copula}
Let $\gamma \in \mathcal{P}_\mathcal{W}$, $\vartheta \in \mathcal{P}_\mathcal{A}$ and $C$ be the corresponding Archimax copula. Then the following assertions hold:
    \begin{itemize}
        \item[$(i)$] If $\gamma$ and $\vartheta$ are absolutely continuous, then $C\in\mathcal{C}_{am}^{abs}$.
        \item[$(ii)$] If $\gamma$ is discrete and $\vartheta$ is discrete with $\vartheta(\{0\}) = 0 = \vartheta(\{1\})$, then $C\in\mathcal{C}_{am}^{dis}$.
        \item[$(iii)$] If $\gamma$ and $\vartheta$ are singular, then $C\in\mathcal{C}_{am}^{sing}$.
    \end{itemize}
\end{theorem}
Next, we give an example of a discrete Archimax copula.
\begin{Ex}\label{ex:discrete_archimax}
Considering the probability measure $\gamma$ defined by
$$
\gamma := \frac{1}{10}\delta_\frac{1}{11} + \frac{1}{5}\delta_\frac{1}{7} + \frac{3}{10}\delta_\frac{1}{6} + \frac{1}{5}\delta_\frac{1}{5} + \frac{1}{5}\delta_\frac{1}{2},
$$
clearly $\gamma \in \mathcal{P}_\mathcal{W}$. Furthermore, defining the probability measure $\vartheta$ by
$$
\vartheta := \frac{1}{2}\delta_\frac{1}{4} + \frac{1}{2}\delta_\frac{3}{4},
$$
it is straightforward to see that $\vartheta \in \mathcal{P}_\mathcal{A}$. Obviously, $\gamma$ and $\vartheta$ are discrete and $\vartheta(\{0\}) = 0 = \vartheta(\{1\})$ holds. Therefore, Theorem \ref{thm:regularity_measures_copula} yields that the corresponding Archimax copula $C_{\gamma,\vartheta}$ fulfills $\mu_{C_{\gamma,\vartheta}}^{dis}(\mathbb{I}^2) = 1$. A sample of $C_{\gamma,\vartheta}$ is depicted in Figure \ref{fig:sample_discrete_archimax}.
\begin{figure}[H]
	\centering
\includegraphics[width=1\textwidth]{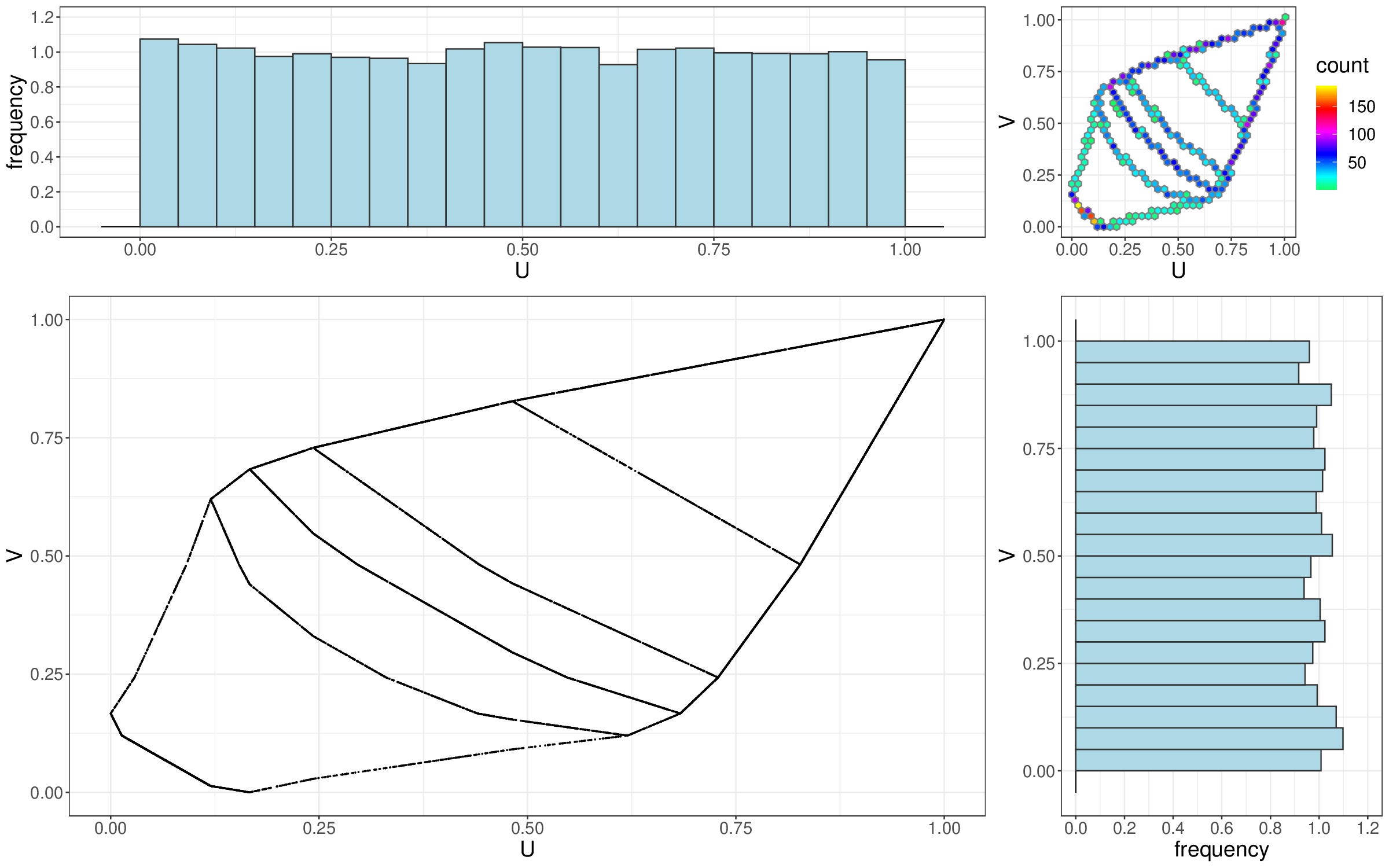}
	\caption{Sample of size 10.000 of the Archimax copula $C_{\gamma,\vartheta}$ with $\gamma$, $\vartheta$ being the Williamson measure and Pickands dependence measure from Example \ref{ex:discrete_archimax}, its histogram and
the two marginal histograms; sample generated via conditional inverse sampling.}\label{fig:sample_discrete_archimax}
\end{figure}
\end{Ex}

\subsection{Support of absolutely continuous, discrete and singular components of Archimax copulas}
In this subsection we generalize the result in  \citep[Corollary 5]{evc-mass}, saying that for $C \in \mathcal{C}_{ev}$ (and if $L<R$), the support of the doubly stochastic measure $\mu_C$ of $C$ coincides with the support of the absolutely continuous component $\mu_C^{abs}$ of $\mu_C$. To be more precise, we establish that if $L<R$ and the Williamson measure $\gamma$ satisfies $\mathrm{supp}(\gamma) = \left[\frac{1}{\varphi(0)}, \infty\right)$, and is either absolutely continuous, discrete or singular, then the support of the measure $\mu_C$ corresponding to $C \in \mathcal{C}_{am}$ coincides with the support of $\mu_C^{abs}$, $\mu_C^{dis}$ or $\mu_C^{sing}$, respectively. The prove of the afore-mentioned result is split into three lemmas.
\begin{Lemma}\label{lem:support_abs}
Let $\gamma \in \mathcal{P}_\mathcal{W}^{abs}$ with $\mathrm{supp}(\gamma) = \left[\frac{1}{\varphi(0)},\infty\right)$, $\vartheta \in \mathcal{P}_\mathcal{A}$ and $C$ be the corresponding Archimax copula. Furthermore, let $L$, $R$ be defined according to eq. \eqref{eq:definition_L_R_measure} with $L < R$.
Then,
$$
\mathrm{supp}(\mu_{C}^{abs}) = \mathrm{supp}(\mu_{C}).
$$
\end{Lemma}
\begin{proof}
Let $\gamma \in \mathcal{P}_\mathcal{W}$, $\psi$ be the corresponding generator, $\vartheta \in \mathcal{P}_\mathcal{A}$, $A$ be its Pickands dependence function, $C_{\psi,A} \in \mathcal{C}_{am}$ and $K_{\psi,A}$ be the Markov kernel of $C_{\psi,A}$ defined in eq. \eqref{eq:markov-kernel}. We fix $x \in (0,1)$ and prove that the absolutely continuous component $K_{\psi,A}^{abs}(x,\cdot)$ of $K_{\psi,A}(x,\cdot)$ fulfills $\mathrm{supp}(K_{\psi,A}^{abs}(x,\cdot)) = [\max\{\fg^0(x),\ff^L(x)\},\ff^R(x)]$. The desired result then follows using disintegration.\\
Let $(y_1,y_2) \subseteq (\max\{\fg^0(x),\ff^L(x)\},\ff^R(x))$ be an arbitrary open interval. We prove that $K_{\psi,A}^{abs}(x,(y_1,y_2)) > 0$. Applying Lemma \ref{lem:help_regularity_H_psi_A} and proceeding analogously to the proof of Lemma \ref{thm:arch_full_support_archimax}, we obtain that the function $H_x$ (defined in eq. \eqref{eq:function_H_x}) is strictly increasing and absolutely continuous and therefore there exists a set $\Lambda_1 \in \mathcal{B}(\mathbb{I})$ with $\lambda((y_1,y_2) \cap \Lambda_1) > 0$ such that $H_x'(s) > 0$ for every $s \in (y_1,y_2) \cap \Lambda_1$. Using monotonicity and convexity implies that the set $\Lambda_2 \in \mathcal{B}(\mathbb{I})$ of all points $s \in \mathbb{I}$ at which the functions $G_A$ and $\varphi$ are differentiable fulfills $\lambda(\Lambda_2) = 1$. Considering $s \in (y_1,y_2) \cap \Lambda_1 \cap \Lambda_2$, calculating the derivative $\partial_2K_{\psi,A}(x,[0,s])$ and simplifying the result yields 
$$
\mathcal{k}_{\psi,A}(x,s) := \underbrace{H_x'(s)G_A\left(\frac{\varphi(x)}{\varphi(x) + \varphi(s)}\right)}_{=: z_1} + \underbrace{H_x(s)G_A'\left(\frac{\varphi(x)}{\varphi(x) + \varphi(s)}\right)\frac{-\varphi(x)\varphi'(s)}{(\varphi(x) + \varphi(s))^2}}_{=: z_2}.
$$
Obviously, all the factors in $z_2$ are non-negative. Observing that $G_A\left(\frac{\varphi(x)}{\varphi(x) + \varphi(s)}\right) > 0$ (see Lemma \ref{lem:regularity_fcts}) and using that $H_x'(s) > 0$, we have that $z_1 > 0$ and therefore $\mathcal{k}_{\psi,A}(x,s)>0$ holds. Furthermore, using the fact that $\lambda( (y_1,y_2) \cap \Lambda_1 \cap \Lambda_2) > 0$, we obtain that $K_{\psi,A}^{abs}(x,(y_1,y_2)) > 0$. Finally, taking into account that $(y_1,y_2)$ was arbitrary and using that $K_{\psi,A}(x,[0,s]) = 0$ for every $s \in [0,\max\{\fg^0(x),\ff^L(x)\})$ as well as $K_{\psi,A}(x,[0,s]) = 1$ for every $s \in (\ff^R(x),1]$, it follows that 
$$
\mathrm{supp}(K_{\psi,A}^{abs}(x,\cdot)) = [\max\{\fg^0(x),\ff^L(x)\},\ff^R(x)].
$$
\end{proof}
We prove that if the Williamson measure $\gamma$ is discrete and $\mathrm{supp}(\gamma) = \left[\frac{1}{\varphi(0)},\infty\right)$, then the support of the discrete component $\mu_C^{dis}$ of the doubly stochastic measure $\mu_C$ coincides with the support of $\mu_C$.
\begin{Lemma}\label{lem:support_dis}
Let $\gamma \in \mathcal{P}_\mathcal{W}^{dis}$ with $\mathrm{supp}(\gamma) = \left[\frac{1}{\varphi(0)},\infty\right)$, $\vartheta \in \mathcal{P}_\mathcal{A}$ and $C$ be the corresponding Archimax copula.
Then,
$$
\mathrm{supp}(\mu_{C}^{dis}) = \mathrm{supp}(\mu_{C}).
$$
\end{Lemma}
\begin{proof}
Let $\gamma \in \mathcal{P}_\mathcal{W}^{dis}$, $\psi$ be its generator, $\vartheta \in \mathcal{P}_\mathcal{A}$, $A$ be the corresponding Pickands dependence function, $C_{\psi,A} \in \mathcal{C}_{am}$ and $K_{\psi,A}$ be the Markov kernel of $C_{\psi,A}$ defined as in eq. \eqref{eq:markov-kernel}. Let $L$, $R$ be defined as in eq. \eqref{eq:definition_L_R}. If $L=R$, then $C_{\psi,A} = M$ and therefore the statement of the Lemma holds.\\
We now assume that $L < R$. Having  $\gamma \in \mathcal{P}_\mathcal{W}^{dis}$ and using that $\mathrm{supp}(\gamma) = \left[\frac{1}{\varphi(0)},\infty\right)$, there exists a countable dense set of $s_1,s_2, ... \in \left[\frac{1}{\varphi(0)}, \infty\right)$ with $\gamma(\{s_i\}) > 0$ for every $i \in \mathbb{N}$. Note, that in the case of $\gamma$ being strict, $s_i > 0$ holds for every $i \in \mathbb{N}$. Applying the fact that $\varphi$ is strictly decreasing, we obtain that for every $s_i$ there exists a $t_i \in [0,1)$ such that $s_i = \frac{1}{\varphi(t_i)}$ and therefore $t_i = \psi(\frac{1}{s_i})$. Denoting by $I$ the set of all point masses of $\vartheta$, we
fix $x \in (0,1) \setminus \bigcup_{i \in \mathbb{N}}\bigcup_{r \in I}\left\{\psi\left(\frac{\varphi(t_i)}{h_A(r)}\right)\right\}$.
The fact that the points $\bigcup_{i \in \mathbb{N}}\bigcup_{r \in I}\left\{\psi\left(\frac{\varphi(t_i)}{h_A(r)}\right)\right\}$ are excluded from the interval $(0,1)$, ensures that Lemma \ref{lem:kernel_jump_on_graph_arch} can be applied. Similar to the proof of Lemma \ref{lem:support_abs}, we show that $\mathrm{supp}(K_{\psi,A}^{dis}(x,\cdot)) = [\max\{\fg^0(x),\ff^L(x)\},\ff^R(x)]$. Obviously, it suffices to prove that $K_{\psi,A}(x,\{y\}) > 0$ for a dense set of $y \in (\max\{\fg^0(x),\ff^L(x)\},\ff^R(x))$. Using denseness of $\{s_1, s_2, ...\}$ in $\left[\frac{1}{\varphi(0)}, \infty\right)$, it is straightforward to see that the set of $t_i$, fulfilling $\psi\left(\frac{(1-L)\varphi(x)}{L}\right) < t_i < x$, is dense in $\left(\psi\left(\frac{(1-L)\varphi(x)}{L}\right), x\right)$. Moreover, applying the properties in Lemma \ref{lem:regularity_fcts}, the function $z \mapsto \fg^z(x)$ (defined as in eq. \eqref{eq:fct_g_t}) is strictly increasing on $\left(\psi\left(\frac{(1-L)\varphi(x)}{L}\right), x\right)$. Combining the afore-mentioned findings, we therefore conclude that the set
$$
\Upsilon_x := \left\{\fg^{t_i}(x) \colon \psi\left(\frac{(1-L)\varphi(x)}{L}\right) \leq t_i < x, \, i\in \mathbb{N}\right\}
$$
is dense in $(\max\{\fg^0(x),\ff^L(x)\},\ff^R(x))$.\\
Furthermore, noting that the inequality $\psi\left(\frac{(1-L)\varphi(x)}{L}\right) < t_i < x$ implies $x \in \left(t_i,\psi\left(\frac{L\varphi(t_i)}{1-L}\right)\right)$ and applying Lemma \ref{lem:kernel_jump_on_graph_arch} yields that $K_{\psi,A}(x,\{y\}) > 0$ for every $y \in \Upsilon_x$. Combined with the fact that $K_{\psi,A}(x,[0,\max\{\fg^0(x),\ff^L(x)\})\cup (\ff^R(x),1]) = 0$, this leads to the conclusion that
$$
\mathrm{supp}(K_{\psi,A}^{dis}(x,\cdot)) = [\max\{\fg^0(x),\ff^L(x)\},\ff^R(x)].
$$
Finally, the result follows using $\lambda\left(\bigcup_{i \in \mathbb{N}}\bigcup_{r \in I}\left\{\psi\left(\frac{\varphi(t_i)}{h_A(r)}\right)\right\}\right) = 0$ and applying disintegration.
\end{proof}
\begin{Lemma}\label{lem:support_sing}
Let $\gamma \in \mathcal{P}_\mathcal{W}^{sing}$ with $\mathrm{supp}(\gamma) = \left[\frac{1}{\varphi(0)},\infty\right)$, $\vartheta \in \mathcal{P}_\mathcal{A}$ and $C$ be the corresponding Archimax copula. Furthermore, let $L$, $R$ be defined as in eq. \eqref{eq:definition_L_R_measure} with $L < R$. Then,
$$
\mathrm{supp}(\mu_{C}^{sing}) = \mathrm{supp}(\mu_{C}).
$$
\end{Lemma}
\begin{proof}
Let $\gamma \in \mathcal{P}_\mathcal{W}^{sing}$, $\psi$ be its generator, $\vartheta \in \mathcal{P}_\mathcal{A}$, $A$ be the corresponding Pickands dependence function, $C_{\psi,A} \in \mathcal{C}_{am}$ and $K_{\psi,A}$ be the Markov kernel of $C_{\psi,A}$ defined as in eq. \eqref{eq:markov-kernel}. 
 We fix $x \in (0,1)$ and prove that $\mathrm{supp}(K_{\psi,A}^{sing}(x,\cdot)) = [\max\{\fg^0(x),\ff^L(x)\},\ff^R(x)]$. Applying Lemma \ref{lem:help_regularity_H_psi_A} and proceeding as in the proof of Lemma \ref{thm:arch_full_support_archimax}, we obtain that the measure generating function $H_x$ defined as in eq. \eqref{eq:function_H_x} is strictly increasing and singular on the interval $[\max\{\fg^0(x),\ff^L(x)\},\ff^R(x)]$. Furthermore, using that the measure generating function $y \mapsto G_A\left(\frac{\varphi(x)}{\varphi(x) + \varphi(y)}\right)$ is positive on $[\max\{\fg^0(x),\ff^L(x)\},\ff^R(x)]$ and utilizing eq. \eqref{eq:alternative_markov_kernel}, applying Lemma \ref{lem:help_singular_comp_full_support} and the fact that $K_{\psi,A}(x,[0,\max\{\fg^0(x),\ff^L(x)\})\cup (\ff^R(x),1)) = 0$ yields that
$$
\mathrm{supp}(K_{\psi,A}^{sing}(x,\cdot)) = [\max\{\fg^0(x),\ff^L(x)\},\ff^R(x)].
$$
Disintegration proves the desired result.
\end{proof}
The next theorem summarizes Lemma \ref{lem:support_abs}, Lemma \ref{lem:support_dis} and Lemma \ref{lem:support_sing}.
\begin{theorem}\label{thrm:support_components}
Let $\gamma \in \mathcal{P}_\mathcal{W}$ with $\mathrm{supp}(\gamma) = \left[\frac{1}{\varphi(0)},\infty\right)$, $\vartheta\in\mathcal{P}_\mathcal{A}$ and $C$ be the corresponding Archimax copula. Moreover, let $L$, $R$ be defined as in eq. \eqref{eq:definition_L_R_measure} with $L < R$. Then the following assertions hold:
\begin{itemize}
    \item[$(i)$] If $\gamma \in \mathcal{P}_{\mathcal{W}}^{abs}$, then $\mathrm{supp}(\mu_{C}^{abs}) = \mathrm{supp}(\mu_{C}).$
    \item[$(ii)$] If $\gamma \in \mathcal{P}_{\mathcal{W}}^{dis}$, then $\mathrm{supp}(\mu_{C}^{dis}) = \mathrm{supp}(\mu_{C}).$
    \item[$(iii)$] If $\gamma \in \mathcal{P}_{\mathcal{W}}^{sing}$, then $\mathrm{supp}(\mu_{C}^{sing}) = \mathrm{supp}(\mu_{C}).$
\end{itemize}
\end{theorem}
The following result for EVCs \citep[Corollary 5]{evc-mass} is an immediate consequence of Theorem \ref{thrm:support_components}.
\begin{Cor}[Trutschnig et al. (2016)]
Let $\vartheta \in \mathcal{P}_\mathcal{A}$, $A$ be its corresponding Pickands dependence function, $C_{A} \in \mathcal{C}_{ev}$ and $L$, $R$ be defined as in eq. \eqref{eq:definition_L_R} with $L<R$. Then,
$$
\mathrm{supp}(\mu_{C_A}^{abs}) = \mathrm{supp}(\mu_{C_A}) = \{(x,y)\in \mathbb{I}^2\colon \ff^L(x) \leq y \leq  \ff^R(x)\}
$$
with $\ff^L(x) = x^{\frac{1}{L}-1}$ and $\ff^R(x) = x^{\frac{1}{R}-1}$ for every $x \in \mathbb{I}$.
\end{Cor}
\begin{proof}
Defining $\psi(z) := \exp(-z)$ for every $z \in [0,\infty)$, the corresponding $\gamma \in \mathcal{P}_\mathcal{W}$ is given by $\gamma([0,z]) = \exp(-\frac{1}{z}) + \exp(-\frac{1}{z})\frac{1}{z}$ for every $z \in [0,\infty)$. It is straightforward to see that $\gamma$ has full support and is absolutely continuous. Having that $C_A = C_{\psi,A} \in \mathcal{C}_{am}$ and applying Corollary \ref{cor:mai_scherer_supp_evc} as well as Theorem \ref{thrm:support_components}, proves the result.
\end{proof}


\section{Appendix}
  \subsection{Auxiliary lemmas}
 We prove the inverse function rule and the chain rule for left-hand derivatives. The proofs are similar to the proofs for classical derivatives. The Lemma is provided for the sake of completeness.
\begin{Lemma}\label{lem:properties_left_hand_der}
  Let $I_1,I_2,I_3$ be open sub-intervals of $\mathbb{R}$ and let $\iota \colon I_1 \rightarrow I_2$ be a function whose left-hand derivative $D^-\iota$ exists. Moreover, let $\eta \colon I_2 \rightarrow I_3$ be an arbitrary function. Then the following assertions hold:
      \begin{itemize}
          \item[$(i)$] If $\eta$ is strictly decreasing and convex, then
          $$
          D^-\eta^{-1}(z) = \frac{1}{D^+\eta(\eta^{-1}(z))}
          $$
          holds for every $z \in I_3$, whereby $\eta^{-1}$ denotes the inverse function of $\eta$.
          \item[$(ii)$] If $D^-\eta$ exists and $\iota$ is a non-decreasing function, then
          $$
          D^-(\eta\circ\iota)(t) = D^-\eta(\iota(t)) \cdot D^-\iota(t)
          $$
          holds for every $t \in I_1$.
        \item[$(iii)$] If $D^-\eta$ exists and $\iota$ is a non-increasing function, then
          $$
          D^-(\eta\circ\iota)(t) = D^+\eta(\iota(t)) \cdot D^-\iota(t)
          $$
          holds for every $t \in I_1$.
      \end{itemize}
  \end{Lemma}
  \begin{proof}
      We prove the first assertion, fix $z \in I_3$ and $\varepsilon < 0$ with $z + \varepsilon \in I_3$ and define $\delta := \eta^{-1}(z+\varepsilon) - \eta^{-1}(z)$. Since $\eta^{-1}$ is strictly decreasing as the inverse of a strictly decreasing function, $\delta > 0$ follows. We obtain that
      $$
      \frac{\eta^{-1}(z+\varepsilon) - \eta^{-1}(z)}{\varepsilon} = \frac{1}{\frac{\eta(\eta^{-1}(z+\varepsilon)) - \eta(\eta^{-1}(z)))}{\eta^{-1}(z+\varepsilon) - \eta^{-1}(z)}} = \frac{1}{\frac{\eta(\eta^{-1}(z) + \delta) - \eta(\eta^{-1}(z))}{\delta}}
      $$
      and therefore using the fact that $\eta$ is convex yields that
      \begin{align*}
      D^-\eta^{-1}(z) &= \lim_{\varepsilon \uparrow 0}\frac{\eta^{-1}(z+\varepsilon) - \eta^{-1}(z)}{\varepsilon} = \lim_{\varepsilon \uparrow 0}\frac{1}{\frac{\eta(\eta^{-1}(z+\varepsilon)) - \eta(\eta^{-1}(z)))}{\eta^{-1}(z+\varepsilon) - \eta^{-1}(z)}} \\&= \lim_{\delta \downarrow 0}\frac{1}{\frac{\eta(\eta^{-1}(z) + \delta) - \eta(\eta^{-1}(z))}{\delta}}= \frac{1}{D^+\eta(\eta^{-1}(z))},
      \end{align*}
      for every $z \in I_3$, whereby the denominator in the last expression is unequal to zero, since $\eta$ is strictly decreasing and convex.\\
      We prove the second assertion, consider  $s,t \in I_1$ with $s<t$ and assume for now that $\iota(s) < \iota(t)$. We define $r := \iota(s) - \iota(t) < 0$ and obtain that
      $$
      \frac{\eta(\iota(s)) - \eta(\iota(t))}{\iota(s) - \iota(t)} = \frac{\eta(\iota(t) + r) - \eta(\iota(t))}{r},
      $$
      which together with the assumption that $\iota(s) < \iota(t)$ for every $s$ near $t$ yields that
      $$
      \lim_{s \uparrow t}\frac{\eta(\iota(s)) - \eta(\iota(t))}{\iota(s) - \iota(t)} = \lim_{r \uparrow 0}\frac{\eta(\iota(t) + r) - \eta(\iota(t))}{r} = D^-\eta(\iota(t)).
      $$
      Considering the general case, we define the function
      $$
      w(y) := \begin{cases}
          \frac{\eta(y) - \eta(\iota(t))}{y-\iota(t)}, &\text{ if } y \neq \iota(t),\\
          D^-\eta(\iota(t)), &\text{ if } y = \iota(t)
      \end{cases}
      $$
      and observe that obviously $\frac{\eta(\iota(s)) - \eta(\iota(t))}{s-t} = w(\iota(s)) \cdot \frac{\iota(s) - \iota(t)}{s-t}$ holds. Taking the limit and using the fact that the left-hand derivative $D^-\iota$ of $\iota$ exists finally yields that
      $$
      \lim_{s \uparrow t}\frac{\eta(\iota(s)) - \eta(\iota(t))}{s-t} = \lim_{s \uparrow t} w(\iota(s)) \cdot \lim_{s \uparrow t}\frac{\iota(s) - \iota(t)}{s-t} = D^-\eta(\iota(t)) \cdot D^-\iota(t).
      $$
      The proof of the third assertion follows similarly.
  \end{proof}
Proving Theorem \ref{thm:convexity_level_curves}, the next lemma will be essential.
\begin{Lemma}\label{lem:auxiliary_conv_level_set}
Let $A \in \mathcal{A}$ and $R\in [\frac{1}{2},1]$ be defined according to eq. \eqref{eq:definition_L_R}. Then the function 
$$
s \mapsto \frac{D^+A(s)}{D^+A(s)s - A(s)}
$$ 
is non-increasing on $(0,R)$.
\end{Lemma}
\begin{proof}
The fact that $D^+A(s) \leq 1 < \frac{A(s)}{s}$ holds for every $s \in (0,R)$ implies that $D^+A(s)s - A(s) < 0$  for every $s \in (0,R)$. We fix $s_1,s_2 \in (0,R)$ with $s_1 < s_2$ and prove that
     \begin{align*}
     &\frac{D^+A(s_2)}{D^+A(s_2)s_2 - A(s_2)} - \frac{D^+A(s_1)}{D^+A(s_1)s_1 - A(s_1)} \\&= \frac{D^+A(s_2)(D^+A(s_1)s_1 - A(s_1)) - D^+A(s_1)(D^+A(s_2)s_2 - A(s_2))}{(D^+A(s_1)s_1 - A(s_1))(D^+A(s_2)s_2 - A(s_2))} \leq 0.
         \end{align*}
     Since $(D^+A(s_1)s_1 - A(s_1))(D^+A(s_2)s_2 - A(s_2)) > 0$, it suffices to show that
     $$
     Q(s_1,s_2) := D^+A(s_2)(D^+A(s_1)s_1 - A(s_1)) - D^+A(s_1)(D^+A(s_2)s_2 - A(s_2)) \leq 0.
     $$
     We distinguish two cases:
     \begin{itemize}
         \item[$(a)$] At first, we consider the case that $D^+A(s_2)\geq 0$ and obtain using convexity of $A$ and $D^+A(s_1) \leq D^+A(s_2)$ that
         \begin{align*}
             Q(s_1,s_2) &\leq D^+A(s_2)[D^+A(s_1)s_1 - D^+A(s_2)s_2 + A(s_2) - A(s_1)] \\&\leq
             D^+A(s_2) [D^+A(s_2)(s_1-s_2) + A(s_2) - A(s_1)] \\& \leq
             D^+A(s_2) [A(s_1) - A(s_1)] = 0.
         \end{align*}
        \item[$(b)$] If $D^+A(s_2)\leq 0$, then using similar arguments as before yields that
        \begin{align*}
             Q(s_1,s_2) &= (-D^+A(s_2))[D^+A(s_1)(s_2-s_1) + A(s_1)] + D^+A(s_1)A(s_2) \\&\leq
             -D^+A(s_2)A(s_2) + D^+A(s_1)A(s_2) \\& =
             A(s_2)(D^+A(s_1) - D^+A(s_2)) \leq 0.
         \end{align*}
        \end{itemize}
         In either case $Q(s_1,s_2) \leq 0$ holds and therefore the function $s \mapsto \frac{D^+A(s)}{D^+A(s)s - A(s)}$ is non-increasing on $(0,R)$.
\end{proof}
We prove that the functions $\fg^t$ defined as in eq. \eqref{eq:fct_g_t} are left-continuous in $1$.
\begin{Lemma}\label{lem:left_cont_g^t}
    Let $\psi \in \mathbf{\Psi}$, $A \in \mathcal{A}$ and $t \in [0,1)$. Then the function $\fg^t$ defined according to eq. \eqref{eq:fct_g_t} is left-continuous in $1$.
\end{Lemma}
\begin{proof}
Let $t \in (0,1)$ and $\varepsilon > 0$ with $t + \varepsilon \in (0,1)$ be arbitrary. Then we choose $\delta := 1- \psi(\varphi(t) - \varphi(t+\varepsilon))$ and fix $x \in (t,1)$ with $1-x <\delta$ arbitrarily, which implies that $\varphi(t + \varepsilon) + \varphi(x) < \varphi(t)$. Using this together with the fact that $A(x) \leq 1$ for every $x \in \mathbb{I}$ yields
$$
h_A\left(\frac{\varphi(x)}{\varphi(t+\varepsilon) + \varphi(x)}\right) \leq \frac{\varphi(t + \varepsilon) + \varphi(x)}{\varphi(x)} < \frac{\varphi(t)}{\varphi(x)}.
$$
Utilizing the pseudo-inverse $h_A^{[-1]}$ then implies $\frac{1}{h_A^{[-1]}\left(\frac{\varphi(t)}{\varphi(x)}\right)} > \frac{\varphi(t + \varepsilon)}{\varphi(x)} + 1$ and therefore
$$
\left(\frac{1}{h_A^{[-1]}\left(\frac{\varphi(t)}{\varphi(x)}\right)}-1\right)\varphi(x) > \varphi(t+\varepsilon),
$$
which yields
$$
|\fg^t(1) - \fg^t(x)| = \fg^t(x) - t < \varepsilon,
$$
proving that $\fg^t$ is left-continuous in $1$. For $\fg^0$, given that $\fg^0(x) = 0$ for every $x \in \mathbb{I}$ when $\psi$ is strict, left continuity of $\fg^0$ is obvious. If $\psi$ is non-strict, left-continuity of $\fg^0$ follows applying the same reasoning used in the case $t \in (0,1)$.
\end{proof}
The following auxiliary lemma is used to prove Theorem \ref{thm:markov_kernel}.
\begin{Lemma}\label{lem:auxiliary_markov_kernel}
Let $\psi \in \mathbf{\Psi}$, $\varphi$ be its pseudo-inverse, $A \in \mathcal{A}$ and $L$ be defined according to eq. \eqref{eq:definition_L_R}. Then for fixed $y \in (0,1)$ the function
$$
s \mapsto (\varphi(s) + \varphi(y)) \cdot A\left(\frac{\varphi(s)}{\varphi(s) + \varphi(y)}\right)
$$
is strictly decreasing on $\left(0,\psi\left(\frac{L\varphi(y)}{1-L}\right)\right)$.
\end{Lemma}
\begin{proof}
    Fixing $y \in (0,1)$, using the fact that the function $s\mapsto \frac{\varphi(s)}{\varphi(s) + \varphi(y)}$ is strictly decreasing and applying Lemma \ref{lem:properties_left_hand_der}, calculating the left-hand derivative w.r.t. the variable $s$ yields that 
    \begin{align*}
    \partial_s^-&\left[(\varphi(s) + \varphi(y))A\left(\frac{\varphi(s)}{\varphi(s) + \varphi(y)}\right)\right] \\&= D^-\varphi(s)\left[A\left(\frac{\varphi(s)}{\varphi(s) + \varphi(y)}\right) + D^+A\left(\frac{\varphi(s)}{\varphi(s) + \varphi(y)}\right)\frac{\varphi(y)}{\varphi(s) + \varphi(y)}\right] \\&=
    D^-\varphi(s) G_A\left(\frac{\varphi(s)}{\varphi(s) + \varphi(y)}\right),
     \end{align*}
     for every $s \in (0,1)$, whereby $G_A$ is defined according to eq. \eqref{eq:fct_f}. If $s \in \left(0,\psi\left(\frac{L\varphi(y)}{1-L}\right)\right)$, then $L < \frac{\varphi(s)}{\varphi(s) + \varphi(y)}$ and therefore applying Lemma \ref{lem:regularity_fcts} and the fact that $D^-\varphi$ is negative yields that $D^-\varphi(s) \cdot G_A\left(\frac{\varphi(s)}{\varphi(s) + \varphi(y)}\right) <0$ for every $s \in \left(0,\psi\left(\frac{L\varphi(y)}{1-L}\right)\right)$. Together with the fact that the function $s \mapsto (\varphi(s) + \varphi(y))A\left(\frac{\varphi(s)}{\varphi(s) + \varphi(y)}\right)$ is continuous, this implies that  $s \mapsto (\varphi(s) + \varphi(y))A\left(\frac{\varphi(s)}{\varphi(s) + \varphi(y)}\right)$ is strictly decreasing on $\left(0,\psi\left(\frac{L\varphi(y)}{1-L}\right)\right)$.
\end{proof}
Proving Lemma \ref{lem:support_sing}, the following technical lemma will be used.
\begin{Lemma}\label{lem:help_singular_comp_full_support}
    Let $[a,b] \subseteq \mathbb{I}$ with $a<b$ be an interval and $F\colon [a,b] \rightarrow \mathbb{I}$ be a strictly increasing, measure generating function with corresponding singular measure $m_F$ fulfilling $m_F([a,b])\leq 1$, and let $P \colon [a,b] \rightarrow \mathbb{I}$ be a positive measure generating function with corresponding measure $m_P$ fulfilling $m_P([a,b])\leq 1$. Furthermore, let $m_Q$ be the measure induced by the measure generating function $Q \colon [a,b] \rightarrow \mathbb{I}$ defined by $Q := F\cdot P$. Then the singular component $m_Q^{sing}$ of $m_Q$ fulfills that
    $$
    \mathrm{supp}(m_Q^{sing}) = [a,b].
    $$
\end{Lemma}
\begin{proof}
    We prove the following three assertions:
    \begin{itemize}
        \item[$(i)$] If $m_P$ is absolutely continuous, then $\mathrm{supp}(m_Q^{sing}) = [a,b]$.
        \item[$(ii)$] If $m_P$ is discrete, then $\mathrm{supp}(m_Q^{sing}) = [a,b]$.
        \item[$(iii)$] If $m_P$ is singular, then $\mathrm{supp}(m_Q^{sing}) = [a,b]$.
    \end{itemize}
    Proving the first assertion and assuming that $m_P$ is absolutely continuous, we denote the Radon-Nikodym density of $m_P$ by $p$ and therefore obtain that
    $$
    Q'(x) = F'(x)P(x) + F(x)P'(x) = F(x)p(x)
    $$
    for $\lambda$-a.e. $x \in [a,b]$. We fix an arbitrary interval $(x_1,x_2) \subseteq [a,b]$ with $x_1 < x_2$ and show that $m_Q^{sing}((x_1,x_2)) > 0$. If $P$ is constant on $(x_1,x_2)$, then obviously $m_Q^{sing}((x_1,x_2)) > 0$ holds. We assume that $P(x_1) < P(x_2)$. Suppose that $m_Q^{sing}((x_1,x_2)) = 0$, then $m_Q$ would be absolutely continuous on $(x_1,x_2)$ with density $F\cdot p$. Applying the fact that $F$ is strictly increasing yields that
    \begin{align*}
        m_Q((x_1,x_2)) &= \int_{(x_1,x_2)} F(s)\cdot p(s) \mathrm{d}\lambda(s) \\& <
        F(x_2)[P(x_2) - P(x_1)] \\& <
        F(x_2)P(x_2) - F(x_1)P(x_1) \\&=
        m_Q((x_1,x_2)).
    \end{align*}
    A contradiction. Therefore, $m_Q^{sing}((x_1,x_2)) > 0$ and thus, since the interval $(x_1,x_2)$ was arbitrary, $\mathrm{supp}(m_Q^{sing}) = [a,b]$ holds.\\
    We prove the second assertion and assume that $m_P$ is discrete, i.e., that there exist $\alpha_1, \alpha_2, ... \in \mathbb{I}$ with $\sum_{i \in \mathbb{N}}\alpha_i \leq 1$ and $q_1, q_2, ... \in [a,b]$, such that $m_P = \sum_{i \in \mathbb{N}}\alpha_i\delta_{q_i}$ holds. Then $Q$ can be represented as
    $$
    Q(x) = F(x) \cdot P(x) =\sum_{i\in\mathbb{N}}\alpha_iF(x)\delta_{q_i}([0,x])
    $$
    for every $x \in [a,b]$. Since $Q$ is continuous for every $x \not\in \{q_i \colon i \in \mathbb{N}\}$, the discrete component of $m_Q$ is given by $m_Q^{dis} := \sum_{i\in\mathbb{N}}\alpha_i F(q_i)\delta_{q_i}$. We denote by $Q^{dis}$ the measure generating function corresponding to $m_Q^{dis}$, define the function $Q^{sing}(x) := Q(x) - Q^{dis}(x)$ for every $x \in [a,b]$ and prove that $Q^{sing}$ is singular. Obviously, singularity of $F$, discreteness of $P$ and $Q^{dis}$ yield that
    $$
    (Q^{sing})'(x) = F'(x)P(x) + F(x)P'(x) - (Q^{dis})'(x) = 0
    $$
    for $\lambda$-a.e. $x \in [a,b]$. Moreover, having that
    \begin{align*}
     Q^{sing}(x-) &= \sum_{\substack{i \in \mathbb{N},\\q_i < x}} \alpha_i[F(x-) - F(q_i)] \\&= \sum_{\substack{i \in \mathbb{N},\\q_i < x}} \alpha_i[F(x) - F(q_i)] \\&= \sum_{\substack{i \in \mathbb{N},\\q_i \leq x}} \alpha_i[F(x) - F(q_i)] = Q^{sing}(x), 
    \end{align*}
    for every $x \in (a,b]$, we obtain continuity of $Q^{sing}$. Using that $F$ is strictly increasing and $P$ is positive, the fact that $Q^{sing}$ is strictly increasing follows immediately and therefore $\mathrm{supp}(m_Q^{sing}) = [a,b]$ holds.\\
Applying the fact that products of singular functions are singular and utilizing strict increasingness of $F$ and positivity of $P$ yields the third assertion.\\
In the case that $m_P$ is arbitrary, at least one of the three components $m_P^{abs}$, $m_P^{dis}$ or $m_P^{sing}$ is non-degenerated and thus the desired result follows from one of the assertions $(i)$, $(ii)$ or $(iii)$.
\end{proof}
  \subsection{Proofs of Sections \ref{section:notation}, \ref{sec:lvl_sets_archimax} and \ref{section:regularity_results}\label{sec:appendix_proofs}}
  \subsubsection*{Proof of Lemma \ref{lem:regularity_fcts}:\label{proof:regularty_fct}}
  \begin{proof}
    Proving the first assertion, the fact that $h_A$ is strictly decreasing on $(0,R]$ and that $h_A(s) = 1$ for every $s \in [R,1]$ has already been established in \citep[Lemma 5]{evc-mass}. Considering convexity of $h_A$, we calculate the right-hand derivative $D^+h_A$ of $h_A$ and obtain that
    $$
    D^+h_A(t) = \frac{D^+A(t)t -A(t)}{t^2}
    $$
    for every $t \in (0,1)$. We fix $t_1,t_2 \in (0,1)$ with $t_1 < t_2$ arbitrarily and show that $D^+h_A(t_1) \leq D^+h_A(t_2)$
    holds. Since the afore-mentioned inequality is equivalent to
    \begin{equation*}\label{eq:euival_convex}
    \frac{t_1^2[D^+A(t_2)t_2-A(t_2)]-t_2^2[D^+A(t_1)t_1-A(t_1)]}{t_1^2t_2^2} \geq 0,
    \end{equation*}
    it suffices to prove that $t_1^2[D^+A(t_2)t_2-A(t_2)]-t_2^2[D^+A(t_1)t_1+A(t_1)] \geq 0$. Using convexity of $A$ and applying the fact that $D^+A(t_1) \leq D^+A(t_2)$ yields that
    \begin{align*}
    D^+A(t_2)t_2 -A(t_2) - &[D^+A(t_1)t_1 - A(t_1)] \\&\geq D^+A(t_2)t_2 -A(t_2) - D^+A(t_1)t_1 + D^+A(t_2)(t_1-t_2) + A(t_2) \\&=
    t_1(D^+A(t_2) - D^+A(t_1)) \geq 0.
     \end{align*}
     Having that $h_A$ is non-increasing, we conclude that $D^+h_A(s) \leq 0$ for every $s \in (0,1)$ and therefore $D^+A(s)s - A(s) \leq 0$ holds for every $s \in (0,1)$.
     Using this fact together with the previous inequality and applying that $t_1 < t_2$, implies that
     \begin{align*}
     t_1^2[D^+A(t_2)t_2-A(t_2)]-&t_2^2[D^+A(t_1)t_1-A(t_1)] \\&\geq t_2^2[D^+A(t_2)t_2 -A(t_2) - (D^+A(t_1)t_1 - A(t_1))] \geq 0.
        \end{align*}
  Therefore, $D^+h_A$ is non-decreasing on $(0,1)$, which, together with the fact that $h_A$ is continuous, yields convexity of $h_A$ on $(0,1]$ \citep[Appendix C]{Pollard2001}.\\
   We prove the second assertion. Since $h_A$ is strictly decreasing on $(0,R]$, $h_A^{[-1]}$ is strictly decreasing on $[1,\infty)$, as it is the inverse of a strictly decreasing function. Proving that $h_A^{[-1]}$ is convex, we show that the left-hand derivative $D^-h_A^{[-1]}$ is non-decreasing on $(1,\infty)$. Using that $h_A$ is strictly decreasing and convex on $(0,R)$ and applying Lemma \ref{lem:properties_left_hand_der}, we obtain that
   $$
   D^-h_A^{[-1]}(t) = \frac{1}{D^+h_A\left(h_A^{[-1]}(t)\right)}
   $$
   for every $t \in (1,\infty)$. Note that $h_A^{[-1]}(t) < R$ holds for every $t \in (1,\infty)$ and therefore $D^+h_A\left(h_A^{[-1]}(t)\right) < 0$ for every $t \in (1,\infty)$. Applying the fact that $D^+h_A$ is non-decreasing on $(0,1)$ and that $h_A^{[-1]}$ is strictly decreasing on $(1,\infty)$, we conclude that $D^-h_A^{[-1]}$ is non-decreasing on $(1,\infty)$ and therefore \citep[Appendix C]{Pollard2001} $h_A^{[-1]}$ is convex on $(1,\infty)$. Continuity of $h_A^{[-1]}$ yields convexity of $h_A^{[-1]}$ on $[1,\infty)$.\\
    Considering the third assertion, the fact that $G_A$ is non-negative and non-decreasing has already been established in \citep[Lemma 5]{evc-mass}. Right-continuity of $G_A$ immediately follows from the fact that $D^+A$ right-continuous. We show that $G_A$ is positive on $(L,1]$. Fix $s \in (L,1]$ arbitrarily and suppose that $D^+A(s) = -1$. Using that $D^+A(w) \leq D^+A(s)$, $D^+A(w) \in [-1,1]$ for every $w \in [0,s]$ and applying \citep[Appendix C]{Pollard2001} would imply that the subsequent identity holds:
    $$
    A(s) = 1 + \int_{[0,s]} D^+A(w) \mathrm{d}\lambda(w) = 1-s.
    $$ 
    This contradicts the fact that $L$ is the largest element with the afore-mentioned property and therefore $D^+A(s) > -1$ holds. Overall we conclude that
    $$
    G_A(s) = D^+A(s)(1-s) + A(s) > -(1-s) + (1-s) = 0
    $$
    and hence $G_A$ is positive on $(L,1]$.
    Let $L>0$. Then having that $A(t) = 1-t$ and $D^+A(t) = -1$ for every $t \in [0,L)$ yields that $G_A(t) = 0$ for every $t \in [0,L)$. Applying that $A(t) = t$ and $D^+A(t) = 1$ for every $t \in [R,1]$ finally implies that $G_A(t) = 1$ for every $t \in [R,1]$.
\end{proof}
\subsubsection*{Proof of Lemma \ref{lem:0-lvl_sets}:\label{proof:0-lvl_sets}}
\begin{proof}
    Let $\psi \in \mathbf{\Psi}$, $\varphi$ be its pseudo-inverse, $A \in \mathcal{A}$ and $h_A$ be defined according to eq. \eqref{eq:fct_h}. If $(x,y) \in \{1\} \times (0,1]$ or $(x,y) \in (0,1] \times \{1\}$, then obviously $(x,y) \notin L_{\psi,A}^0$ and since $\fg^0(1) = 0$ as well as $\fg^0(x) <1$, also $(x,y) \notin \{(u,v) \in \mathbb{I} \colon v \leq f^0(x)\}\cup (\{0\}\times [\ff^R(0),1])$ holds. Assuming that $(x,y) \in \{0\}\times \mathbb{I}$ or $(x,y) \in \mathbb{I}\times \{0\}$, applying that $\ff^R(0) \leq y$ or $y < \ff^R(0) = \fg^0(0)$ yields $(x,y) \in L_{\psi,A}^0$ if, and only if $(x,y) \in \{(u,v) \in \mathbb{I} \colon v \leq f^0(x)\} \cup (\{0\}\times [\ff^R(0),1])$. Now, we assume that $(x,y) \in (0,1)^2$ and obtain that $(x,y) \in L_{\psi,A}^0$ is equivalent to
    \begin{equation}\label{eq:equiv_lvl_set}
        h_A\left(\frac{\varphi(x)}{\varphi(x) + \varphi(y)}\right)\varphi(x) \geq \varphi(0).
    \end{equation}
    We consider two cases: 
    \begin{itemize}
        \item[$(a)$]  If $R \leq \frac{\varphi(x)}{\varphi(x) + \varphi(y)}$, then $h_A(\frac{\varphi(x)}{\varphi(x) + \varphi(y)}) = 1$, implying that eq. \eqref{eq:equiv_lvl_set} holds if, and only if $x = 0$. Together with the fact that $R \leq \frac{\varphi(x)}{\varphi(x) + \varphi(y)}$ is equivalent to $\ff^R(x) \leq y$, eq. \eqref{eq:equiv_lvl_set} therefore holds if, and only if $(x,y) \in \{0\} \times [\ff^R(0),1]$.
        \item[$(b)$]If $\frac{\varphi(x)}{\varphi(x) + \varphi(y)} < R$, then following a similar approach to the proof of Lemma \ref{lem:t-lvl_sets}, eq. \eqref{eq:equiv_lvl_set} is equivalent to $y \leq \fg^0(x)$.
          \end{itemize}
   If $(x,y) \in L_{\psi,A}^0$, then combining the previous arguments directly yields $(x,y) \in \{(u,v) \in \mathbb{I}^2\colon v \leq \fg^0(u)\} \cup (\{0\}\times [\ff^R(0),1])$. If $(x,y) \in \{0\}\times [\ff^R(0),1]$, we obviously have $(x,y) \in L_{\psi,A}^0$. Finally, considering $(x,y) \in \{(u,v) \in \mathbb{I}^2 \colon v \leq \fg^0(u)\}$, we obtain that $y \leq \fg^0(x) \leq \ff^R(x)$ and therefore $\frac{\varphi(x)}{\varphi(x) + \varphi(y)} \leq R$. If $\frac{\varphi(x)}{\varphi(x) + \varphi(y)} < R$, the fact that $(x,y) \in L_{\psi,A}^0$ follows form $(b)$. If $\frac{\varphi(x)}{\varphi(x) + \varphi(y)} = R$, then $y = \ff^R(x) \leq \fg^0(x)$ and therefore $x = 0$, implying $(x,y) \in \{0\} \times [\ff^R(0),1]$. The fact that $(x,y) \in L_{\psi,A}^0$ is then an immediate consequence of $(a)$.\\
    The second assertion follows applying the same chain of arguments as in the proof of Lemma \ref{lem:t-lvl_sets}.
\end{proof}
\subsubsection*{Proof of Lemma \ref{lem:auxiliary_lvl_curves_ii}:\label{proof:auxiliary_lvl_curves_ii}}
\begin{proof}
    We fix $t \in (0,1)$ in the case that $\psi$ is strict and $t \in [0,1)$ in the case that $\psi$ is non-strict.\\
    If $0<L<R$ and $x \in \left[\psi\left(\frac{L\varphi(t)}{1-L}\right),1\right)$, then obviously $h_A(L) = \frac{1-L}{L} \leq \frac{\varphi(t)}{\varphi(x)}$ holds. Using the fact that $h_A^{[-1]}\left(z\right) = \frac{1}{1 + z}$ holds for every $z \in \left[\frac{1-L}{L},\infty\right)$ implies that $h_A^{[-1]}\left(\frac{\varphi(t)}{\varphi(x)}\right) = \frac{1}{1 + \frac{\varphi(t)}{\varphi(x)}}$ and therefore $\xi^t(x) = \varphi(t)$ follows. This proves the second part of the second assertion.\\
    Proving that $\xi^t$ is non-decreasing, we calculate the left-hand derivative $D^-\xi^t$ of $\xi^t$ and show that it is non-negative. Throughout this prove we frequently make use of the fact that the functions $\varphi$, $h_A$ and $h_A^{[-1]}$ are convex and thus, the involved left-hand derivatives exist.
    Applying the product/quotient rule for left-hand derivatives yields that
       \begin{align*}
    D^-\xi^t(x) &=
    \partial_x^-\left[\frac{1}{h_A^{[-1]}\left(\frac{\varphi(t)}{\varphi(x)}\right)} - 1\right]\varphi(x) + \left[\frac{1}{h_A^{[-1]}\left(\frac{\varphi(t)}{\varphi(x)}\right)} - 1\right]D^-\varphi(x) \\& =
    \frac{-\partial_x^{-}\left[h_A^{[-1]}\left(\frac{\varphi(t)}{\varphi(x)}\right)\right]\varphi(x) + \left(1-h_A^{[-1]}\left(\frac{\varphi(t)}{\varphi(x)}\right)\right)h_A^{[-1]}\left(\frac{\varphi(t)}{\varphi(x)}\right)D^-\varphi(x)}{h_A^{[-1]}\left(\frac{\varphi(t)}{\varphi(x)}\right)^2}
    \end{align*}
    for every $x\in (t,1)$.
    Using the fact that $x \mapsto \frac{\varphi(t)}{\varphi(x)}$ is strictly increasing, that $h_A$ is strictly decreasing on $(0,R)$ and applying Lemma \ref{lem:properties_left_hand_der}, we conclude that
    $$
    \partial_x^{-}\left[h_A^{[-1]}\left(\frac{\varphi(t)}{\varphi(x)}\right)\right] = \frac{-\varphi(t) D^-\varphi(x)}{D^+h_A\left(h_A^{[-1]}\left(\frac{\varphi(t)}{\varphi(x)}\right)\right)\varphi(x)^2}
    $$
    for every $x\in (t,1)$. Combining the two most recent calculations, we find that
    \begin{align*}
    D^-\xi^t(x) =
    \frac{D^-\varphi(x)}{h_A^{[-1]}\left(\frac{\varphi(t)}{\varphi(x)}\right)^2}\frac{\frac{\varphi(t)}{\varphi(x)}+ \left(1-h_A^{[-1]}\left(\frac{\varphi(t)}{\varphi(x)}\right)\right)h_A^{[-1]}\left(\frac{\varphi(t)}{\varphi(x)}\right)D^+h_A\left(h_A^{[-1]}\left(\frac{\varphi(t)}{\varphi(x)}\right)\right)}{D^+h_A\left(h_A^{[-1]}\left(\frac{\varphi(t)}{\varphi(x)}\right)\right)}
    \end{align*}
    for every $x \in (t,1)$. Note that $h_A^{[-1]}\left(\frac{\varphi(t)}{\varphi(x)}\right) \leq R$ together with the fact that $h_A$ is strictly decreasing and convex on $(0,R]$ implies $D^+h_A\left(h_A^{[-1]}\left(\frac{\varphi(t)}{\varphi(x)}\right)\right) < 0$.
    Using both $h_A(s) = \frac{A(s)}{s}$ and $D^+h_A(s) = \frac{D^+A(s)s -A(s)}{s^2}$ for every $s \in (0,1)$, it follows that
    $$
    h_A(s) + (1-s)sD^+h_A(s) = D^+A(s)(1-s) + A(s) = G_A(s)
    $$
    for every $s \in (0,1)$, whereby $G_A$ is defined as in eq. \eqref{eq:funct_f_a}.
    Utilizing the preceding equation and applying the identity  $h_A\left(h_A^{[-1]}\left(\frac{\varphi(t)}{\varphi(x)}\right)\right) = \frac{\varphi(t)}{\varphi(x)}$, we obtain that the numerator of the second factor of 
 $D^-\xi^t$ simplifies to
    \begin{align*}
 h_A\left(h_A^{[-1]}\left(\frac{\varphi(t)}{\varphi(x)}\right)\right) &+ \left(1-h_A^{[-1]}\left(\frac{\varphi(t)}{\varphi(x)}\right)\right)h_A^{[-1]}\left(\frac{\varphi(t)}{\varphi(x)}\right)D^+h_A\left(h_A^{[-1]}\left(\frac{\varphi(t)}{\varphi(x)}\right)\right)\\&= G_A\left(h_A^{[-1]}\left(\frac{\varphi(t)}{\varphi(x)}\right)\right)
    \end{align*}
for every $x \in (t,1)$. Combining all of the above, we observe that
    \begin{align*}
    D^-\xi^t(x) &= \frac{D^-\varphi(x)\cdot G_A\left(h_A^{[-1]}\left(\frac{\varphi(t)}{\varphi(x)}\right)\right)}{h_A^{[-1]}\left(\frac{\varphi(t)}{\varphi(x)}\right)^2\cdot D^+h_A\left(h_A^{[-1]}\left(\frac{\varphi(t)}{\varphi(x)}\right)\right)}
     \end{align*}
     for every $x \in (t,1)$. Using the properties given in Lemma \ref{lem:regularity_fcts} and the fact that $\varphi$ is strictly decreasing and convex, it follows that $D^-\xi^t(x) \geq 0$. Therefore, applying the afore-mentioned continuity of $\xi^t$, the function $\xi^t$ is non-decreasing. This proves the first assertion.\\
     If $0<L<R$, then applying Lemma \ref{lem:regularity_fcts} again, we find that $G_A$ is positive on $(L,1]$. Consequently, $D^-\xi^t(x) > 0$ holds for every $x \in  \left(t,\psi\left(\frac{L\varphi(t)}{1-L}\right)\right)$, which combined with the continuity of $\xi^t$ shows that $\xi^t$ is strictly increasing on $\left(t,\psi\left(\frac{L\varphi(t)}{1-L}\right)\right)$, proving the first part of the second assertion.\\
    We prove that $D^-\xi^t$ is non-increasing on $(t,1)$. Applying that $D^+h_A(s) = \frac{D^+A(s)s -A(s)}{s^2}$ and $G_A(s) = D^+A(s)(1-s) + A(s)$ holds for every $s \in (0,1)$, we obtain that
       \begin{align}\label{eq:der_xi}
    \nonumber D^-\xi^t(x) &= \frac{D^-\varphi(x)\cdot G_A\left(h_A^{[-1]}\left(\frac{\varphi(t)}{\varphi(x)}\right)\right)}{D^+A\left(h_A^{[-1]}\left(\frac{\varphi(t)}{\varphi(x)}\right)\right)h_A^{[-1]}\left(\frac{\varphi(t)}{\varphi(x)}\right)-A\left(h_A^{[-1]}\left(\frac{\varphi(t)}{\varphi(x)}\right)\right)} \\&=
    D^-\varphi(x) \cdot \left[-1 + \frac{D^+A\left(h_A^{[-1]}\left(\frac{\varphi(t)}{\varphi(x)}\right)\right)}{D^+A\left(h_A^{[-1]}\left(\frac{\varphi(t)}{\varphi(x)}\right)\right)h_A^{[-1]}\left(\frac{\varphi(t)}{\varphi(x)}\right)-A\left(h_A^{[-1]}\left(\frac{\varphi(t)}{\varphi(x)}\right)\right)}\right]
    \end{align}
     for every $x \in (t,1)$. According to Lemma \ref{lem:auxiliary_conv_level_set}, the function $s \mapsto\frac{D^+A(s)}{D^+A(s)s - A(s)}$ is non-increasing on $(0,R)$, which, together with the fact that $x \mapsto h_A^{[-1]}\left(\frac{\varphi(t)}{\varphi(x)}\right)$ is non-increasing, implies that the second factor in eq. \eqref{eq:der_xi} is non-positive and non-decreasing. Using that $D^-\varphi$ is non-decreasing and negative on $(0,1)$, we finally conclude that $D^-\xi^t$ is non-increasing on $(t,1)$. This proves the third assertion.
\end{proof}
\subsubsection*{Proof of Lemma \ref{lem:help_regularity_H_psi_A}:\label{proof:help_regularity_H_psi_A}}
\begin{proof}
Throughout this proof we let $x \in (0,1)$ be arbitrary. We prove the first assertion and assume that $\gamma \in \mathcal{P}_\mathcal{W}^{abs}$. It is straightforward to see that $z \mapsto \int_{[0,z]} t \mathrm{d}\gamma(t)$ is absolutely continuous on every compact interval $[0,a]$. Using convexity of $\varphi$ and $h_A$ yields that the function $y \mapsto h_A\left(\frac{\varphi(x)}{\varphi(x) + \varphi(y)}\right)$ is locally Lipschitz continuous on $(0,1]$ and therefore absolutely continuous on every compact interval $[b,1]$, which in turn implies (both, in the case that $\varphi$ is strict and in the case that $\varphi$ is non-strict) that the function defined by 
$$
\iota(y) := \frac{1}{\varphi(x) h_A\left(\frac{\varphi(x)}{\varphi(x) + \varphi(y)}\right)}
$$ 
for every $y \in \mathbb{I}$ is absolutely continuous on $\mathbb{I}$. Furthermore, utilizing that the function $\iota$ is non-decreasing and applying \citep[Proposition 129]{pap2002}, we conclude that $H_x$ is absolutely continuous. Consequently, the corresponding probability measure $\nu_{H_x}$ is also absolutely continuous.\\
Prior to proofing the remaining two assertions we notice that $y \in [0,\fg^0(x))$ implies that $\frac{1}{\varphi(x)h_A\left(\frac{\varphi(x)}{\varphi(x) + \varphi(y)}\right)} < \frac{1}{\varphi(0)}$ and therefore $H_x(y) = 0$ holds for every $y \in [0,\fg^0(x))$. Moreover, assuming $y \in [\ff^R(x),1]$ yields that $R \leq \frac{\varphi(x)}{\varphi(x) + \varphi(y)}$ and hence $H_x(y) = 1$ holds for every $y \in [\ff^R(x),1]$.\\
We prove the second assertion and assume that $\gamma \in \mathcal{P}_\mathcal{W}^{dis}$. To be more precise, we assume that there exist $\alpha_1,\alpha_2,... \in \mathbb{I}$ with $\sum_{i \in \mathbb{N}} \alpha_i = 1$ and $s_1,s_2,...\in (0,\infty)$ such that $\gamma = \sum_{i \in \mathbb{N}}\alpha_i \delta_{s_i}$. Note, that in the case that $\gamma$ is non-strict, $s_i \geq \frac{1}{\varphi(0)}$ holds for every $i \in \mathbb{N}$.
Using the fact that $s_i \in I_y$ is equivalent to $y \geq \fg^{\psi(\frac{1}{s_i})}(x)$ for every $y \in [\fg^0(x),\ff^R(x))$, it therefore follows that
$$
H_x(y) = \frac{\int_{I_y}t \mathrm{d}\gamma(t)}{\int_{I_1}t \mathrm{d}\gamma(t)} = \sum_{i\colon s_i \in I_1}\beta_i\mathbf{1}_{[\fg^{\psi(\frac{1}{s_i})}(x),1]}(y),
$$
for every $y \in [\fg^0(x),\ff^R(x))$, whereby $\beta_i := \frac{\alpha_is_i}{\int_{I_1}t \mathrm{d}\gamma(t)}$. Working with the afore-mentioned representation of $H_x$ and using that $H_x(y) \in \{0,1\}$ for every $y \in [0,\fg^0(x)) \cup [\ff^R(x),1]$, it is now straightforward to see that $H_x$ induces a discrete probability measure with mass $\beta_i$ in $\fg^{\psi(\frac{1}{s_i})}(x)$.\\
Proving the third assertion, we consider $\gamma \in \mathcal{P}_\mathcal{W}^{sing}$. Showing that $\nu_{H_x}$ is singular, it suffices to prove that $H_x$ is continuous and $H_\psi'(y) = 0$ holds for $\lambda$-a.e. $y \in \mathbb{I}$. Since $\gamma$ has no point masses, continuity of $H_x$ follows immediately.\\
We show that $H_x'(y) = 0$ holds for $\lambda$-a.e. $y \in \mathbb{I}$ and proceed as follows: We define the $\sigma$-finite measure $m$ by
$$
m(E) := \int_E t \mathrm{d}\gamma(t)
$$
for every $E \in \mathcal{B}([0,\infty))$ and observe that due to the singularity of $\gamma$ there exists a set $O \in \mathcal{B}([0,\infty))$ such that $m([0,\infty)\setminus O) = 0$ and $\lambda(O) = 0$. We denote by $Q_m\colon I_{0,1} \rightarrow [0,\infty)$ the corresponding measure generating function restricted to $I_{0,1}$ defined by $Q_m(z) := m([0,z])$ for every $z \in I_{0,1}$, whereby $I_{0,1} := \left[\frac{1}{\varphi(x)h_A(\frac{\varphi(x)}{\varphi(x) + \varphi(0)})}, \frac{1}{\varphi(x)}\right]$. Applying that $m$ is singular yields that $Q_m'(z) = 0$ for $\lambda$-a.e. $z \in I_{0,1}$. Defining the set
$$
\Lambda := \{z \in I_{0,1} \colon Q_m'(z) \text{ exists and } Q_m'(z) = 0\},
$$
we therefore have $\lambda(\Lambda) = \lambda(I_{0,1})$. In the case that $y \in (\ff^R(x),1)$, $H_x(y) = 1$ holds and hence $H_x'(y) = 0$ follows. Considering the case that $y \in [0,\ff^R(x)]$ we define the set 
$$
\Upsilon := \{y \in [0,\ff^R(x)] \colon \iota(y) \in \Lambda\}
$$
and let $\eta \colon I_{0,1} \rightarrow [0,\ff^R(x)]$ denote the inverse of $\iota\big|_{[0,\ff^R(x)]}$ given by $\eta(z) = \fg^{\psi(\frac{1}{z})}(x)$ for every $z \in I_{0,1}$. Then obviously $\eta(\Lambda) = \Upsilon$ holds. Since $\psi$ and $h_A^{[-1]}$ are convex (see Lemma \ref{lem:regularity_fcts}), $\eta$ is locally Lipschitz continuous and therefore locally absolutely continuous, implying that \citep[Theorem 3.41]{leoni2024} $\eta$ maps sets of full $\lambda$-measure to sets of full $\lambda$-measure. Thus, we obtain that $\lambda(\Upsilon) = \lambda([0,\ff^R(x)])$. Combining all of the afore-mentioned facts, we conclude that
\begin{align*}
H_x'(y) &= \frac{Q_m'(\iota(y))\iota'(y)}{Q_m(\iota(1))} \\&=
\frac{Q_m'(\iota(y))}{Q_m(\iota(1))}\frac{h_A'\left(\frac{\varphi(x)}{\varphi(x) + \varphi(y)}\right)\varphi'(y)}{h_A\left(\frac{\varphi(x)}{\varphi(x) + \varphi(y)}\right)^2(\varphi(x) + \varphi(y))^2} = 0
\end{align*}
for every $y \in \Upsilon \cup (\ff^R(x),1)$. Since $\lambda(\Upsilon \cup (\ff^R(x),1)) = 1$, this proves the result.
\end{proof}
\subsubsection*{Proof of Lemma \ref{lem_help_regularity}:\label{proof:reg_G_A}}
\begin{proof}
We prove the first assertion and assume that $\vartheta \in \mathcal{P}_\mathcal{A}^{abs}$. Applying eq. \eqref{eq:right_der_pickands_measure} yields that $D^+A(t) = 2\vartheta([0,t]) - 1$ holds for every $t \in [0,1)$ and thus, using the fact that $D^+A(1) = D^+A(1-)$ obviously implies that $D^+A$ is absolutely continuous on $\mathbb{I}$. Since $A$ is Lipschitz continuous, absolute continuity of $G_A$ follows using that sums of absolutely continuous functions are absolutely continuous.\\
We prove the second assertion and assume that $\vartheta$ is discrete with $\vartheta(\{0\}) = 0 = \vartheta(\{1\})$ and therefore assume that there exist $\alpha_1, \alpha_1, ... \in \mathbb{I}$ with $\sum_{i \in \mathbb{N}}\alpha_i = 1$ and $s_1,s_2,... \in (0,1)$ such that $\vartheta = \sum_{i \in \mathbb{N}}\alpha_i\delta_{s_i}$. Applying eq. \eqref{eq:rel_pickands_fct_measure} yields that
\begin{equation}\label{eq:A_dis}
A(t) = 1-t + 2t\sum_{\substack{i \in \mathbb{N} \\ s_i \leq t}}\alpha_i - 2\sum_{\substack{i \in \mathbb{N} \\ s_i \leq t}}\alpha_is_i.
\end{equation}
Furthermore, using eq. \eqref{eq:right_der_pickands_measure} implies that
\begin{equation}\label{eq:der_A_dis}
D^+A(t) = 2\sum_{\substack{i \in \mathbb{N} \\ s_i \leq t}}\alpha_i - 1
\end{equation}
for every $t \in \mathbb{I}$. Combining eq. \eqref{eq:A_dis} and eq. \eqref{eq:der_A_dis} we now conclude that
\begin{align*}
G_A(t) &= 1-t + 2t\sum_{\substack{i \in \mathbb{N} \\ s_i \leq t}}\alpha_i - 2\sum_{\substack{i \in \mathbb{N} \\ s_i \leq t}}\alpha_is_i +
\left(2\sum_{\substack{i \in \mathbb{N} \\ s_i \leq t}}\alpha_i - 1\right)(1-t)\\&=
2\left(\sum_{\substack{i \in \mathbb{N} \\ s_i \leq t}}\alpha_i - \sum_{\substack{i \in \mathbb{N} \\ s_i \leq t}}\alpha_is_i\right) \\&=
2 \sum_{i \in \mathbb{N}}\alpha_i(1-s_i)\delta_{s_i}([0,t])
\end{align*}
for every $t \in \mathbb{I}$ and thus, $G_A$ induces a discrete probability measure $\kappa_{G_A}$.\\
Proving the third assertion, we assume that $\vartheta\in \mathcal{P}_\mathcal{A}^{sing}$. Observing that $\vartheta$ has no point masses and applying eq. \eqref{eq:right_der_pickands_measure} yields that $D^+A$ is continuous and therefore continuity of $G_A$ follows. We prove that $G_A'(t)$ exists and $G_A'(t) = 0$ for $\lambda$-a.e. $t \in \mathbb{I}$. Applying convexity of $A$ yields that $A'(t)$ exists for all but at most countably many $t \in \mathbb{I}$, as well as $D^+A = A'$ wherever $A'$ exists. Now, using singularity of $\vartheta$ together with eq. \eqref{eq:right_der_pickands_measure} implies that $DD^+A(t) = 0$ for $\lambda$-a.e. $t \in \mathbb{I}$. Altogether, we therefore conclude that
\begin{align*}
    G_A'(t) = D^+A(t) + DD^+A(t)(1-t) -D^+A(t) = DD^+A(t)(1-t) = 0
\end{align*}
for $\lambda$-a.e. $t \in \mathbb{I}$. This shows that the probability measure $\kappa_{G_A}$ is singular and therefore the result is proved.
\end{proof}

\noindent
\textbf{Acknowledgements.} The author gratefully acknowledges the support of Red Bull GmbH within the ongoing Data Science collaboration with the university of Salzburg.
\vspace{0.5cm}
\\
\noindent 
\textbf{Author contributions} 
All the results are obtained by the author.
\vspace{0.5cm}
\\
\textbf{Funding} Financial support from IDA Lab Salzburg is kindly acknowledged.
\vspace{0.5cm}
\\
\textbf{Data Availability} No datasets were analyzed in the current paper.

\section*{Declarations}

\textbf{Competing interests} The author declares that he has no conflicts of interest to this work.

\bibliographystyle{plainnat}

\end{document}